\documentclass[12pt]{amsart}

\usepackage{ifpdf}

\usepackage{amsmath}
\usepackage{amsthm}
\usepackage{amsfonts}
\usepackage{amssymb}

\usepackage[colorlinks,citecolor=blue,linkcolor=blue]{hyperref}

\usepackage{graphicx}
\usepackage{color}

\makeatletter
\def\blfootnote{\xdef\@thefnmark{}\@footnotetext}
\makeatother

\IfFileExists{.git/git.tex}{
\input{.git/git.tex}

}
{

\newcommand{\gitshash}{NA}

\newcommand{\gitauthsdate}{NA}
}

\newcommand{\dimension}{n}

\newcommand{\dist}{\operatorname{dist}}

\newcommand{\interior}{\operatorname{int}}
\newcommand{\trace}{\operatorname{tr}}

\newcommand{\esssup}{\operatorname*{ess\,sup}}
\newcommand{\essinf}{\operatorname*{ess\,inf}}

\newcommand{\argmin}{\operatorname*{arg\,min}}
\newcommand{\argmax}{\operatorname*{arg\,max}}

\newcommand{\divo}{\operatorname{div}}

\newcommand{\dx}{\;dx}

\newcommand{\abs}[1]{\left|#1\right|}
\newcommand{\pth}[1]{\left(#1\right)}
\newcommand{\bra}[1]{\left[#1\right]}
\newcommand{\set}[1]{{\left\{#1\right\}}}

\newcommand{\norm}[1]{\left\|#1\right\|}

\newcommand{\cl}[1]{\overline{#1}}	% closure

\newcommand{\e}{\ensuremath{\varepsilon}}

\newcommand{\R}{\ensuremath{\mathbb{R}}}

\newcommand{\Rd}{\ensuremath{{\mathbb{R}^{\dimension}}}}
\newcommand{\Rn}{\Rd}
\newcommand{\Z}{\ensuremath{\mathbb{Z}}}

\newcommand{\T}{\ensuremath{\mathbb{T}}}
\newcommand{\Tn}{{\T^\dimension}}

\definecolor{grey}{rgb}{0.6,0.6,0.6}

\numberwithin{equation}{section}

\newtheorem{theorem}{Theorem}[section]
\newtheorem{lemma}[theorem]{Lemma}
\newtheorem{proposition}[theorem]{Proposition}
\newtheorem{corollary}[theorem]{Corollary}

\newtheorem{definition}[theorem]{Definition}

\theoremstyle{definition}

\newtheoremstyle{remarkstyle}% name of the style to be used
  {3pt}% measure of space to leave above the theorem. E.g.: 3pt
  {3pt}% measure of space to leave below the theorem. E.g.: 3pt
  {}% name of font to use in the body of the theorem
  {}% measure of space to indent
  {\bfseries}% name of head font
  {.}% punctuation between head and body
  { }% space after theorem head; " " = normal interword space
  {}% Manually specify head

\theoremstyle{remarkstyle}
\newtheorem{remark}[theorem]{Remark}
\newtheorem{example}[theorem]{Example}

\newcommand{\TT}{\mathcal T}
\usepackage{tikz}

\title[Motion by crystalline-like mean curvature: a survey]{Motion by crystalline-like mean curvature: \\ a survey}

\author[Y.\ Giga]{Yoshikazu Giga}
\address{Graduate School of Mathematical Sciences\\
The University of Tokyo\\
3-8-1 Komaba, Meguro-ku\\
Tokyo 153-8914, Japan}
\email{labgiga@ms.u-tokyo.ac.jp}

\author[N.\ Po\v{z}\'ar]{Norbert Po\v{z}\'ar}
\address{Faculty of Mathematics and Physics\\
Institute of Science and Engineering\\
Kanazawa University\\
Kakuma\\
Kanazawa 920-1192, Japan}
\email{npozar@se.kanazawa-u.ac.jp}
\thanks{YG is supported by JSPS through grants No. 19H00639 (Kiban A), No.\ 18H05323 (Kaitaku), No.\ 17H01091.
 NP is supported by JSPS KAKENHI Wakate Grant (No. 18K13440).}

\date{\today\ (git: \gitauthsdate, \gitshash)}



\newcommand{\sign}{\operatorname{sign}}
\newcommand{\E}{\mathcal{E}}
\newcommand{\one}{\mathbf{1}}
\newcommand{\compl}{{\mathsf{c}}}

\begin{document}
\baselineskip=17pt

\begin{abstract}
We consider a class of anisotropic curvature flows called a crystalline curvature flow.
 We present a survey on this class of flows with special emphasis on the well-posedness of its initial value problem.
\end{abstract}

\maketitle

\tableofcontents

\section{Introduction} \label{Int}

The famous mean curvature flow was introduced by W.\ W.\ Mullins \cite{Mu56} to model the motion of an antiphase grain boundary in annealing metals.
 Its governing equation is called the mean curvature equation and it is an equation for one-parameter family of hypersurfaces $\{\Gamma_t\}$ (an evolving hypersurface) in $\mathbb{R}^n$ which imposes that the normal velocity $V$ equals the mean curvature $\kappa$, i.e.,
\[
	V = \kappa \quad \text{on} \quad \Gamma_t;
\]
here, the curvature and the velocity is taken in the direction of the normal vector field $\nu$ of $\Gamma_t$.
 This equation can be interpreted as a steepest descent flow of the surface area.
 In materials science the surface area is considered as an interfacial energy of the grain boundary.
 It is quite natural to consider anisotropic effects.
 For this purpose, one considers the anisotropic interfacial energy
\[
	I(\Gamma) = \int_\Gamma \sigma (\nu)\; d\mathcal{H}^{n-1},
\]
where $\sigma$ is a given positive function called the interfacial energy density;
 here, $d\mathcal{H}^{n-1}$ is the surface area element of a hypersurface $\Gamma$.
 Its first variation is called the anisotropic mean curvature denoted by $\kappa_\sigma$;
 this is often called the weighted mean curvature.
 If one replaces the mean curvature by the anisotropic mean curvature in the mean curvature flow equation, the resulting equation is of the form
\begin{equation} \label{AM1}
	V = \kappa_\sigma \quad\text{on}\quad \Gamma_t.
\end{equation}
 In general, this equation may not be parabolic even if $\sigma$ is smooth.
 We consider the one-homogeneous extension of $\sigma$ in $\mathbb{R}^n$ and still denote it by $\sigma$, i.e.,
\begin{equation} \label{H}
	\sigma(p) = |p| \sigma ( p/|p| ), \quad
	p \in \mathbb{R}^n \setminus \set0.
\end{equation}
If $\sigma$ is convex, the equation \eqref{AM1} is at least degenerate parabolic.
 Although the problem when $\sigma$ is not convex is interesting, we do not touch this problem in this paper.
 The reader is referred to \cite{BGeN} for such an ill-posed problem.

The anisotropic mean curvature flow can be considered as the mean curvature flow in a Minkowski metric or a Finsler metric.
 In this case, $V$ should be replaced by the Minkowski normal velocity.
  If one uses the Euclidean normal velocity, it is of the form
\[
	V = \sigma \kappa_\sigma;
\]
see \cite{BP96} for this perspective.

The curvature flow is not restricted to the form \eqref{AM1}.
 For second-order model, a general form of the flow is
\begin{equation} \label{GE}
	V = g (\nu, \kappa_\sigma)
\end{equation}
with $g$ non-decreasing in the second variable.
 A typical example in themodynamics is
\[
	V = M (\nu) (\kappa_\sigma + C)
\]
with mobility $M(\nu)>0$ and a driving force $C$, where $C$ is a constant \cite{Gu}, \cite{AG}.
 There are several other examples when $g$ is nonlinear in $\kappa_\sigma$.
 For example,
\[
	V = |\kappa_\sigma|^{\alpha-1} \kappa_\sigma
\]
with some positive number $\alpha$.
 We shall discuss these examples in Section \ref{SM}. 

For later convenience, we say that $\sigma: \mathbb{R}^n \to [0,\infty)$ is an \emph{anisotropy} if $\sigma$ is positively one-homogeneous, convex and $\sigma>0$ outside the origin.
 By definition, $\sigma$ satisfies \eqref{H} and the \emph{Frank diagram}
\[
	F_\sigma = \left\{ p \in \mathbb{R}^n \mid
	\sigma(p) \leq 1 \right\}
\]
is bounded, convex and contains the origin as an interior point.

For many applications, especially in low temperature physics, it is often considered the case that $\sigma$ is not $C^1$.
 An extreme case is that the anisotropy $\sigma$ is (purely) \emph{crystalline}, i.e., $\sigma$ is piecewise linear so that $F_\sigma$ is a convex polytope.
 A crystalline mean curvature flow is formally \eqref{GE} when anisotropy $\sigma$ is crystalline.
 In mathematical community, it was introduced by J.\ E.\ Taylor \cite{T1} and independently by S.\ B.\ Angenent and M.\ E.\ Gurtin \cite{AG} around 1990.

One might be curious on the value of $\kappa_\sigma$ when $\sigma$ is crystalline.
 To motivate it we consider an anisotropic isoperimetric problem of the form \\ 
``Find a shape $D$ in $\mathbb{R}^n$ with fixed volume which minimizes the surface energy $I(\Gamma)$ with $\Gamma=\partial D$."\\
This problem was first studied by Wulff \cite{W} and it turns out that the minimizer is the \emph{Wulff shape}
\[
	W_\sigma = \bigcap_{|m|=1} \left\{ x \in \mathbb{R}^n \mid
	x \cdot m \leq \sigma(m) \right\},
\]
which is the polar of $F_\sigma$.
 This has been proved in quite general setting; see e.g.\ \cite{T78}, \cite{FM}.
 For recent progress related to optimal transport theory, see \cite{FiMP}.
 Note that if $\sigma$ is crystalline so that $F_\sigma$ is a polytope, then $W_\sigma$ is also a polytope.
 For smooth anisotropy, one observes that the anisotropic $\kappa_\sigma$ on the surface of $W_\sigma$ is a non-zero constant, and so $W_\sigma$ plays the same role as a ball for the usual curvature.
 More precisely, if one takes $\nu$ inward $\kappa_\sigma = n-1$.
 If $\sigma$ is crystalline, then $W_\sigma$ is a polytope.
 Nevertheless, $\kappa_\sigma$ should not be zero.
 This simple observation shows that the value $\kappa_\sigma$ cannot be determined by infinitesimal quantities like tangent and second fundamental form of the surface.
 We say that \eqref{GE} is a \emph{crystalline (mean) curvature flow (equation)} if $\sigma$ is crystalline.

% 原稿7ページ
We now consider a simple example of a crystalline curvature flow for a graph-like curve.
 For later convenience, we write the equation \eqref{AM1} when $\Gamma_t$ is given as the graph of a function $w=w(x',t)$, i.e., $x_n=w(x',t)$ for $x=(x',x_n)\in\mathbb{R}^n$, $x'\in\mathbb{R}^{n-1}$.
 The upward normal velocity is given as
\[
	V = \frac{w_t}{\left( 1+|\nabla'w|^2 \right)^{1/2}},
\]
where $w_t=\partial w/\partial t$, $\nabla'w=(\partial_{x_1}w,\ldots,\partial_{x_{n-1}}w)$, $\partial_{x_j} = \partial/\partial x_j$, $w_{x_j}=\partial_{x_j}w$.
 The anisotropic mean curvature is formally of the form
\[
	\kappa_\sigma = -\operatorname{div}_{\Gamma_t} \zeta(\nu) \quad\text{with}\quad
	\zeta(\nu) = (\nabla_p \sigma)(\nu),
\]
where $\nabla_p \sigma$ denotes the gradient of $\sigma$, i.e., $\nabla_p\sigma=(\partial_{p_1}\sigma,\ldots,\partial_{p_n}\sigma)$ for anisotropy $\sigma=\sigma(p_1,\ldots,p_n)$.
 The divergence $\operatorname{div}_{\Gamma_t}$ denotes the surface divergence, i.e.,
\[
	\operatorname{div}_{\Gamma_t} X = \operatorname{trace} (I-\nu\otimes\nu) \nabla X;
\]
here, we extend $X$ in a tubular neighborhood of $\Gamma_t$ in a suitable way and $\nabla X$ denotes its Jacobi matrix.
 This value is independent of the way of extension; see e.g.\ \cite{G06}.
% 原稿8ページ
% If $\Gamma_t$ is a curve in $\mathbb{R}^2$, then $\operatorname{div}_{\Gamma_t}$ is nothing but the differentiation with respect to the arc length parameter $s$, i.e.,
%\[
%	\operatorname{div}_{\Gamma_t} = \partial_s
%	= \frac{1}{(1 + w^2_{x_1})^{1/2}} \partial_{x_1}.
%\]
%The weighted curvature is of the form
%\[
%	-\operatorname{div}_{\Gamma_t} \zeta 
%	= \frac{-1}{(1 + w^2_{x_1})^{1/2}} 
%	\partial_{x_1} \left( \frac{\partial\sigma}{\partial p_1}(-w_{x_1},1) \right),
%\]
%since $\nu=(-w_{x_1},1)/(1+w^2_{x_1})^{1/2}$ and $\nabla_p \sigma$ is positively zero-homogeneous.
 In our setting,
\[
	\operatorname{div}_{\Gamma_t} \zeta(\nu)
	= \sum^{n-1}_{\ell=1} \frac{\partial}{\partial x_\ell} \left(\frac{\partial\sigma}{\partial p_\ell}(\nu) \right)
\]
where $\nu=(-\nabla'w, 1)/\left(1+|\nabla'w|^2\right)^{1/2}$.
 Indeed,
\[
	\operatorname{trace}(\nu\otimes\nu\nabla\zeta)
	= \sum^n_{i,j=1} \nu_i \nu_j \frac{\partial}{\partial x_i} \left((\partial_{p_j}\sigma)(\nu)\right)
	= \sum^n_{i,j,\ell=1} \nu_i \nu_j (\partial_{p_j} \partial_{p_\ell}\sigma)(\nu)\partial_{x_j}\nu^\ell = 0
\]
since $\sum^n_{j=1}\nu_j \partial_{p_j}\left((\partial_{p_\ell}\sigma)(\nu)\right)=0$ by positively zero-homogeneity\footnote{Let $s$ be a real number. A function $f$ allowing values $\pm\infty$ defined in a vector space $V$ is called \emph{positively $s$-homogeneous} if $f(\lambda v)=\lambda^s f(v)$ holds for all $\lambda>0$ and $v\in V$.} of $\partial_{p_\ell}\sigma$.
 Moreover, since $\partial_{p_\ell}\sigma(\nu)$ is independent of $x_n$, we have the desired identity.
 If $\Gamma_t$ is a curve in $\mathbb{R}^2$, then
\[
	-\operatorname{div}_{\Gamma_t} \zeta
	= -\partial_{x_1} \left(\frac{\partial\sigma}{\partial p_1}(-w_x,1) \right),
\]
since $\nabla_p \sigma$ is positively zero-homogeneous.

 We now observe that \eqref{AM1} is formally of the form
\begin{equation} \label{AMG}
	\frac{w_t}{(1+w^2_{x_1})^{1/2}} = -\partial_{x_1} \left( \frac{\partial\sigma}{\partial p_1}(-w_{x_1},1) \right).
\end{equation}
If $\sigma(p)=|p|$, then
\[
	\frac{\partial\sigma}{\partial p_1}(p) = \frac{p_1}{|p|} \quad\text{so that}\quad
	\frac{\partial\sigma}{\partial p_1}(-w_{x_1},1)
	= -\frac{w_{x_1}}{(1 + w^2_{x_1})^{1/2}},
\]
which yields a curve-shortening equation for a graph-like curve $\Gamma_t: x_2 = w(x_1,t)$, i.e.,
\[
	\frac{w_t}{(1+w^2_{x_1})^{1/2}}  = \partial_{x_1} \left( \frac{w_{x_1}}{(1 + w^2_{x_1})^{1/2}} \right) \quad\text{or}\quad
	w_t = \frac{w_{x_1x_1}}{1 + w^2_{x_1}}.
\]
We are interested in the case when $\sigma$ is crystalline.
 Let us consider
\[
	\sigma(p) = |p_1| + |p_2|
\]
so that the Frank diagram $F_\sigma$ is a square whose vertices are $(\pm1,0)$ and $(0, \pm1)$; see Figure \ref{figure1} for $F_\sigma$ and the corresponding Wulff shape $W_\sigma$.
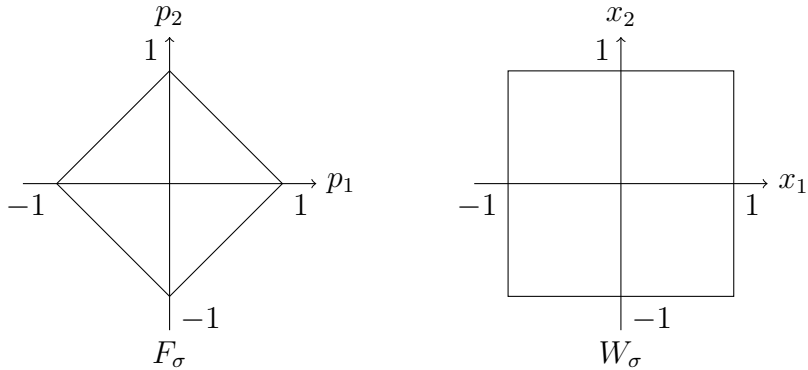
\begin{figure}[hbtp] % Figure 1
 \centering
\begin{tikzpicture}[scale=1.5]
\draw[->] (-1.3,0) -- (1.3,0) node[right] {$x_1$};
\draw[->] (0,-1.3) node[below] {$W_\sigma$} -- (0,1.3) node[above] {$x_2$};
\draw (1,1) -- (-1,1) -- (-1,-1) -- (1,-1) -- cycle;
\draw (0,-1) node[below right] {$-1$};
\draw (-1,0) node[below left] {$-1$};
\draw (1,0) node[below right] {$1$};
\draw (0,1) node[above left] {$1$};
\begin{scope}[xshift=-4cm]
\draw[->] (-1.3,0) -- (1.3,0) node[right] {$p_1$};
\draw[->] (0,-1.3) node[below] {$F_\sigma$} -- (0,1.3) node[above] {$p_2$};
\draw (1,0) -- (0,1) -- (-1,0) -- (0,-1) -- cycle;
\draw (0,-1) node[below right] {$-1$};
\draw (-1,0) node[below left] {$-1$};
\draw (1,0) node[below right] {$1$};
\draw (0,1) node[above left] {$1$};
\end{scope}
\end{tikzpicture}
 \caption{The Frank diagram and the Wulff shape for $\sigma(p) = |p_1| + |p_2|$.}
 \label{figure1}
\end{figure}
 Then \eqref{AMG} becomes $w_t=(1+w^2_{x_1})^{1/2} \partial_{x_1}(\operatorname{sgn}w_{x_1})$, which is formally equivalent to
% 原稿9ページ
\begin{equation} \label{TOT1}
	w_t = \partial_{x_1} ( \operatorname{sgn} w_{x_1}),
\end{equation}
where $\operatorname{sgn} p_1 = p_1/|p_1|$.
 This equation is a total variation flow equation in one-dimensional setting.
 If one calculates the right-hand side formally, then \eqref{TOT1} is
\[
	w_t = 2 \delta(w_{x_1}) w_{x_1 x_1},
\]
where $\delta$ denotes Dirac's delta. This shows 
\[
\partial_{x_1}(\operatorname{sgn} w_{x_1})=(1+w^2_{x_1})^{1/2} \partial_{x_1}(\operatorname{sgn} w_{x_1}).
\]
However, the quantity $\delta(w_{x_1})$ is undefined because it is a pull-back of the delta measure although it suggests the diffusion coefficient equals zero if $w_{x_1}$ is not equal to zero.
 In other words, the place where $w_{x_1}$ is not zero does not move.
 To see the speed where $w_{x_1}$ is zero, let us consider a special (Lipschitz) profile $x_2=w_0(x_1)$ which takes the minimum value on $[a,b]$ and $w_{0x_1}>0$ (resp.\ $w_{0x_1}<0$) in $x_1>b$ (resp.\ $x_1<a$), where $a<b$ (Figure \ref{figure2}).
\begin{figure}[hbtp] % Figure 1
 \centering
\begin{tikzpicture}[scale=3]
\draw[->] (-1.2,0) -- (1.7,0) node[right] {$x_1$};
\draw[->] (-1,-0.5) -- (-1,0.5) node[left] {$x_2$};
\draw[dashed] (0,0 + 0.6pt) node[above] {$a$} -- (0,-0.5);
\draw[dashed] (1,0 + 0.6pt) node[above] {$b$} -- (1,-0.5);
\draw[thick] (-0.9,0.3) .. controls (-0.6,-0.25) .. (0, -0.5) -- (1,-0.5)  .. controls (1.4,-0.25) .. node[below right] {$w_0$} (1.6,0.3);
\end{tikzpicture}
\caption{The graph of $w_0$.}
\label{figure2}
\end{figure}
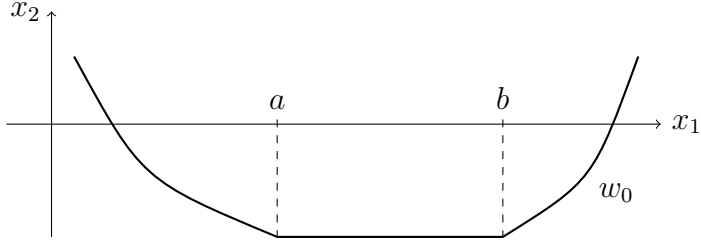
 We try to move this function by \eqref{TOT1}.
 Since it is natural to assume that the speed equals zero outside $[a,b]$, the important thing is to calculate the speed on $[a,b]$.
 Here we put ansatz: \\
``The speed $w_t$ on $[a,b]$ is spatially constant.'' \\
In other words, a flat part (called facet) stays as a facet and no bending nor facet splitting occurs.

% 原稿10ページ
We integrate \eqref{TOT1} in a neighborhood of $[a,b]$, i.e., $(a-\varepsilon,b+\varepsilon)$ with small $\varepsilon>0$ and obtain at $t=0$
\begin{align*} 
	\int^{b+\varepsilon}_{a+\varepsilon} w_t dx
	= \int^{b+\varepsilon}_{a+\varepsilon} \partial_{x_1} (\operatorname{sgn}w_{0x_1})dx
	&= \operatorname{sgn}w_{0x_1}(b+\varepsilon)
	- \operatorname{sgn}w_{0x_1}(a-\varepsilon) \\
	&= 1 - (-1) = 2.
\end{align*}
By our ansatz, the left-hand side is of the form
\[
	w_t (b-a)
\]
as $\varepsilon\to 0$.
 Thus, we obtain
\[
	w_t = 2/(b-a).
\]
The right-hand side is a nonlocal quantity and this is a one-dimensional version of the \emph{Cheeger ratio} $\mathcal{H}^{n-1}(\partial\Omega)/\mathcal{L}^n(\Omega)$ defined for a domain $\Omega$ in $\mathbb{R}^n$, where $\mathcal{L}^n(\Omega)$ denotes the Lebesgue measure of $\Omega$ while $\mathcal{H}^{n-1}(\partial\Omega)$ denotes the $(n-1)$-dimensional Hausdorff measure of the boundary $\partial\Omega$ of $\Omega$.
 We now observe that the crystalline curvature should be determined by a semilocal quantity like Cheeger ratio if one assumes the ansatz.

In one-dimensional setting, this ansatz is justified in the sense that such a profile is approximated by a solution of uniformly parabolic equations which approximates the equation \eqref{TOT1}.
 For example, order-preserving property called comparison principle is expected to hold.
 However, in higher dimensional setting, as we see later this ansatz is no longer appropriate.
 For example, this ansatz violates the comparison principle.

% 原稿11ページ
For curve evolutions, using this ansatz J.\ E.\ Taylor \cite{T1} and independently S.\ B.\ Angenent and M.\ E.\ Gurtin \cite{AG} introduced a special class of polygonal curves called admissible.
 We say that an oriented polygon is \emph{admissible} if the orientation (normal $\nu$) of each facet (edge) is one of that in $\partial W_\sigma$ and the orientation of adjacent facets should be adjacent in $\partial W_\sigma$.
 Here $W_\sigma$ is the Wulff shape associated with anisotropy $\sigma$ and it is a convex polygon if $\sigma$ is crystalline.
 If the second condition (called adjacency condition) is not required, one expects that new facets may be created because of a strong curvature effect.
 We shall discuss this point in Section \ref{PF}.
 Let $\{\Gamma_t\}$ be a smooth family of admissible polygons.
 In other words, vertices of $\Gamma_t$ are assumed to move $C^1$ in time $t$.
 The motion of vertices is completely determined by the crystalline flow equation \eqref{GE}.
 Here, $\kappa_\sigma$ of each facet with normal $\nu$ is assumed to be equal to $\chi\Delta/L$, where $L$ is the length of the facet and $\Delta$ is the length of the facet of $W_\sigma$ with normal $\nu$; $\chi$ takes $+1$, $-1$, $0$ depending upon convexity near the facet.
% 原稿12ページ
 Since $L$ depends upon vertices, combining these equations, a system of ordinary differential equations (ODEs) for vertices or lengths is obtained.
 Its initial value problem is uniquely solvable at least when $g$ is (locally) Lipschitz continuous.
 For later convenience, we say that $\{\Gamma_t\}$ is a crystalline flow if $\Gamma_t$ is a smooth family of admissible polygons satisfying the system of these ODEs. 
 However, there is a chance that in finite time a facet disappears.
 Fortunately, in many cases at the time when a facet disappears, $\Gamma_t$ is still admissible so one is able to continue to solve the system of ODEs with fewer facets.
 This approach is very simple and it is easy to compute the crystalline flow \cite{T1}, \cite{T3}, \cite{T0}. 
 Moreover, it satisfies the desired property like comparison principle which says that if one admissible polygon encloses another, then the corresponding crystalline flow starting from these polygons keeps this order; see \cite{T3}, \cite{GGu}.

% 原稿13ページ
There is another approach based on the theory of maximal monotone operators initiated by Y.\ K\=omura \cite{Ko} and developed by H.\ Brezis \cite{Br73} and others in late 1960s and 1970s.
 A basic theory asserts the unique global-in-time solvability of the initial value problem for the gradient flow equation whose ``energy'' $\mathcal{E}$ is a convex, lower semicontinuous functional in a Hilbert space $H$ equipped with an inner product $\langle\ ,\ \rangle$ so that $\|f\|^2_H=\langle f,f \rangle$.
 More precisely, it is a solvability for the system $w_t\in-\partial\mathcal{E}(w)$ where $\partial\mathcal{E}(w)$ is the subdifferential of $\mathcal{E}$ at $w$, which is an extended notion of a differential of $\mathcal{E}$.
 It is defined as
\[
	\partial\mathcal{E}(w) = \left\{ f \in H \mid
	\mathcal{E}(w+h) - \mathcal{E}(w) \geq \langle f,h \rangle\ \text{for all}\ h \in H \right\}.
\]
 Note that $\mathcal{E}$ may not be differentiable so that $\partial\mathcal{E}(w)$ may not be a singleton.
 However, the solution is unique and it ``knows'' how to grow even though the evolution law looks ambiguous.
 Actually, the solution is right differentiable in time and its speed equals to the minimal section (canonical restriction) $\partial^0\mathcal{E}(w)$ of $\partial\mathcal{E}(w)$, i.e.,
\[
	\partial^0\mathcal{E}(w) 
	= \mathrm{argmin} \bigl\{ \|f\|_H \bigm| f \in \partial\mathcal{E}(w) \bigr\},
\]
which is uniquely determined.
 In \cite{FG}, it is shown that if $\{\Gamma_t\}$ is given as the graph of a periodic function of one variable, then the equation $V=M(\nu)\kappa_\sigma$ can be written as the gradient flow system.
Moreover, the speed given by the general theory is the same as the one given in the ansatz on a facet.
 This suggests the approach by \cite{T1}, \cite{AG} is quite natural.
% 原稿14ページ
 In fact, it is shown in \cite{FG} that the crystalline flow is obtained as a limit of approximate solutions solving a usual uniformly parabolic problem approximating the original problem.
 This justifies the ansatz for curve evolution.
 The proof is based on a general convergence theory for the gradient system developed by \cite{BP} and \cite{Wa}.
 To apply the theory, it suffices to prove that the approximating energy $\mathcal{E}^\varepsilon$ converges to $\mathcal{E}$ in the sense of Mosco, i.e., it satisfies
\begin{subequations}
\label{mosco}
\begin{enumerate}
\item[(i)] lower semicontinuity under weak topology:
\begin{align}
\label{mosco-i}
	\mathcal{E}(w) \leq \varliminf_{\varepsilon\downarrow 0}\mathcal{E}^\varepsilon(w_\varepsilon)
	\quad\text{for}\quad w_\varepsilon \rightharpoonup w \ (\text{as }\varepsilon \to 0);
\end{align}
\item[(ii)] existence of strong recovery sequence: for any $v\in H$, there is $v_\varepsilon\to v$ as $\varepsilon\to 0$ such that
\begin{align}
\label{mosco-ii}
	\mathcal{E}(v) = \lim_{\varepsilon\downarrow 0}\mathcal{E}^\varepsilon(v_\varepsilon).
\end{align}
\end{enumerate}
\end{subequations}
The nonlocal property of the speed related to a total-variation-type singular energy was also observed in \cite{HZ}.

% 原稿15ページ
If the flow equation is written as a gradient flow of a convex, lower semicontinuous functional in a Hilbert space, one is able to calculate the speed by calculating the minimal section.
 It is a kind of an obstacle problem as we will see later.
 Reflecting this idea, G.\ Bellettini, M.\ Novaga and M.\ Paolini \cite{BNP99} gave an example that the speed of a facet may not be a constant on a facet.
 In other words, the quantity $\kappa_\sigma$ may not be a constant on a facet since otherwise it would contradict a comparison principle.
 Later, they gave a characterization of non-constancy of $\kappa_\sigma$ on a facet depending on shape.
 To illustrate the problem, let us consider a closely related problem: the total variation flow equation
\begin{equation} \label{TOT2}
	w_t = \operatorname{div} \left( \nabla w/|\nabla w| \right)
\end{equation}
on an $n$-dimensional torus $\mathbb{T}^n = \Pi^n_{i=1}(\mathbb{R}/\omega_i \mathbb{Z})$, $\omega_i>0\ (i=1,\ldots,n)$.
 Except Section~\ref{FO}, we shall assume $\omega_i=1$ for simplicity.
 It can be interpreted as a gradient flow of the total variation energy
\[
	E[w] = \int_{\mathbb{T}^n} |\nabla w| 
	:= \sup \left\{\int_{\mathbb{T}^n} w \operatorname{div}z\; dx \Bigm|
	\left| z(x) \right|\leq 1,\ z \in C^1(\mathbb{T}^n, \mathbb{R}^n) \right\}
\]
for an $L^2$ function $w$.
 We set the energy $\mathcal{E}$ in the Hilbert space $H=L^2(\mathbb{T}^n)$ such that $\mathcal{E}=E$.
 Then, it is not difficult to see that $\mathcal{E}$ is convex and lower semicontinuous in $H = L^2(\mathbb{T}^n)$.
% 原稿16ページ
 The problem \eqref{TOT2} should be interpreted as
\[
	w_t \in  -\partial\mathcal{E}(w)
\]
and there is a unique solution starting from $w_0 \in H=L^2(\mathbb{T}^n)$.
 The speed is given as the minimal section and we are interested in the value.
 We restrict ourselves to a facet where $w$ is ``convex'' in its neighborhood.
 We fix $t>0$ and let $w$ take its minimum on a facet, i.e.,
\[
	F = \left\{ x \in \mathbb{T}^n \Bigm| w(x,t) = \min_{y\in\mathbb{T}^n} w(y,t) \right\}.
\]
Assume that the boundary of $F$ is smooth.
 Then it turns out that
\begin{align*}
	\left. -\partial^0 \mathcal{E}(w) \right|_F &= \operatorname{div} z, \\
	z &= \operatorname{argmin} \bigg\{ \int_F |\operatorname{div} \zeta|^2 \Bigm| \zeta \cdot \nu_F =1\ \text{on}\ \partial F,\ 
	|\zeta| \leq 1\ \text{in}\ F \bigg\}. 
\end{align*}
Here $\nu_F$ is the exterior unit normal of $F$.
 This is a convex minimization problem but it is of obstacle type because of the constraint $|\zeta|\leq 1$.
 Although the minimizer is not unique, $\operatorname{div}z$ is uniquely determined.
 The characterization of the minimal section is nontrivial but it can be done for the total variation flow equation.
 For a detailed explanation, the reader is referred to a very nice book by F.\ Andreu-Vaillo, V.\ Caselles and J.\ M.\ Maz\'on \cite{ACM}.
 If $\operatorname{div}z$ is constant, we say that $F$ is \emph{calibrable}.
% 原稿17ページ
 There are several necessary and sufficient conditions; see e.g.\ \cite{BNP01c} for the curvature flow.
 The reader is referred to \cite{ACM}.
 We shall discuss this topic in Section \ref{AO}.
 If it is calibrable, then $\operatorname{div}z$ must be the Cheeger ratio, i.e., $\operatorname{div}z=\mathcal{H}^{n-1}(\partial F)/\mathcal{L}^n(F)$.
 Indeed, integration by parts yields
\[
	(\operatorname{div}z) \mathcal{L}^n(F) 
	= \int_F \operatorname{div}z\; dx
	= \int_{\partial F} z \cdot \nu_F\; d \mathcal{H}^{n-1}
	= \mathcal{H}^{n-1} (\partial F).
\]
In general, $\operatorname{div}z \in L^\infty \cap BV$ but may be discontinuous as shown in \cite{BNP01a}, \cite{BNP01b}.
 Since there may exist non-calibrable facets, it took a long time to construct a solution in a general setting.
 G.\ Bellettini and M.\ Novaga \cite{BN} introduced a notion of a solution based on distance function reflecting the variational structure and proved its uniqueness.
 However, its existence is only proved for convex initial data \cite{BCCN}.
 It is quite recent that the well-posedness problem is settled by two groups through level-set method, which is the main topic of this survey.

% 原稿18ページ
Although there are several approaches to solve the problems by now, they are roughly classified into three main ones.
 The first approach is to consider a special class of evolving polygons by reducing the problem to a system of ODEs we discussed before.
 This approach is valid only for curve evolution.
 The second approach is a variational approach.
 A simple way is to apply the theory of maximal monotone operators which is restricted for the graph case but it has an advantage to apply to a higher order crystalline flow for example crystalline surface diffusion equation for a graph-like surface.
 The reader is referred to \cite{GG10} for this topic as well as Section~\ref{FO2}.
 A variant of this variational approach involving a distance function yields a global well-posedness for convex sets as mentioned before \cite{BCCN}.

% 原稿19ページ
The third approach is a viscosity approach.
 This is based on the theory of viscosity solutions, which was originally introduced to characterize the value function of a control problem as a solution of a Hamilton-Jacobi equations; see \cite{CIL}.
 The notion of a viscosity solution is based on a comparison principle for the second-order elliptic or parabolic equations which can be degenerate.
 It does not depend on a variational structure.
 However, since the crystalline flow is non-local, one needs to adjust the theory.
 This is not trivial even for an evolution of a curve.
 In the case of graph-like curves, i.e., $\Gamma_t$ is given as a graph $w=w(x_1,t)$, the notion of a viscosity solution was adjusted for general crystalline flow when $w$ is periodic in $x_1$ \cite{GG}, \cite{GG1}.
 It can be approximated by a smoother problem as proved in \cite{GG2}.

% 原稿20ページ
This viscosity approach was later extended to a closed curve by adjusting the level-set method \cite{GG4}, \cite{GG3}.
 The original level-set method based on the theory of viscosity solution was introduced by \cite{ES}, \cite{CGG} for the mean curvature flow equations.
 The idea of the original level-set method for the mean curvature flow $V=\kappa$ is to consider its level set flow equation
\[
	u_t - |\nabla u| \operatorname{div} \left( \frac{\nabla u}{|\nabla u|} \right) = 0
\]
which requires that \emph{each} level set moves by $V=\kappa$.
 For a given initial hypersurface $\Gamma_0$, one constructs a continuous function $u_0$ such that $\Gamma_0$ is the zero level set of $u_0$ and solves the level-set flow equation globally-in-time and sets $\Gamma_t$ as the zero level set of the solution.
 A unique solvability is guaranteed by the theory of viscosity solutions.
 Moreover, $\Gamma_t$ is uniquely determined by $\Gamma_0$.
 However, as already pointed out in \cite{ES}, $\Gamma_t$ may have interior even if $\Gamma_0$ has no interior, Figure~\ref{fig:fattening}; see also \cite{G06}.
\begin{figure}
\centering
\begin{tikzpicture}
\fill[even odd rule, fill = black!30!white]
% generated by python/fattening.py
(0,0) rectangle (1.549, 1.549) (0.5528, 0.5528) rectangle (1.447, 1.447)
(0,0) rectangle (-1.549, -1.549) (-0.5528, -0.5528) rectangle (-1.447, -1.447)
;
\draw[thick] (0,0) rectangle (2, 2) (0,0) rectangle (-2, -2);
\end{tikzpicture}
\caption{Example of a fattening in the crystalline flow for $\sigma(p) = |p_1| + |p_2|$ with initial curve $\Gamma_0$ given by the figure-8-shaped solid line. The set $\Gamma_t$ immediately ($t > 0$) fattens. The gray area denotes $\Gamma_t$ at $t = 0.8 t^*$ where $t^*$ is the extinction time of the individual squares.}
\label{fig:fattening}
\end{figure}
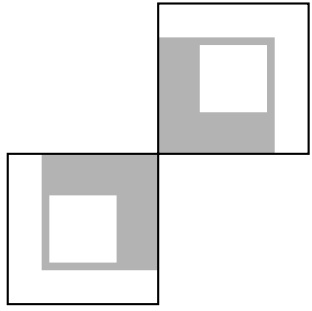
 This phenomenon is called fattening and from the point of an evolution of hypersurfaces this is considered a non-uniqueness phenomenon.
 A basic merit of this approach is to handle a topological change.
 The generalized solution $\Gamma_t$ of course agrees with a smooth solution if the latter exists though the proof is less trivial \cite{ES}, \cite{GGo}.
 For a general theory of the level-set method for smooth anisotropy, see \cite{CGG} or a book \cite{G06}.
 The level-set method itself was introduced by \cite{OS} for numerical study and independently by \cite{OJK} to explain a scaling law of $V=\kappa$.
 For the development of the numerical approach, see \cite{Se} and \cite{OF}.

It took quite a long time to extend this theory to evolution of a hypersurface mainly because the crystalline curvature $\kappa_\sigma$ may not be a constant on a facet.
 A first breakthrough is done by \cite{MGP1}, where the viscosity theory was extended to a total-variation-flow-like equation; see also \cite{MGP2}.
 Later it was extended to level-set flow equations, \cite{GP1}, \cite{GP2}, and
 to the case when there is a spatially inhomogeneous driving force term \cite{GP3}.
 In the meanwhile, another approach to construct a level-set flow based on distance functions which goes back to \cite{So} was developed independently.
 In fact, A.\ Chambolle, M.\ Morini and M.\ Ponsiglione \cite{CMP} constructed a level-set flow for $V=\sigma\kappa_\sigma$ for very general $\sigma$ containing crystalline $\sigma$ as a special case.
 With M.\ Novaga they even extended their approach in \cite{CMNP} for more general equations  with mobility and spatially inhomogeneous driving force term.
 In both theories, the theory of maximal monotone operators is reflected in some sense.
 In the purely viscosity approach by \cite{GP1}, \cite{GP2}, \cite{GP3}, the value $\kappa_\sigma$ is defined as the minimal section of the crystalline interfacial energy.
 In the approach by \cite{CMP}, \cite{CMNP}, the distance function from the zero level-set of a solution is interpreted as a supersolution of the original gradient flow of the form $u_t\in-\partial E(u)$, where $E$ is an anisotropic total variation energy with density $\sigma$.

We warn the reader that the value $\kappa_\sigma$ is not determined completely by the facet $F$ if the problem is spatially inhomogeneous as pointed out by \cite{GP3}.
 If there is a non-constant driving force $C=C(x)$, then $\kappa_\sigma+C$ is not just the sum of the two quantities. See \cite{GP3} for more details and futher references. 

We do not intend to cover all topics related to well-posedness for a crystalline flow.
 Several interesting topics like a crystalline multi-phase curvature flow are missing in this paper.
 For a multi-phase crystalline flow, see \cite{BCherN}.

This paper is organized as follows.
 In Section \ref{SM}, we give several model equations for curvature flow equations involving a crystalline curvature.
 In Section \ref{PF}, an evolution of a polygon is discussed.
 In Section \ref{ES}, some explicit solutions such as self-similar solutions are discussed.
 In Section \ref{AO}, we give an approach by the theory of maximal monotone operators.
 In Section \ref{AV}, we give an approach based on viscosity solutions.
 In Section \ref{AD}, we give an approach based on distance functions.
 In Section \ref{SN}, some numerics are given.
 In Section \ref{FO}, examples of a fourth-order problem and a volume-preserving flow are discussed.

%%%%%%% Section 2
\section{Some models} \label{SM}

% 原稿 23ページ
We begin with second-order models in materials sciences.
 There is an axiomatic derivation of evolution laws of phase-interfaces involving bulk energy and surface energy with constitutive relation compatible with thermodynamical laws in \cite{AG}, \cite{Gu}.
 Its explicit form is
\[ 
	b(\nu,V)V = \kappa_\sigma -f
	\quad\text{with}\quad b(\nu,V) \geq 0
\]
where $f$ is a driving force term coming from bulk interface difference which is assumed to be a constant in \cite{AG}, \cite{Gu}.
 The function $b$ is called a kinetic coefficient.
 If $b(\nu,V)$ is independent of $V$ and positive, then it is reduced to
\[
	V = M(\nu)(\kappa_\sigma + C)
\] 
with $C=-f$, $M(\nu)=b(\nu)^{-1}$.
 If $b(\nu, V)$ is taken so that
\[
	b(\nu, V)V = \mathrm{log}(1+V)
\]
with $f=0$, this is nothing but the model of thermal grooving of a surface due to evaporation-condensation proposed by W.\ W.\ Mullins \cite{Mu57}. Here is a way of derivation. The Gibbs-Thomson law reads
\[
	\operatorname{log}(p/p_0) = \beta(-\kappa_\sigma)
\]
with positive constant $\beta>0$. Here $p$ is the pressure and $p_0$ is the atmospheric pressure. The evolution law is
\[
	V = M(\nu)(p_0-p).
\]
If $M(\nu)\equiv 1$, $p_0=1$, then one gets
\begin{equation} \label{2EXP}
	V = 1 -\operatorname{exp}(-\beta\kappa_\sigma).
\end{equation}
If  the right-hand side is linearized around $\kappa_\sigma=0$, we get $V=\beta\kappa_\sigma$.
 See the discussion by N.\ Hamamuki \cite{H}.
% 原稿 24ページ
As we will see later in this section, a model similar to $V=\kappa_\sigma$ was introduced by H.\ Spohn \cite{Sp} when $\sigma$ is a kind of crystalline anisotropy to model evaporation-condensation below the roughening temperature. 

Another source of equations stems from an image processing.
 An axiomatic derivation is provided by \cite{AGLM}.
 For curve evolution, equation
\begin{equation} \label{2POW}
	V = |\kappa_\sigma|^{\alpha-1} \kappa_\sigma \quad
	\alpha>0
\end{equation}
is important especially with $\alpha=1/3$, where the evolution law is invariant under affine transform (not only under rotation, dilation and translation) when $\sigma$ is isotropic.
 In higher dimensional case, the corresponding equation should be $V=K^{1/(n+1)}$ where $K$ is the Gauss curvature not the mean curvature. 
 A crystalline Gaussian curvature flow $V=K_\sigma$ has been studied to approximate the Gaussian curvature flow; see e.g.\ \cite{UY}.
 However, we do not touch this topic in this paper.
 There are many examples of curvature flows (see e.g.\ \cite[Chapter 1]{G06}).
 In the case that the mean curvature is involved like the inverse mean curvature flow equation, it is easy to generalize
\[
	V = -1/\kappa_\sigma.
\]
% 原稿 25ページ
If $\sigma$ is isotropic, then the equation was used to prove the positive mass conjecture \cite{HI} since the Geroch mass is monotone under this flow.

We note that the total variation flow
\[
	w_t = \operatorname{div}' \left( \nabla'w/|\nabla'w| \right)
\]
can be understood as a particular case of $V=M(\nu)\kappa_\sigma$ as discussed in the introduction for evolution of graph-like curves.
 If an evolving surface $\Gamma_t$ is given as the graph of $w=w(x',t)$, $x'\in\R^{n-1}$, the total variation flow for $w$ can be written as
\[
	V = M(\nu) \kappa_\sigma
\]
with
\begin{align*}
	\sigma(p) &= |p'| + |p_n| \quad\text{with}\quad
	p = (p', p_n) \\
	M(\nu) &= \nu_n \quad\text{with}\quad
	\nu= (\nu', \nu_n)
\end{align*}
provided that the slope of $w$ is less than $1$.
 Here $\nu'=-\nabla'w/\left(1+|\nabla'w|^2\right)^{1/2}$ and $\nu_n=1/\left(1+|\nabla'w|^2\right)^{1/2}$.

% 原稿 26ページ
The model proposed by H.\ Spohn \cite{Sp} is almost the same.
 Here $w$ denotes the height of the crystal surface at $x'$ and at time $t$.
 It is of the form
\[
	w_t = \operatorname{div}' \left( \nabla'w/|\nabla'w| \right)
	+ \beta\operatorname{div}' \left( |\nabla'w| \nabla'w \right),
\]
where $\beta>0$ is a constant.
 If one writes it in the form of a surface evolution, it is
\[
	V = M(\nu) \kappa_\sigma
\]
with $\sigma(p)=|p'|+\beta|p'|^3/3+|p_n|$ under the same slope restriction; without slope restriction, we may take $\sigma(p)=|p'|+\beta|p'|^3/3$.

There are several fourth-order models.
 For relaxation of crystal surface, a fourth-order total variation type equation is proposed by \cite{Sp}.
 Its explicit form is
\[
	w_t = -\Delta' \left( \operatorname{div} \left( \nabla'w/|\nabla'w| \right)
	+ \beta\operatorname{div} \left( |\nabla'w| \nabla'w \right) \right),
\]
where $\Delta'$ denotes the Laplacian in $x'$ variable, i.e., $\Delta'=\operatorname{div}' \operatorname{grad}' =\nabla'\cdot\nabla'$.
 This equation is derived as a continuum limit of models describing motion of steps on crystal surface as discussed in \cite{Od}, where a numerical simulation is given.
 This model describing step-motion is microscopic in the direction of height but macroscopic in the horizontal direction.
 We refer the reader to a nice review article by R.\ V.\ Kohn \cite{Koh} on this issue.
 Of course, if $\beta=0$, this is nothing but the fourth-order total variation flow.
 This is popular for image processing.
 For example, Osher-Sol\'e-Vese \cite{OSV} model gives the fourth-order total variation flow of the form
\[
	u_t = -\Delta \operatorname{div} \left( \nabla u/|\nabla u| \right)+ \lambda(f-u)
\]
for $\lambda>0$, and given $f$.
 See also \cite{ElS} for such a flow, where the well-posedness of the equation is proved by using the Galerkin method.
 For relaxation phenomena, W.\ W.\ Mullins \cite{Mu57} introduced a surface diffusion flow equation; see also \cite{CT94} for derivation.
 It is of the form
\[
	V = - \operatorname{div}_{\Gamma_t} j, \quad
	j = - \operatorname{grad}_{\Gamma_t} \rho 
\]
\[
	\operatorname{log} (\rho/\rho_0) = \frac{k\mu}{T}, \quad
	\mu = \kappa_\sigma,
\]
where $T$ is a given temperature and $\rho_0$ is an equilibrium density; $k$ is a positive constant.
 The quantity $j$ is the mass flux and $\mu$ is the chemical potential.
% 原稿 27ページ
 The resulting equation is
\begin{equation} \label{4EXP}
	V = \Delta_{\Gamma_t} \operatorname{exp} (-k \kappa_\sigma/T), \quad
	\Delta_\Gamma = \operatorname{div}_\Gamma \operatorname{grad}_\Gamma;
\end{equation}
here, $\Delta_\Gamma$ denotes the Laplace-Beltrami operator on the surface $\Gamma$.
 We shall set $k=1$, $T=1$ for simplicity of presentation to get
\[
	V = \Delta_{\Gamma_t} \operatorname{exp} (-\kappa_\sigma).
\]
If one linearizes around $\kappa_\sigma=0$, the resulting equation is
\[
	V = - \Delta_{\Gamma_t} \kappa_\sigma.
\]
If $V$ is replaced by an upward velocity and $\Delta_{\Gamma_t}$ is replaced by $\Delta'$ for the graph of $w$, then the equation becomes the fourth-order total variation flow if $\sigma(p)=|p'|$, i.e.,
\[
	w_t = - \Delta' \left( \operatorname{div}' \left( \nabla'w/|\nabla'w| \right) \right).
\]

One significant property of the surface diffusion flow is the preserving property of the volume (area) enclosed by $\Gamma_t$.
 This is not the case for the second-order problem.
 However, one is able to consider a volume-preserving crystalline curvature flow, which is a nonlocal equation.
 For example, the volume-preserving version of \eqref{AM1} is of the form
\[
	V = \kappa_\sigma - \frac{1}{\mathcal{H}^{n-1}(\Gamma_t)} \int_{\Gamma_t} \kappa_\sigma d\mathcal{H}^{n-1}
\]
so that $\int_{\Gamma_t} Vd\mathcal{H}^{n-1}=0$. See Section~\ref{FO1} for more discussion of the volume-preserving problem.

%%%%%%% Section 3
\section{Polygonal flow} \label{PF}
% 原稿 3-1
In this section, we consider a special class of a polygonal flow called admissible introduced by J.\ Taylor \cite{T1} and S.\ B.\ Angenent and M.\ E.\ Gurtin \cite{AG} for a planar purely crystalline curvature flow equation. \\
\textbf{Admissible polygonal flow.}
 We first introduce a special class of a polygonal flow associated to a purely crystalline anisotropy $\sigma$.
 Let $W_\sigma$ denote the Wulff shape corresponding to $\sigma$.
 Since the anisotropy $\sigma$ is purely crystalline, $W_\sigma$ is a bounded, convex polygon containing the origin as an interior point.
 Let $\mathcal{N}$ be a finite subset of the unit circle so that it is the set of all orientations (exterior normals) of edges on the boundary $\partial W_\sigma$ of $W_\sigma$.
 We call $\mathcal{N}$ the set of \emph{admissible directions}.
 This set can be written as
\[
	\mathcal{N} = \{ \mathbf{n}_k \}^m_{k=1}
	\quad\text{with}\quad \mathbf{n}_k = (\cos\theta_k, \sin\theta_k)
\]
with $0\leq\theta_1<\cdots<\theta_m<2\pi$.
 The set $\Theta=\{ \theta_k \}^m_{k=1}$ is called the set of \emph{admissible angles}, which is considered as a subset in $\mathbb{T}=\mathbb{R}/2\pi\mathbb{Z}$.
 For example, $\theta_1$, $\theta_{m-1}$ are adjacent to $\theta_m$.
 We say that an (oriented) polygon is \emph{admissible} if
% 原稿 3-2
\begin{enumerate}
\item[(i)] (direction condition) the orientation of each facet (edge) is in $\mathcal{N}$; 
\item[(ii)] (adjacence condition) the angles of orientations of adjacent facets should be adjacent.
\end{enumerate}
An evolving polygon $\{\Gamma_t\}_{t\in I}$ is an \emph{admissible polygonal flow} if $\Gamma_t$ is an admissible polygon for $t\in I$ and the motion of all vertices is $C^1$ in time $t\in I$, where $I$ is a time interval. \\
\textbf{Crystalline curvature.}
 Since the Wulff shape is a substitute of the unit disk, it is natural to postulate that $\kappa_\sigma=-1$ on $\partial W_\sigma$.
 Let $\Delta(\mathbf{n})$ denote the length of a facet (edge) of $\partial W_\sigma$ whose orientation equals $\mathbf{n}$.
 For a general admissible polygon $\Gamma$, let $S$ denote one of its facets.
 By the ansatz for curve evolution, $\kappa_\sigma$ on $S$ must be a constant and its value must be a kind of Cheeger ratio.
 In our setting on $S$ with orientation $\mathbf{n}_S$, it is natural to assign
\[
	\kappa_\sigma = \chi \Delta(\mathbf{n}_S)/L,
\]
where $L$ is the length of the facet $S$ and $\chi$ is a transition number, i.e., $\chi=+1$ (resp.\ $-1$) if $\Gamma$ is convex (concave) in the direction of $\mathbf{n}_S$ near $S$, and otherwise $\chi=0$; see Figure~\ref{fig:chi-sign}.
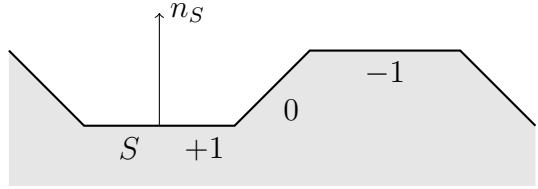
\begin{figure}
\centering
\begin{tikzpicture}
\fill[black!10!white] (0,1) -- (1,0) --   (3,0) --  (4,1) --  (6,1) -- (7,0) -- (7,-0.8) -- (0,-0.8) -- cycle;
\draw[thick] (0,1) -- (1,0) -- node[pos=0.3,below] {$S$} node[pos=0.8,below] {$+1$} (3,0) -- node[below right] {$0$} (4,1) -- node[below] {$-1$} (6,1) -- (7,0);
\draw[->] (2,0) -- (2, 1.5) node[right] {$n_S$};
\end{tikzpicture}
\caption{Value of $\chi$ based on the convexity/concavity of the facets.}
\label{fig:chi-sign}
\end{figure}
 By this definition, $\kappa_\sigma=-1$ on $\partial W_\sigma$ since $\mathbf{n}_S$ is taken outward from $W_\sigma$; this is the outward curvature.
 We measure the curvature by comparing with the Wulff shape, which is consistent with the definition of the usual curvature by the inverse of the radius of the osculating circle called a circle of curvature.
 This quantity $\kappa_\sigma$ is often called a \emph{crystalline curvature}. \\
% 原稿 3-3
\textbf{Derivation of a system of ODEs.}
 Let $\{\Gamma_t\}_{t\in I}$ be an admissible polygonal flow such that for $t\in I$, $\Gamma_t$ is an $\ell$-polygon consisting of facets $\left\{S_j(t)\right\}^\ell_{j=1}$ numbered counterclockwise and vertices of $S_j(t)$ whose motion is $C^1$ in time. 
 Let $V_j(t)$ denote the normal speed of $S_j(t)$ in the direction of the orientation $\mathbf{n}_j$ of $S_j(t)$.
 We consider a general form of the equation
\begin{equation} \label{CrG}
	V = g (\nu, \kappa_\sigma)
\end{equation}
with $g$ non-decreasing in the second variable so that the problem is at least degenerate parabolic.
 For an admissible polygonal flow, this equation is formally reduced to
\begin{equation} \label{CrG1}
	V_j(t) = g \left( \mathbf{n}_j, \chi_j\Delta(\mathbf{n}_j)/L_j(t) \right),\quad
	j = 1, \ldots, \ell,
\end{equation}
where $L_j(t)$ is the length of $S_j(t)$ and $\chi_j$ is the transition number of $S_j(t)$.
 By an elementary geometry Fig.~\ref{fig:page1}, we observe that
\begin{figure}
\centering
\includegraphics[width=3.5in]{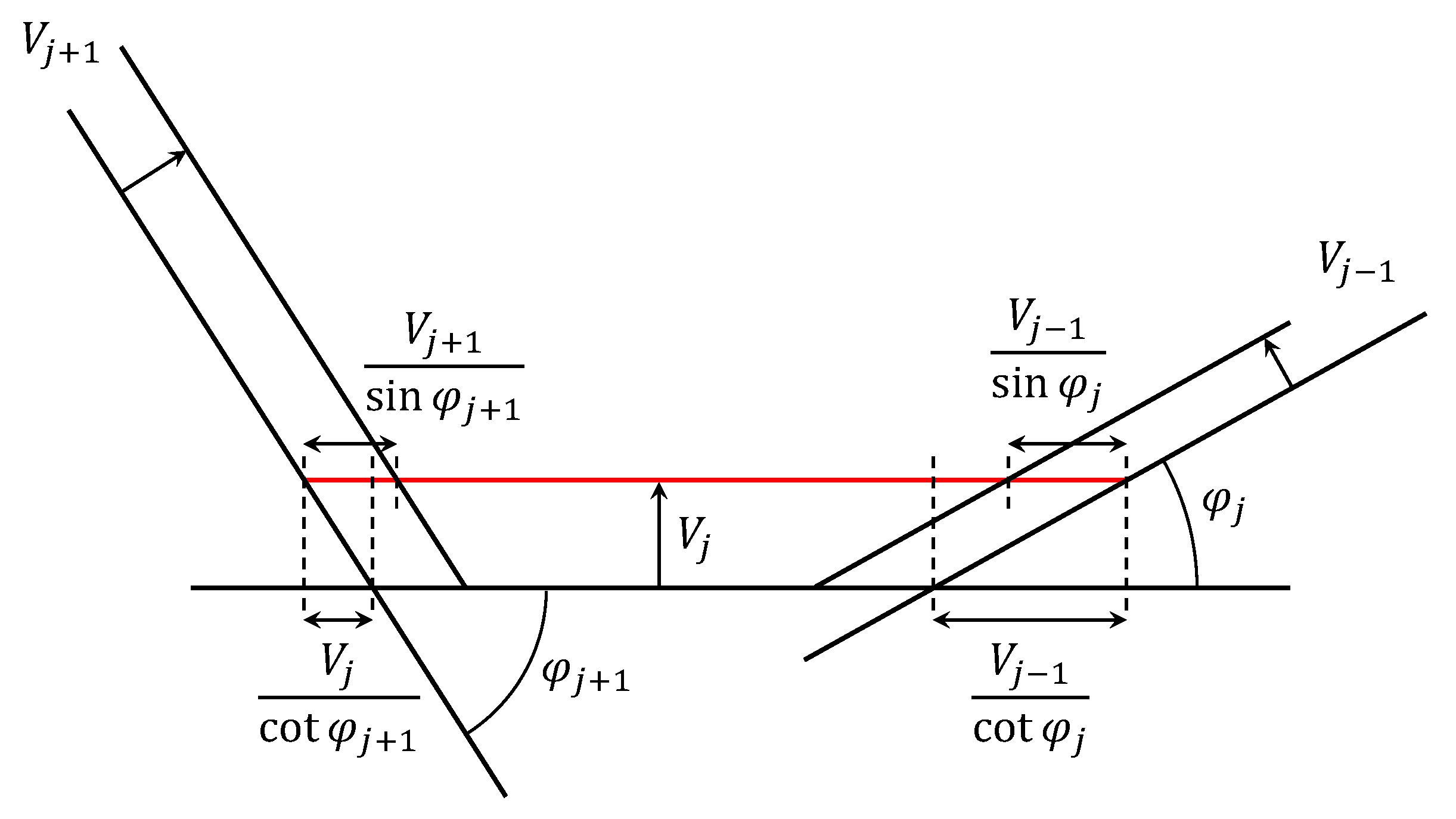}
\caption{}
\label{fig:page1}
\end{figure}
\begin{align}
&\begin{aligned}\label{Tr} % (9)
	\frac{dL_j(t)}{dt} = & -\frac{1}{\sin\varphi_j} V_{j-1}(t)
	+ (\cot\varphi_j + \cot\varphi_{j+1}) V_j(t) \\
	& -\frac{1}{\sin\varphi_{j+1}} V_{j+1}(t), \quad j = 1, \ldots, \ell,
\end{aligned}
\end{align}
where $\varphi_j=\theta_j-\theta_{j-1}$ and $\theta_j$ is the angle of $\mathbf{n}_j$, i.e.,
\[
	\mathbf{n}_j = (\cos\theta_j, \sin\theta_j).
\]
We use the convention that the indices are considered modulo $\ell$, i.e., we identify $\theta_{\ell+j}=\theta_j$.
% 原稿 3-4
 We conclude \eqref{CrG1} and \eqref{Tr} to get a system of $\ell$ ODEs for $L_j$'s.
 The initial value problem of this system is locally-in-time solvable for example when $g$ is $C^1$ in the second variable.
 The resulting admissible polygonal flow is called a \emph{crystalline flow}.
 This idea is introduced by J.\ Taylor \cite{T1} for $V=\sigma\kappa_\sigma$ and S.\ B.\ Angenent and M.\ E.\ Gurtin \cite{AG} for $V=M(\nu)(\kappa_\sigma+C)$; both examples are introduced in Section \ref{SM}. \\
\textbf{Starting from a general polygon.}
 If one considers a polygon whose orientation belongs to $\mathcal{N}$ but violates the adjacence condition, it is expected that new facets with ``missing directions" are created from a corner.
 To be more precise, let us consider the equation
\[
	V = \kappa_\sigma.
\]
We consider adjacent facets $S_A$, $S_B$ of a polygon $\Gamma$ whose angles $\theta_A$, $\theta_B$ of orientation $\mathbf{n}_A$, $\mathbf{n}_B$ are not adjacent; see Figure~\ref{fig:page2}.
\begin{figure}
\centering
\includegraphics[width=3in]{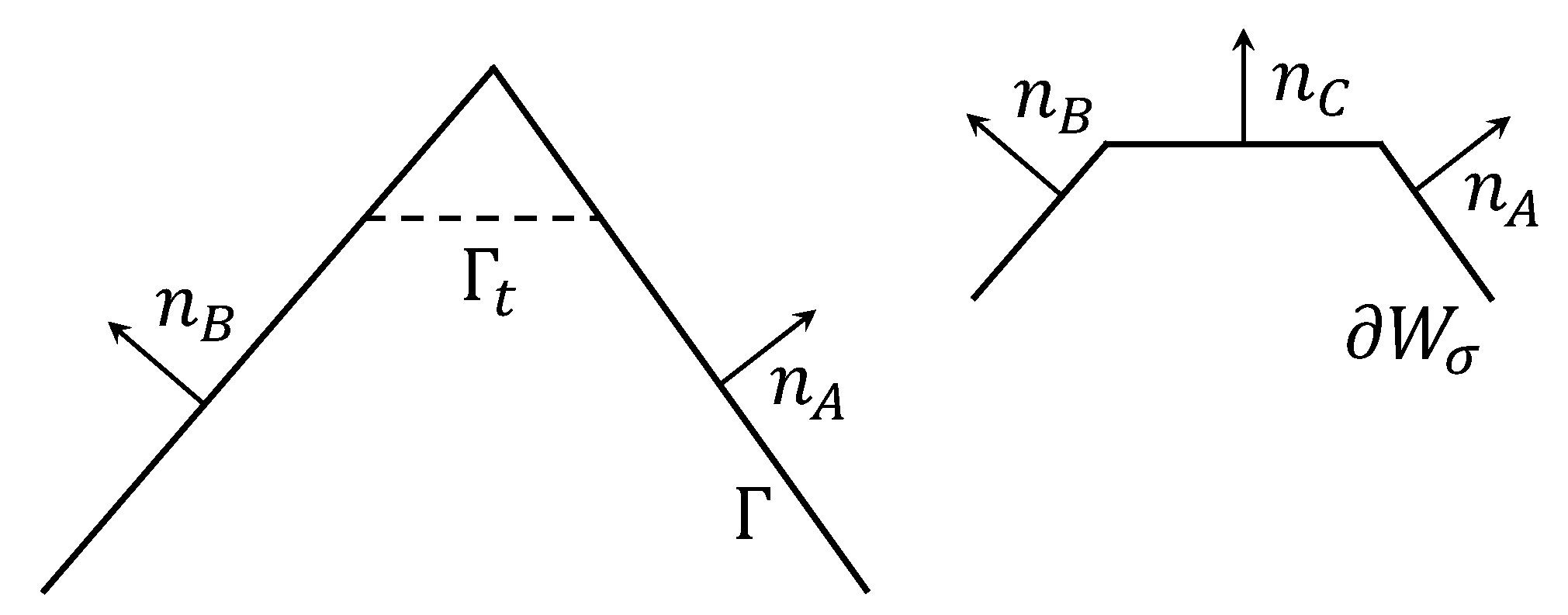}
\caption{}
\label{fig:page2}
\end{figure}
 In other words, there are missing admissible angles between $\theta_A$ and $\theta_B$.
 If $S_A$ and $S_B$ do not move, i.e., they are stationary, it is expected that there is a unique self-similar expanding crystalline flow which converges to $S_A \cup S_B$ as the time tends to zero.
 The unique existence of such a self-similar expanding crystalline flow has been claimed in a pioneering work by J.\ Taylor \cite[Proposition 2.2 (1)]{T3}.
 However, unfortunately, the proof skips over important details.
% 原稿 3-5
 Here, by self-similar we mean that the flow $\Gamma_t$ is of the form $\Gamma_t=t^{1/2}\Gamma_*$ with some admissible $\Gamma_*$; 
 we here assume that the vertex connecting $S_A$ and $S_B$ is the origin by translation.
 Note that $\Gamma_*$ may not be a part of the Wulff shape as observed in a numerical calculation \cite{HGGD}.
 The unique existence of such a self-similar expanding solution is proved in D.\ Campbell \cite{Ca} in the case that $W_\sigma$ is a regular polygon and in \cite{GGH} for general $W_\sigma$.
 This problem is reduced to solving a system of algebraic equations and methods presented in \cite{Ca} and \cite{GGH} are quite different.
 Approximating by such a self-similar expanding solution, one is able to construct an expanding solution even if $S_A$ and $S_B$ are moving.
 This is carried out by Y.\ Ochiai \cite{O} for $V=\kappa_\sigma$ and is extended to an equation including $V=M(\nu)(\kappa_\sigma+C)$ by R.\ Kuroda \cite{K};
 see also \cite{GGKO} for a complete proof for more general equations.

% 原稿 3-6
Although there is a large number of articles studying crystalline flows, this type of facet creation problems are not discussed frequently.
 A facet creation problem was observed in \cite{GG1} and further developed in \cite{Mu}, \cite{MuR1}, \cite{MuR2} mostly for graph-like solutions.
 However, the number of newly created facets in one point is just one.
 This aspect is quite different from works by \cite{T3}, \cite{Ca}, \cite{GGH}, \cite{O}, \cite{K}, \cite{GGKO}, where several facets are created from one point (corner).  

% 原稿 3-7
On the other hand, it is not difficult to handle the case when the direction condition is violated.
 In this case, we just regard $\Delta(\mathbf{n})=0$ for such directions.
 Such a facet is preserved at least for a short time, so we may call such a polygonal flow satisfying ``adjacence condition'' a \emph{weakly admissible} polygonal flow \cite{GG96}. \\
\textbf{Behavior of convex crystalline flow.}
 If the initial polygon is a convex (admissible) polygon, the behavior of a solution (crystalline flow) has been well-studied for $V=M(\nu)\kappa_\sigma$.
 It is easy to see that the convexity is preserved.
\begin{enumerate}
\item[(i)]
The case when $M$ is parallel to $\sigma$, i.e., $M(\nu)=c\sigma(\nu)$ with some $c>0$.
 It is easy to see that there always exists a self-similar solution shrinking to a point whose profile is the Wulff shape $W_\sigma$.
 By a spatial translation, this solution can be written as $\Gamma_t=(2c)^{1/2}(T-t)^{1/2}\partial W_\sigma$, where $T$ is the extinction time. 
 The uniqueness of a self-similar solution is proved when the Wulff shape $W_\sigma$ is symmetric with respect to the origin and the number of its vertices is more than four in \cite{S1};
 in the case $W_\sigma$ is a parallelogram, all parallelograms shrink self-similarly; see the next section for an explicit solution.
 Moreover, it is shown in \cite{S1} that all convex solutions shrink asymptotically similarly to the self-similar solution.
 These results are parallel to those for conventional curve shortening flow as established in \cite{Ga93}, \cite{GaL94}, \cite{DGM}, \cite{DG}.  
% 原稿 3-8
\item[(ii)] The case where $M$ is unrelated to $\sigma$.
 In this case, the situation is complicated as discussed in \cite{S2}, \cite{A}.
 In \cite{A} a rather complete picture is given.
 We first consider the case of orientation-free i.e., $M(\mathbf{n})\Delta(\mathbf{n})=M(-\mathbf{n})\Delta(-\mathbf{n})$ for $n\in\mathcal{N}$ where $\mathcal{N}=-\mathcal{N}$.
 In this case, there are two possibilities.
 Either phenomenon similar to (i) occurs or there is no self-similar shrinking solution and the isoperimetric ratio of a solution may tend infinity \cite{A}.
 Moreover, in the second case it is shown in \cite{A} that the minimal length of facets at time $t$ behaves like $\left\{(T-t)/\log(t-t)\right\}^{1/2}$ or $(T-t)^\beta$, $1/2<\beta<1$ as $t$ tends to $T$, where $T$ is the extinction time.
 For a self-similar solution, the length should behave like $(T-t)^{1/2}$ so it is shorter than that of a self-similar solution.
 This has a strong contrast compared to the conventional orientation-free anisotropic curvature flow, where all flows shrink in a self-similar way.
 This indicates that a qualitative property of a solution may differ from the conventional curve shorting equation depending upon the Wulff shape.  
% 原稿 3-9
 If the motion is not orientation-free, it is shown in \cite{IUYY} that a crystalline flow may not become convex.
 There also exists a non-convex self-similar shrinking solution when the Wulff shape is a square or a regular triangle for $V=M\kappa_\sigma$ with $M$ unrelated to $\sigma$ which is not orientation-free \cite{IUYY}.
\end{enumerate}

We next consider the equation $V=|\kappa_\sigma|^{\alpha-1}\kappa_\sigma$ for $\alpha>0$.
 The situation depends on the value of $\alpha$.
 We have discussed the case $\alpha=1$.
 In the case $\alpha\geq 1$, it is shown in \cite{GG3} that there is no degenerate pinching at the extinction time $T$.
 By degenerate pinching we mean that two parallel facets touch with positive length at the extinction time.
 For $\alpha>1$, all (convex) solutions shrink to a point in a self-similar way like (i) \cite{A}.
 If $\alpha<1$, a degenerate pinching may happen and there is a solution whose enclosed area tends to zero but the limit of the length remains positive \cite{A}.
 For $\alpha<1$, there also exists a non-convex self-similar solution for $V=M|\kappa_\sigma|^{\alpha-1}\kappa_\sigma$ even if the equation is orientation-free \cite{IUYY}.

We now consider the case when the initial polygon does not fulfill the direction condition but satisfies the adjacency condition with interpretation that $\Delta(\mathbf{n}_i)=0$ for a non-admissible direction of the initial polygon and that $\mathbf{n}_i$ belongs to $\mathcal{N}$.
 In \cite{Ya} a quite general results are established.
% 原稿 3-10
 The equation considered there is $V=g(\nu,\kappa_\sigma)$ with $g(\nu,0)=0$ which is non-decreasing and locally Lipschitz in the second variable.
 By solving the system of ODEs, we see that the number of facets is unchanged during a short time.
 At some time either at least one of the facets with a non-admissible direction disappears or the whole evolution shrinks to a point \cite{Ya}. \\
\textbf{Behavior of a general admissible polygon.}
 If the initial polygon is admissible but not convex, it must have an inflection facet, i.e., a facet with $\chi=0$.
 There is a crystalline flow with such initial data until the length of some facet tends to zero.
 It is already proved in \cite[Theorem 3.2]{T3} that for the equation $V=\sigma\kappa_\sigma$, at such occasion only at most two adjacent inflection facets disappear unless the flow shrinks to a point.
 However, the proof there is rather sketchy.
 In \cite{IS} a full proof is given when $W_\sigma$ is a regular polygon with even number of facets.
 The resulting polygon at the time when infection facets disappear stays admissible, so one can extend a solution as a crystalline flow until it loses another facet.
 We are able to complete this procedure until it shrinks to a point.
 Such an extended flow is called an extended crystalline flow.

For the curve shortening equation $V=\kappa$, it is shown that the solution (flow) becomes convex in finite time \cite{Gr}.
 It seems that the corresponding result is not established even when $W_\sigma$ is symmetric with respect to the origin and the equation is $V=\sigma\kappa_\sigma$.
% 原稿 3-11
 To the best of our knowledge, the (extended) crystalline flow (after losing several inflection facets) becomes almost convex in the sense that all facets have positive crystalline curvature possibly except two adjacent inflection facets for $V=\sigma\kappa_\sigma$ with symmetric $W_\sigma$ as shown in \cite{I08}. \\
\textbf{Equations with a driving force term.}
 We next consider the equation $V=M(\nu)(\kappa_\sigma+C)$, where $C$ is a constant.
 This equation is sometimes called the eikonal curvature flow equation.
 There are several new phenomena in this equation compared to the case $C=0$.
 For example, this motion certainly depends on the orientation.
 If $C$ is taken positive and $\nu$ is taken outwards, it grows to the whole plane in infinite time provided that the polygon is sufficiently large.
 The large time behavior of an (extended) crystalline flow is studied in \cite{GG13} with special emphasis on the anisotropic effect of mobility $M$ and $\sigma$.
 For $V$-shaped initial data, its evolution was studied in \cite{I11a}, \cite{I11b}.
 A crystalline flow is also applied to the study of a growth of spirals since the work of \cite{I14}, which is further developed in \cite{IO1}.
 Various methods for the numerical computation of the crystalline flow are compared in \cite{IO2}. \\
% 原稿 3-12
\textbf{A few remarks on consistency and stability.}
 If the initial data is given as the graph of a periodic piecewise linear function, for $V=M(\nu)\kappa_\sigma$ the theory of maximal monotone operators applies to construct a solution \cite{FG}.
 This notion of a solution is consistent with the (extended) crystalline flow;
 see also \cite{EGS} where a numerical scheme based on a variational inequality is given.

Note that the crystalline flow satisfies a comparison principle or an order preserving property.
 It reads that if an admissible polygon $\Gamma^a$ is enclosed by another admissible polygon $\Gamma^b$, then the corresponding crystalline flows $\{\Gamma^a_t\}$ and $\{\Gamma^b_t\}$ starting from $\Gamma^a$ and $\Gamma^b$, respectively, have the same property, i.e., $\{\Gamma^b_t\}$ encloses $\{\Gamma^a_t\}$ as far as both exist;
 see \cite{T3} and \cite{GGu}.
 This is easily seen by comparing their crystalline curvatures.
 Based on this property, one is able to establish a notion of viscosity solutions.
 This was first introduced in the case where $\Gamma_t$ is given as the graph of a periodic function \cite{GG1};
 see also \cite{GG} and its consistency with an (extended) crystalline flow already discussed in \cite{GG96}.
 Moreover, their solution can be obtained as a limit of a smoother problem, i.e., the problems where $\sigma$ is smooth \cite{GG2}.
% 原稿 3-13
 This stability property applies for a variational solution \cite{FG}.
 Note also that in both frameworks solutions with a smooth $\sigma$ can be approximated by a crystalline flow \cite{FG}, \cite{GG2}.
 This gives a numerical algorithm to solve a smooth anisotropic curvature flow or even the heat equation by approximating it by crystalline flows.
 This topic is studied in \cite{FG}, \cite{GirK}, \cite{GG2} for a graph-like solution.
 In \cite{GirK} a convergence rate is also given.
 The approach by viscosity solution is extended to closed curves through a level-set method \cite{GG4} and its consistency is discussed in \cite{GG3}.
 The stability is also discussed in \cite{GG4}.
 Among other results, a solution with a smooth $\sigma$ can be approximated by extended crystalline flows.
 It is proved for $V=\kappa$ in \cite{Gir} for convex curves with convergence rate and in \cite{IS} for a general curve.
 In \cite{GG4} such stability is discussed for a general equation $V=g(\nu,\kappa_\sigma)$.
 Note that it is also shown in \cite{GG4} that an extended crystalline flow is a limit of flows of problems with smooth anisotropy.
 More precisely, if $W_\sigma$ is close in the sense of the Hausdorff distance, the solution must be close (up to fattening).

When one discusses consistency for equations with driving force term like the eikonal-curvature flow $V=M(\nu)(\kappa_\sigma+C)$, one should be careful to handle corners.
 If we consider just the eikonal equation $V=C>0$ for a bounded convex polygon, it is expected that the solution will be rounded following the Huygens principle.
 To preserve corners, one has to restrict the mobility $M(\nu)$.
% 原稿 3-14
 We consider a general equation $V=g(\nu,\kappa_\sigma)$.
 Let us explain the corner preserving condition explicitly stated in \cite[Lemma 4.1, Lemma 4.2]{GG13}.
 We say that $g$ satisfies the \emph{corner preserving condition} if for each $\mathbf{n}_k\in\mathcal{N}$
\[
	g(\mathbf{m},0) = \frac{1}{\sin\varphi_{k+1}}
	\left( g(\mathbf{n}_k,0) \sin\psi_{k+1} + g(\mathbf{n}_{k+1}) \sin\psi_k \right)
\]
for all $\mathbf{m}=(\cos\theta,\sin\theta)$ with $\theta_k<\theta<\theta_{k+1}$, where $\varphi_{k+1}=\theta_{k+1}-\theta_k$ and $\psi_k$ (resp.\ $\psi_{k+1}$) is the angle between $\mathbf{n}_k$ ($\mathbf{n}_{k+1}$) and $\mathbf{m}$ so that $\varphi_{k+1}=\psi_{k+1}+\psi_k$.
 Geometrically speaking, this condition can be written as
\[
	A_k \subset \left\{ x \in \mathbb{R}^2 \mid
	x\cdot\mathbf{m} \leq g(\mathbf{m},0),\ \mathbf{m}=(\cos\theta,\sin\theta),\
	\theta_k<\theta<\theta_{k+1} \right\} \subset B_k
\]
with
\[
	A_k = H_k \cap H_{k+1}, \ 
	B_k = H_k \cup H_{k+1}, \ 
	H_{k+j} = \left\{ x \in \mathbb{R}^2 \mid
	x\cdot\mathbf{n}_{k+j} \leq g(\mathbf{n}_{k+j},0) \right\}.
\]
If $\Gamma$ is convex with outward orientation, we only need the inclusion of $A_k$. In other words, in the above identity the equality should be replaced by $\geq$ so that $g(\mathbf{m},0)$ is always larger than the right-hand side.
This condition says that in the corner all segments whose orientation is between that of facets forming the corner move faster than corner facets for $V=g(\nu,0)$.
 This condition is first pointed out explicitly by \cite{GHK} and independently by \cite{GSS}.
 It is stated in a different from in \cite{GG96}.
 The geometric version is found in \cite{GG3};
 however, unfortunately, the definition of $B_k$ was mistyped.

% 原稿 3-15
We shall postpone the definition of viscosity solutions to Section~\ref{AV}.
 We note that the theory covers a wide range of $\sigma$ not necessarily purely crystalline in planar case for general equation $V=g(\nu, \kappa_\sigma)$ including \eqref{2EXP}, \eqref{2POW}, while in higher dimension, for such setting it is limited for purely crystalline $\sigma$ for general equations;
 see Section \ref{AV}.
 In \cite{GG4} it is only assumed that $F_\sigma$ is convex, $C^2$ except finitely many vertices and the curvature is bounded.
 
Although the approach by admissible polygonal flow is convenient to study planar curvature flow equations, it is limited because it implicitly requires that the speed of a facet is spatially constant.
 For example, even in $\mathbb{R}^2$ if one considers the equation with spatially inhomogeneous driving force like
\[
	V = M(\nu) \left(\kappa_\sigma + f(x) \right),
\]
then it is not appropriate to assign the speed of a facet as a spatially constant to obtain a comparison principle.
For a graph-like solution with special $M$, as a variational solution several facet-breaking solutions are given in \cite{GG98}.

\section{Explicit solutions} \label{ES}

In this section we given examples of a few interesting explicit solutions to illustrate the behavior of the equations.

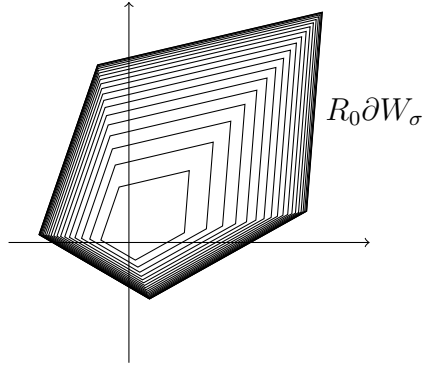
\begin{figure}
\centering
\begin{tikzpicture}[scale=0.8]
\foreach \t in {0, 0.05, ..., 1} {
\begin{scope}[scale=sqrt(1 - \t * \t)]
\draw (50:5) -- (100:3) -- (175:1.5) -- (290:1) -- (10:3) -- cycle;
\end{scope}
}
\draw[->] (-2,0)--(4,0);
\draw[->] (0,-2)--(0,4);
\path (10:3) -- node[right] {$R_0 \partial W_\sigma$} (50:5);
\end{tikzpicture}
\caption{Timesteps $\Gamma_{it^*/20}$, $i = 0, 1, \ldots$, of a homethetic Wulff shape solution of $V = \sigma \kappa_\sigma$. Note that even though $\kappa_\sigma$ is a constant on $\Gamma_t$, the edges further from the origin move faster due to the  mobility factor $\sigma$ and the solution is homothetic.}
\label{fig:wulff-shrink}
\end{figure}
The simplest solution of the crystalline mean curvature flow is the homothetic (self-similar) solution starting from the Wulff shape, Fig.~\ref{fig:wulff-shrink}, that can be translated and scaled. Rotations are of course not allowed. As noted in the introduction, the crystalline mean curvature on the surface of the Wulff shape $W_\sigma$ is the constant $n-1$; here, the orientation (normal) is taken inward. Therefore
\begin{align*}
\Omega_t = \sqrt{R_0^2 - 2(n-1)t} \ W_\sigma
\end{align*}
is a solution of the crystalline mean curvature flow $V = \sigma(\nu) \kappa_\sigma$ for any $R_0 > 0$ on the interval $t \in [0, t^*)$, where $t^* = \frac{R_0^2}{2(n-1)}$ is the \emph{extinction time}. Note the factor $\sigma(\nu)$ in the velocity law. The (inner) normal velocity of $R(t) W_\sigma$ at a boundary point $x$ with inner unit normal $\nu$ is $-R'(t) x \cdot \nu = -R'(t) \sigma(\nu) R(t)$.

 One might ask whether the above solutions are the only homothetic solutions of the flow. This is however not always the case as the following simple example illustrates. We consider $n = 2$ and the anisotropy $\sigma(p) = |p_1| + |p_2| = \norm{p}_1$. Let $\Omega_0 = (-a, a) \times (-b, b)$ be a rectangle for some $a > 0$, $b > 0$. Then $\Omega_t = R(t) \Omega_0$ for $R(t) = \sqrt{1 - \frac{2}{ab} t}$ is a solution of both $V = \sigma \kappa_\sigma$ and $V = \kappa_\sigma$. In $n = 2$, the uniqueness of the Wulff shape homothetic solution was proved by Stancu \cite{S1} for even anisotropies $\sigma$ when (so that the problem is orientation-free) the Wulff shape $W_\sigma$ is not a quadrilateral as mentioned in Section~\ref{PF} (i). 

A related question is whether a solution starting from an arbitrary convex initial data will asymptotically approach the homothetic Wulff shape solution as in the case of the usual mean curvature flow.
 As mentioned in Section \ref{PF} (i), this was shown again by Stancu \cite{S2} in $n = 2$ for even non-quadrilateral anisotropies. The situation is much more complex in $n > 2$ and is studied in \cite{NP_MMMAS}.

By an interpretation different from Section \ref{PF} (ii), we also mention that for $V=\sigma\kappa_\sigma$ examples of non-convex homothetic solutions in $n = 2$ given in \cite{IUYY} for non-even anisotropies $\sigma$, that is, $\sigma(p) \neq \sigma(-p)$ for some $p$. This shows that one cannot in general expect that a non-convex connected initial shape will become convex before extinction time.

In dimensions $n > 2$ the situation is more complex and nonzero genus explicit homothetic solutions are known. For example, for $\sigma(p) = \norm{p}_1$ a cube with a square-shaped hole along each axis is a homothetic solution, Figure~\ref{fig:sponge}. See \cite{Po} for more details.
\begin{figure}
\centering
\includegraphics[width=2in]{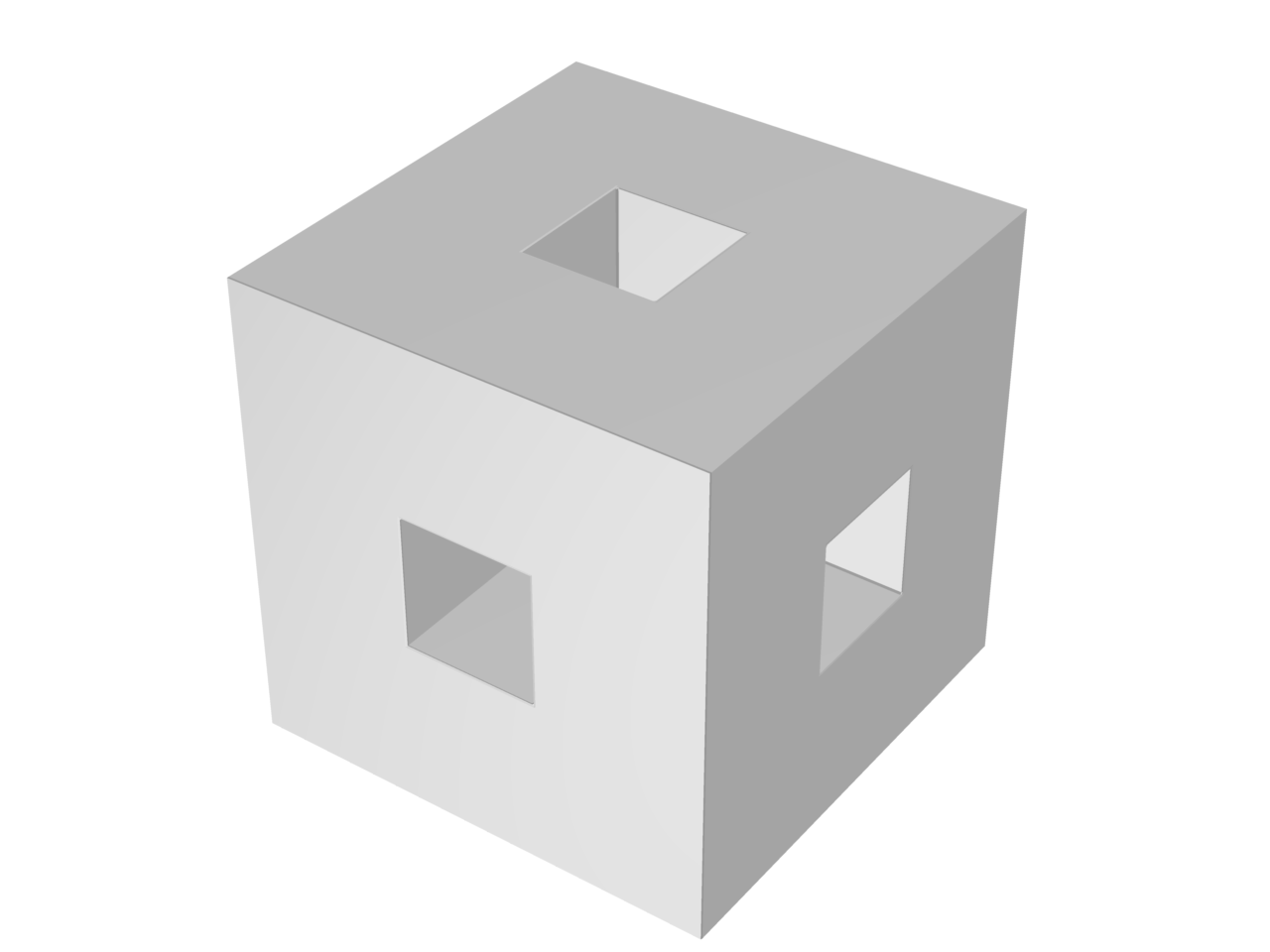}
\caption{Sponge-like homothetic solution of the crystalline mean curvature flow in dimension 3 with $\sigma(p) = \|p\|_1$.}
\label{fig:sponge}
\end{figure}

For examples of solutions of the related total variation flow see for example \cite[Sec.~5]{Moll05}.

\section{Approach by the theory of maximal monotone operators}
\label{AO}
\label{se:max-monotone}

%%%%%%%%%% Subsection 5.1
\subsection{Abstract theory} \label{AO1}

In this section we introduce the crystalline mean curvature as a solution of a certain minimization problem. This interpretation is based on the theory of maximal monotone operators of K\=omura~\cite{Ko} and Brezis~\cite{Br73}. 

Let us give a motivation for this point of view.
It is natural to expect that the crystalline mean curvature flow with anisotropy $\sigma$ can be approximated by anisotropic mean curvature flow with smooth anisotropies $\sigma_m$ so that $\sigma_m \to \sigma$ in some sense.

If $\Gamma$ is a $C^2$ surface in $\Rn$ and $\sigma_m$ is a $C^2$ smooth anisotropy, the anisotropic mean curvature $\Gamma$ at $x \in \Gamma$ is given as $\divo \nabla \sigma_m(\nabla u)(x) = \trace [\nabla^2 \sigma_m(\nabla u(x)) \nabla^2 u(x)]$, where $u$ is any $C^2$ level set function of $\Gamma$ in the neighborhood of $x$ with $\nabla u(x) \neq 0$. 

Consider now the anisotropic total variation flow
\begin{align*}
u_t - \divo \nabla \sigma_m(\nabla u) = 0
\end{align*}
on $L^2(\Tn)$, $\Tn = \Rn / \Z^n$, or more rigorously,
\begin{align}
\label{atf}
u_t \in -\partial \mathcal E_m(u),
\end{align}
where 
\begin{align}
\label{tv-energy}
\mathcal E_m(v) := 
\begin{cases}
\int_\Tn \sigma_m(\nabla v) \;dx & v \in BV(\Tn) \cap L^2(\Tn),\\
+\infty & \text{otherwise},
\end{cases}
\end{align}
is the anisotropic total variation energy.

Since $\mathcal E_m$ is a convex lower semicontinuous functional on the Hilbert space $L^2(\Tn)$ with a dense domain, the operator $\partial \mathcal E_m$ is a maximal monotone operator and the anisotropic total variation flow \eqref{atf} has a unique solution for any initial data in $L^2(\Tn)$.

Suppose now that $(\sigma_m)_{m\geq 1}$ is a sequence of $C^2$ anisotropies that monotonically converge to the crystalline anisotropy $\sigma$. Then it is known \cite{Attouch} that $\mathcal E_m \to \mathcal E$ in the sense of Mosco convergence, see \eqref{mosco}, where
\begin{align*}
\mathcal E(v) :=
\begin{cases}
\int_\Tn \sigma(\nabla v) \;dx & v \in BV(\Tn) \cap L^2(\Tn),\\
+\infty & \text{otherwise}.
\end{cases}
\end{align*}
But Mosco convergence implies the convergence of the nonlinear semigroups \cite[Theorem~3.26]{Attouch} and \cite[Theorem~3.2]{BrezisPazy70}: the solutions of \eqref{atf} locally uniformly converge to the unique solution $u: [0, \infty) \to L^2(\Tn)$ of
\begin{align}
\label{crystallinetf}
\left\{
\begin{aligned}
\frac{du}{dt} &\in -\partial \mathcal E(u(t)),\qquad t > 0,\\
u(0) &= u_0
\end{aligned}
\right.
\end{align}
for any initial data $u_0 \in L^2(\Tn)$.

As we will see below, $\partial \mathcal E(v)$ is in general multivalued even if $\nabla v \neq 0$ for typical crystalline mean curvature evolutions. Nevertheless, the unique solution of \eqref{crystallinetf} is right-differentiable at every $t > 0$, $\partial E(u(t)) \neq \emptyset$ and $d^+u/dt(t) = - \partial^0 \mathcal E(u(t))$ for $t > 0$, where $\partial^0 \mathcal E(v)$ is the \emph{canonical restriction} or \emph{minimal section} of the subdifferential $\partial \mathcal E(v)$, i.e., the unique element of $\partial \mathcal E(v) \subset L^2(\Tn)$ with the minimal norm.

This strongly suggests that we should use $-\partial^0 \mathcal E(v)$ as the definition of the crystalline mean curvature to hope to obtain stability under approximation by anisotropic mean curvature flow.

Fortunately the characterization of $\partial \mathcal E$ is well understood even for rather general $\sigma = \sigma(x, p)$, see \cite{Moll05} for example.

We include the proof here for completeness for $\sigma= \sigma(p)$ and space $L^2(\Tn)$. We need to introduce a number of definitions. 

The functional $\E(u)$ can be defined in two equivalent ways. The first one is a generalization of the definition of the total variation,
\begin{align*}
\E(u) := \sup \set{-\int u\divo z \Bigm| \sigma^\circ(z) \leq 1,\ z \in C^1(\Tn, \Rn)}, \qquad u \in L^2(\Tn).
\end{align*}
Note the minus sign since $\sigma$ is not assumed even.
The function $\sigma^\circ$ is the support function of the Frank diagram $F_\sigma=\{\sigma \leq 1\}$, i.e.,
\[
	\sigma^\circ(x) = \sup \left\{ x \cdot p \mid p \in F_\sigma \right\}
\]
so that $W_\sigma=\{\sigma^\circ \leq 1\}$.
This $\mathcal{E}$ is clearly a convex, positively one-homogeneous, lower semi-continuous functional on $L^2(\Tn)$.
It is known \cite{AmarBellettini} that it is the \emph{relaxation} (\emph{closure} or lower semicontinuous envelope) of the functional
\begin{align*}
J(u) := 
\begin{cases}
\int \sigma(\nabla u), & u \in W^{1,1}(\Tn) \cap L^2(\Tn),\\
+\infty, &\text{otherwise}.
\end{cases}
\end{align*}
In fact $\E(u) = J(u)$ for any $u \in W^{1,1}(\Tn) \cap L^2(\Tn)$.

To characterize the subdifferential, we here present a simplification of the proof in \cite{Moll05}, which itself is based on the unpublished note of F.~Alter; see also \cite{ACM} for the proof when $\sigma(p) = |p|$ and more details. The idea is based on the characterization of the subdifferential using the polar of $\E$ defined as
\begin{align*}
\E^\circ(v) := \sup\set{(u, v) \mid u \in H, \ \E(u) \leq 1} = \sup \set{\frac{(u,v)}{\E(u)} \Big| u \in H},
\end{align*}
where we set $H := L^2(\Tn)$ the Hibert space with the $L^2$-inner product $(u, v) = \int u v\;dx$. In the formula we use the convention $\frac00 = 0$, $\frac{a}0 = +\infty$ for any $a > 0$ and $\frac{a}{+\infty} = 0$ for any $a \in \R$. Since $\E$ is positively one-homogeneous, convex and lower semicontinuous, we have the following standard characterization \cite[Lemma~1.7]{ACM}:
\begin{align}
\label{subdiff-polar-char}
v \in \partial \E(u) \qquad \Leftrightarrow \qquad \E^\circ(v) \leq 1 \text{ and } (v, u) = \E(u).
\end{align}
We will show that $\E^\circ$ coincides with the functional
\begin{align*}
\Psi(v) := \inf\set{\norm{\sigma^\circ(z)}_\infty \mid v = -\divo z,\ z \in L^\infty(\Tn)}, \qquad v \in H = L^2(\Tn).
\end{align*}
The equality $v = -\divo z$ is understood in the sense of distributions: the function $-v$ is the distributional divergence of $z$.
$\Psi$ is again a convex, positively one-homogeneous, lower semicontinuous functional. For any such functional we have $(\Psi^\circ)^\circ = \Psi$ \cite[Proposition~1.6]{ACM}. Moreover, if $\Psi(v) < \infty$ the infimum is attained by a vector field and hence it is a minimum.

\begin{theorem}
\label{th:polar-characterization}
The equality $\E^\circ \equiv \Psi$ holds.
\end{theorem}

\begin{proof}
\fbox{$\leq$}: Take $v \in H$ with $\Psi(v) < \infty$ and fix $z \in L^\infty(\Tn)$ with $\divo z = - v$. Since $\E$ is the lower semicontinuous envelope of $J$, for any $u \in H$ with $\E(u) \leq 1$ there is a sequence $(u_k)_k \subset W^{1,1}(\Tn) \cap H$ with $u_k \to u$ in $H$ and $J(u_k) = \E(u_k) \to \E(u)$. We have
\begin{align*}
(u_k, v) &= \int z \cdot \nabla u_k \leq \int \sigma^\circ(z) \sigma(\nabla u_k) \\&\leq \norm{\sigma^\circ(z)}_\infty \int \sigma(\nabla u_k)
= \norm{\sigma^\circ(z)}_\infty \E(u_k).
\end{align*}
In the limit $k \to \infty$ we obtain
\begin{align*}
(u, v) \leq \norm{\sigma^\circ(z)}_\infty\qquad \text{for all } u \in H \text{ with } \E(u) \leq 1.
\end{align*}
Thus by definition of $\Psi$ we deduce
$\E^\circ(v) \leq \Psi(v)$.

\fbox{$\geq$}: Fix $u \in H$. By definition we have
\begin{align*}
\E(u) &= \sup \set{-\int u\divo z \Bigm| \sigma^\circ(z) \leq 1,\ z \in C^1(\Tn, \Rn)}\\
&= \sup_{z \in C^1} \frac{(u, -\divo z)}{\norm{\sigma^\circ(z)}_\infty} \leq \sup_{z \in C^1} \frac{(u, -\divo z)}{\Psi(-\divo z)}\\
&\leq \Psi^\circ(u),
\end{align*}
where we again use $0/0 := 0$. We deduce $\E^\circ \geq (\Psi^\circ)^\circ = \Psi$.
\end{proof}

We now have the following characterization of the subdifferential for Lipschitz functions.

\begin{corollary}
\label{co:lip-subdiff-char}
Let $u \in Lip(\Tn)$ and $v \in L^2(\Tn)$. The following are equivalent:
\begin{itemize}
\item $v \in \partial \E(u)$
\item there exists $z \in L^\infty(\Tn)$ with $v = - \divo z$ such that $z \in \partial \sigma(\nabla u)$ a.e.
\end{itemize}
\end{corollary}

\begin{proof}
\fbox{$\Rightarrow$}: $v \in \partial \E(u)$ implies that $\E^\circ(v) \leq 1$ and $\E(u) = (u, v)$. In particular there exists a vector field $z \in L^\infty(\Tn)$ with $v = -\divo z$ and $\norm{\sigma^\circ(z)} = \E^\circ(v) \leq 1$.
We have
\begin{align*}
\nabla u \cdot z \leq \sigma(\nabla u) \sigma^\circ(z) \leq \sigma(\nabla u) \qquad a.e.
\end{align*}
However, $\E(u) = (u, v)$ and therefore
\begin{align*}
\int \sigma(\nabla u) = \int u v = \int \nabla u \cdot z,
\end{align*}
and so we can deduce that $\nabla u \cdot z = \sigma(\nabla u)$ a.e., which with $\sigma^\circ(z) \leq 1$ a.e. implies $z \in \partial \sigma(\nabla u)$ a.e.

\fbox{$\Leftarrow$}: The opposite implication can be proved by reversing the above steps.
\end{proof}

The vector fields $z$ play a central role and we define
\begin{align*}
X^2(U) := \set{z \in L^\infty(U) \mid \divo z \in L^2(U)},
\end{align*}
for $U \subset \Rn$ open or $U = \Tn$, following \cite{Anzellotti}.
The vector fields that characterize the subdifferential are often called \emph{Cahn-Hoffman vector fields} and we define
\begin{align}
\label{CH}
CH(u; U) := \set{z \in X^2(U) \mid z \in \partial \sigma(\nabla u) \text{ a.e.}}
\end{align}
for any $u \in Lip(U)$. Note that if $U = \Tn$, by Corollary~\ref{co:lip-subdiff-char}
\begin{align*}
-\partial \E(u) = \divo CH(u; \Tn) := \set{\divo z \mid z \in CH(u; \Tn)}.
\end{align*}
Recall that this is a closed convex set, but it might be empty.

Since the set $\partial \E(u)$ is in general not a singleton, we need to determine how to select a value that gives a reasonable generalization of the anisotropic mean curvature to the crystalline case. The theory of maximal monotone operators suggests that we should choose the unique element of $-\partial \E(u)$ with the smallest $L^2$-norm. We will denote this element $-\partial^0 \E(u)$ if $\partial \E(u) \neq \emptyset$, since it is the projection of the origin $0$ on the convex closed set $-\partial \E(u)$.

\begin{example}
Suppose that $\sigma \in C^2(\Rn \setminus \set0)$ and $u \in C^2(\Tn)$. Let $x \in \Tn$ with $\nabla u(x) \neq 0$. Then $\partial \sigma(\nabla u) = \set{\nabla \sigma(\nabla u)}$ in the neighborhood of $x$ and therefore if $z \in CH(u; \Tn)$ we necessarily have $\divo z(x) = \divo \nabla \sigma(\nabla u)(x)$.
\end{example}

As was shown in the introduction, the element $-\partial^0 \E(u)$ is a solution of a minimization problem with an $n$-dimensional obstacle $z \in \partial \sigma(\nabla u)$. The value of the minimizer $\divo z_{\rm min}$ can depend nonlocally on $u$ whenever $\partial \sigma(\nabla u)$ is not a singleton, as is illustrated in the introduction. However this nonlocality is restricted to ``flat'' parts of $u$. Those correspond to facets and edges of the evolving crystal. The following technical ``patching'' lemma was proved in \cite[Lemma~2.8]{GP1}.
 Let $\mathbf{1}_E$ denote the characteristic function of $E$, i.e., $\mathbf{1}_E(x)=1$ for $x\in E$ and $\mathbf{1}_E(x)=0$ for $x\notin E$.

\begin{lemma} % Lemma 5.4
\label{le:patching}
Let $\sigma: \Rn \to \R$ be a positively one-homogeneous convex function. Let $U_1$, $U_2$ be two open subsets of $\Rn$ and $\psi_i \in Lip(U_i)$ two Lipschitz functions. Let $\delta > 0$ and set $G := \set{x \in U_1 \mid |\psi_1(x)| < \delta}$. Suppose that $\cl G \subset U_1 \cap U_2$ and $\psi_1 = \psi_2$ on $G$. If $z_i \in CH(\psi_i; U_i)$ are two Cahn--Hoffman vector fields, then
\begin{align*}
z := z_1 \one_{U_1 \setminus G} + z_2 \one_G
\end{align*}
is also a Cahn--Hoffman vector field $z \in CH(\psi_1; U_1)$, and
\begin{align*}
\divo z = \divo z_1 \one_{U_1 \setminus G} + \divo z_2 \one_G.
\end{align*}
\end{lemma}

We add the following simple observation that follows from $|\set{0 < |\psi_1| < \delta}| \to 0$ as $\delta \to 0$. Note that we still need $\psi_1 = \psi_2$ on a neighborhood of $\set{\psi_1 = 0}$.

\begin{corollary}
\label{co:patching-zero}
Under the assumptions of Lemma~\ref{le:patching},
\begin{align*}
z := z_1 \one_{U_1 \setminus \set{\psi_1 = 0}} + z_2 \one_{\set{\psi_1 = 0}}
\end{align*}
is also a Cahn--Hoffman vector field $z \in CH(\psi_1; U_1)$, and
\begin{align*}
\divo z = \divo z_1 \one_{U_1 \setminus \set{\psi_1 = 0}} + \divo z_2 \one_{\set{\psi_1 = 0}}.
\end{align*}

\end{corollary}

The above lemma shows that we can isolate $\divo z_{\rm min}$ on a neighborhood of $\set{\psi = 0}$. This is necessary to have some locality of the crystalline mean curvature which allows us to localize the construction of test functions to a given facet.

We conclude this section by an important way of approximating the values $\partial^0 \mathcal E(\psi)$.
Let us now fix the domain $\Tn$ for simplicity. For given $\psi \in L^2(\Tn)$ and $a > 0$, we consider the \emph{resolvent problem}
\begin{align}
\label{resolvent-problem}
v + a \partial \mathcal E(v) \ni \psi
\end{align}
for unknown $v \in L^2(\Tn)$. This can be viewed as the implicit Euler discretization of the gradient flow \eqref{crystallinetf}. It is also the Euler--Lagrange equation of the minimization problem
\begin{align*}
\argmin_{v \in L^2(\Tn)} \frac{\norm{v - \psi}^2_{L^2(\Tn)}}a + \mathcal E(v),
\end{align*}
which appears in an important discrete approximation of the crystalline mean curvature flow, Chambolle's scheme discussed in Section~\ref{AV4}.

We have the following standard existence and approximation result that is valid for any convex proper lower semi-continuous functional like $\E$, see for example \cite{Attouch}.
\begin{proposition}
\label{pr:resolvent-approximation}
For every $\psi \in L^2(\Tn)$ and $a > 0$ the resolvent problem \eqref{resolvent-problem} has a unique solution $\psi_a \in L^2(\Tn)$ and $\psi_a \to \psi$ as $L^2(\Tn)$.

If furthermore $\partial \mathcal E(\psi) \neq \emptyset$, then
\begin{align*}
\frac{\psi_a - \psi}a \to - \partial^0 \mathcal E(\psi) \qquad \text{in $L^2(\Tn)$ as $a \to 0$.}
\end{align*}
\end{proposition}

The solutions also satisfy a comparison principle, see \cite{CC} for a proof.

\begin{proposition}
\label{pr:resolvent-comp-principle}
If $\psi^1, \psi^2 \in L^2(\Tn)$ are two right-hand sides with $\psi^1 \leq \psi^2$ and $a > 0$, we have $\psi_a^1 \leq \psi_a^2$ where $\psi^1_a$ and $\psi^2_a$ are the respective solutions of \eqref{resolvent-problem}.
\end{proposition}

%%%%%%%%%% Subsection 5.2
\subsection{Calibrability and Cheeger sets} \label{AO2}

As we already briefly mentioned in the introduction, the minimization problem one needs to solve to find the value $\partial^0 \mathcal E(\psi)$ for a given $\psi$ has interesting connections to the so-called Cheeger problem for sets. For a given open set $U \subset \Rn$, define the \emph{Cheeger constant} as
\begin{align*}
h(U) := \inf \set{\frac{P(F)}{\mathcal L^n(F)} : F \text{ Borel}\subset \Rn, \ \mathcal L^n(F) \in (0, \infty)},
\end{align*}
where $P(F) = \mathcal E(\one_F)$ is the anisotropic perimeter of $F$. Usually the isotropic $\sigma(\nu) = 1$ is considered, in which case this is just the usual perimeter equal to $\mathcal H^{n-1}(\partial F)$ for sufficiently regular sets. A set $F \subset U$ such that $\frac{P(F)}{\mathcal L^n(F)} = h(U)$ is called a Cheeger set of $U$. If $U$ itself is a Cheeger set of $U$, it is called \emph{self-Cheeger}. 
Finding the value $h(U)$ or characterizing the Cheeger subsets of $U$ is then often referred to as the Cheeger problem. For a recent review of this topic see \cite{Leonardi}.

In the current note, the question whether a given set $U$ is self-Cheeger is closely related to the questions whether the value of $\partial^0 \mathcal E(\psi)$ is constant on a facet $\cl{U}$ of $\psi$.
If $\partial^\circ \mathcal E(\psi)$ is constant on a given facet, the facet is called \emph{calibrable} or \emph{$\sigma$-calibrable}, see \cite{BNP01c}.

We point out that this notion of calibrability is slightly weaker than the notion used in the context of total variation flows \cite{ACC,Leonardi}. There an open bounded set $U$ is called calibrable if the total variation flow \eqref{crystallinetf} with initial data $\one_U$ has the unique solution $a(t) \one_U$ with $a(t) = \max(1 - \frac{P(U)}{\mathcal L^n(U)}t, 0)$. This therefore implies that $\partial^0 \mathcal E(\one_U)$ is constant on $U$ \emph{and} on $U^\compl$.

We use the former notion of calibrability. The following theorem in a more general setting (but still only in dimension $n = 2$), including non-uniform forcing, was proved in \cite[Th.~6.1]{BNP01c}. See also \cite{ABT} for further developments.

\begin{theorem}
\label{th:self-cheeger}
Let $n = 2$ and let $\sigma$ be an even anisotropy on $\R^2$, $\sigma(p) = \sigma(-p)$.
Suppose that $\psi \in Lip(\R^2)$ such that $CH(\psi; \R^2)$ is nonempty. Let $U$ be a bounded connected component of $\interior\set{\psi = 0}$.
The following are equivalent:
\begin{itemize}
\item[(i)] $U$ is calibrable ($\partial^0 \mathcal E(\psi)$ is constant on $U$)
\item[(ii)] for any $F \subset U$ of finite perimeter
\begin{align}
\label{cheeger-ratio-minimality}
\frac{SP(F)}{\mathcal L^2(F)} \geq \frac{SP(U)}{\mathcal L^2(U)}.
\end{align}
Here $SP(F)$ is the signed perimeter of $F$ defined using the reduced boundary $\partial^* F$ as
\begin{align*}
SP(F) = \int_{\partial^* F_+} \sigma(\nu) \;d\mathcal H^1 -\int_{\partial^* F_-} \sigma(\nu) \;d\mathcal H^1,
\end{align*}
with $\partial^* F_- := \set{x \in \partial^* F \cap \partial^* U : \nu_U(x) \cdot \nabla \psi(x) < 0}$ and $\partial^* F_+ := \partial^* F \setminus \partial^* F_-$.
\end{itemize}
\end{theorem}

The quantity $\frac{SP(U)}{\mathcal L^n(U)}$ is a generalization of the usual Cheeger ratio $\frac{P(U)}{\mathcal L^n(U)}$ to facets: sets with signed boundary determined by whether the surface at the boundary point is convex or concave in the normal direction of the facet; see also Section~\ref{AO3} for a notion of facet.

To illustrate proof of (i) $\Rightarrow$ (ii) in a simplified setting, consider now a Lipschitz function $\psi \in Lip(\Rn)$ whose $\interior \set{\psi = 0}$ is simply connected bounded open set $U \in \Rn$ with Lipschitz boundary. Let us also for simplicity assume that we can define $\nabla \psi \neq 0$ on $\partial U$ $\mathcal H^{n-1}$-a.e. as the limit of $\nabla \psi$ from $\cl{U}^\compl$. Suppose that there exists vector field $z \in L^\infty(\Rn) \cap C(\Rn)$ with $\divo z \in L^2(\Rn)$ and $z \in \partial \sigma(\nabla \psi)$ a.e. that is sufficiently regular and assume that $\divo z = \lambda$ on $U$ for some $\lambda$. Then the divergence theorem  yields
\begin{align*}
\lambda \mathcal L^n(U) = \int_U \divo z \dx = \int_{\partial U} z \cdot \nu \;\mathcal H^{n-1}.
\end{align*}
We observe that $\nu = \frac{\nabla \psi}{|\nabla \psi|}$ on $\partial U_+$ and $\nu = - \frac{\nabla \psi}{|\nabla \psi|}$ on $\partial U_-$. Since $z \in \partial \sigma(\nabla \psi)$, we have $z \cdot \nu = \pm \sigma(\nu)$ on $\partial U_\pm$. We have
\begin{align*}
\int_{\partial U} z \cdot \nu \;\mathcal H^{n-1}.
 = \int_{\partial U_+} \sigma(\nu) \;d\mathcal H^{n-1} - \int_{\partial U_-} \sigma(\nu) \;d\mathcal H^{n-1} = SP(U).
\end{align*}

In particular, $\lambda = \frac{SP(U)}{\mathcal L^n(U)}$. However, for any smooth subset $F$ of $U$ we have
\begin{align*}
\int_F \divo z \dx &= \int_{\partial F} z \cdot \nu \;d\mathcal H^{n-1} =
\int_{\partial F \setminus \partial U} z \cdot \nu \;d\mathcal H^{n-1} + \int_{\partial F \cap \partial U} z \cdot \nu \;d\mathcal H^{n-1}.
\end{align*}
Using the estimate
\begin{align*}
\int_{\partial F \setminus \partial U} z \cdot \nu \;d\mathcal H^{n-1} \leq \int_{\partial F \setminus \partial U} \sigma^\circ(z) \sigma(\nu) \;d\mathcal H^{n-1} \leq \int_{\partial F \setminus \partial U} \sigma(\nu) \;d\mathcal H^{n-1},
\end{align*}
we deduce that $\frac{SP(F)}{\mathcal L^n(F)} \geq \lambda = \frac{SP(U)}{\mathcal L^n(U)}$. 

However, it seems that the proof of Theorem~\ref{th:self-cheeger} is available only for $n = 2$. We expect it to be valid in arbitrary dimension.

Let us give a well-known example of a facet that breaks immediately in the evolution.

\begin{example}
\begin{figure}[hbtp]
\centering
\begin{tikzpicture}[scale=1]
\draw[<->] (-1, 1.2) -- node[above] {2} (1, 1.2);
\draw[<->] (-1.2, -1) -- node[left] {2} (-1.2, 1);
\draw[<->] (1.2, 0.5) -- node[right] {$\frac 12$} (1.2, 1);
\draw[<->] (-1, -1.2) -- node[below] {1} (0, -1.2);
\draw[thick] (-1,-1) -- (-1,1) --(0,1)--(0,-1) -- cycle;
\draw[thick] (0,0.5) -- (0,1) --(1,1)--(1,0.5) -- cycle;
\draw (-0.5,0) node {$A$};
\draw (0.5,0.75) node {$B$};
\end{tikzpicture}
\caption{}
\label{fig:breaking-facet}
\end{figure}
Let $n = 2$ and $\sigma(p) = \norm{p}_1 = |p_1| + |p_2|$. Consider the set $C = A \cup B$ with $A = [-1, 0] \times [-1, 1]$ and $B = [0, 1] \times [\frac 12, 1]$, see Figure~\ref{fig:breaking-facet}, and let $\psi (x) = \dist(x, C)$. It is well-known that $C$ considered as a facet of $\psi$ is not calibrable and breaks into two facets $A$ and $B$ moving at different speeds. See \cite{BNP99} for the computation in the crystalline flow case and \cite[Sec.~5]{Moll05} for the explicit computation in the anisotropic total variation flow case. \cite{Moll05} shows that the solution of the anisotropic total variation flow \eqref{crystallinetf} with initial data $u_0 = \one_C$ is given as 
\begin{align*}
u(x, t) = \max(1 - 3t, 0) \one_A + \max(1 - 4t, 0) \one_B. 
\end{align*}
Let us set $U = \interior C$.
In terms of Theorem~\ref{th:self-cheeger} note that $SP(U) = P(U) = 8$ and $\mathcal L^2(U) = \frac{5}2$, yielding a Cheeger ratio $SP(U)/\mathcal L^2(U) = \frac{16}{5} = 3 + \frac 15$, while $A$ has a Cheeger ratio $SP(A)/\mathcal L^2(A) = \frac{6}{2} = 3$, violating \eqref{cheeger-ratio-minimality}. $U$ therefore cannot be calibrable.
\end{example}

%%%%%%%%%% Subsection 5.3
\subsection{Curvature-like quantity} \label{AO3}

The characterization of the subdifferential of the anisotropic total variation and the localization of the canonical restriction $-\partial^0 \E$ motivates the following definition of the crystalline mean curvature. To allow for a forced mean curvature flow, we need to include the forcing into the definition. We follow the notation in \cite{GP3}.

Suppose that $U \subset \Rn$ is an open set and $\psi \in Lip(U)$. If $CH(\psi; U)$ defined in \eqref{CH} is nonempty we define the $\sigma^\circ$-$(L^2)$ divergence of $\psi$ for any $f \in L^2(U)$ as
\begin{align*}
\Lambda_f[\psi] := \divo z_{\rm min} - f \qquad \text{on } \set{\psi = 0}
\end{align*}
where $z_{\rm min}$ is a minimizer of $\norm{\divo z - f}_L^2(U)$ on $CH(\psi; U)$, that is, $\divo z_{\rm min}$ is the projection of $f$ onto $\divo CH(\psi; U)$. Since $\divo CH(\psi; U)$ is closed convex, the value $\divo z_{\rm min}$ is unique, but $z_{\rm min}$ might not be.

One might wonder whether the value of $\Lambda_f[\psi]$ depends on the choice $U$, but thanks to the patching Lemma~\ref{le:patching} that is not the case. For details see \cite[Prop.~4.10]{GP1}.

\begin{remark}
Note that since $\partial\sigma$ is positively $0$-homogeneous, $\Lambda_f[t\psi]$ does not depend on $t > 0$, and in fact for any Lipschitz function $\theta: \R \to \R$ with $\theta(0) = 0$ and $\theta'(s) > 0$ for a.e. $s$ we have $\Lambda_f[\theta \circ \psi] = \Lambda_f[\psi]$. Indeed, by the chain rule for the Lipschitz functions $\nabla (\theta \circ \psi)(x) = \theta'(\psi(x)) \nabla \psi(x)$ a.e. if we interpret the right-hand side as 0 when $\nabla \psi = 0$. The 0-homogeneity of $\nabla \sigma$ implies that $CH(\psi; U) = CH(\theta \circ \psi; U)$.
\end{remark}

We also note the scaling invariance
\begin{align*}
\Lambda_f[\psi](x) = a^{-1}\Lambda_{af(a\cdot)}[\psi(a\cdot)](ax),
\end{align*}
thanks to which we can always assume that $U \subset (-\frac12,\frac12)^n$.

In general, $\Lambda_0$ is only $BV$ and can be discontinuous as was shown in \cite{BNP01a}, \cite{BNP01b}.  
Finding the value of $\Lambda_f[\psi]$ explicitly in dimensions $n \geq 2$ is in general difficult. However, if $\set{\psi=0}$ has a sufficiently regular boundary and there is a vector field in $CH(\psi; U)$ with constant divergence on $\set{\psi=0}$, then $\Lambda_f[\psi]$ can be found as the ratio of the signed anisotropic perimeter and the volume of the facet. Such facets are referred to as \emph{calibrable}. Even though this is well-known in the literature, we have not found a statement that applies precisely to our setting and therefore we present it here with a proof. 
\begin{lemma} % Lemma 5.7
\label{le:value-of-Lambda}
Let $U \subset \Rn$ be bounded open set. Suppose that $\psi \in Lip(U)$, $|\psi| > 0$ on $\partial U$ and there exists $\delta_0 > 0$ such that $|\nabla \psi| > 0$ a.e. on $\set{0 < |\psi| < \delta_0}$ and the sets $\set{\psi < \delta}$, $\set{-\psi < \delta}$ are Lipschitz regular for $\delta \in (0, \delta_0)$, and
\begin{align*}
\int_{\partial \set{\pm\psi < \delta}} \sigma(\pm\nu) \;d\mathcal H^{n-1} &\to \int_{\partial \set{\pm\psi < 0}} \sigma(\pm\nu) \;d\mathcal H^{n-1} \qquad \text{as } \delta \to 0,
\end{align*}
where $\nu$ is the outer unit normal to the respective sets.
If there exists $z_C \in CH(\psi; U)$, that satisfies $\divo z_C = C$ a.e. on $\set{\psi = 0}$ for some constant $C \in \R$, then
\begin{align}
\label{div-z-value}
\Lambda_0[\psi] = C = \frac{\int_{\partial\set{\psi \leq 0}} \sigma(\nu) \;d\mathcal H^{n-1} - \int_{\partial\set{\psi \geq 0}} \sigma(-\nu) \;d\mathcal H^{n-1}}{|\set{\psi = 0}|} \qquad \text{a. e. on } \set{\psi = 0}.
\end{align}
\end{lemma}
If $\psi$ is non-positive and $\sigma=1$, then this number $C$ is the Cheeger ratio of the set $\{ \psi=0 \}$ if the boundary $\partial\{\psi\geq 0\}$ is Lipschitz. Note that we invoke only approximability of surface energy by that of Lipschitz regular set and do not assume Lipschitz regularity of $\partial\{\psi\geq 0\}$ itself.

\begin{proof}
Due to the existence of $z_C$ we know that $\Lambda_0[\psi]$ is well-defined. Let us first prove that for all $z \in CH(\psi; U)$ we have
\begin{align}
\label{int-div-z}
\int_{\set{\psi=  0}} \divo z \dx = \int_{\partial\set{\psi \leq 0}} \sigma(\nu) \;d\mathcal H^{n-1} - \int_{\partial\set{\psi \geq 0}} \sigma(-\nu) \;d\mathcal H^{n-1} = C |\set{\psi = 0}|.
\end{align}
The characterization of $\partial \sigma$ in \eqref{subdiff-polar-char} yields $z \cdot \nabla \psi = \sigma(\nabla \psi)$ a.e. on $\set{0 < |\psi| < \delta_0}$. For $\e > 0$ let $\eta_\e$ be the standard mollifier with radius $\e$ and let $z_\e := z * \eta_\e$, where we extend $z$ by 0 outside $U$. We have
\begin{align*}
z_\e \cdot \nabla \psi &\to \sigma(\nabla \psi) &&\text{a.e. in $U$,}\\
\divo z_\e &\to \divo z && \text{in $L^2(U)$,}
\end{align*}
as $\e \to 0$. The divergence theorem gives
\begin{align*}
\int_{\set{|\psi| < \delta}} \divo z_\e \dx = \int_{\partial \set{|\psi| < \delta}} z_\e \cdot \nu \;d\mathcal H^{n-1}.
\end{align*}
By the coarea formula, $\nu = \frac{\nabla \psi}{|\nabla \psi|}$ $\mathcal H^{n-1}$-a.e. on $\partial\set{\psi < \delta}$ and $\nu = -\frac{\nabla \psi}{|\nabla \psi|}$ $\mathcal H^{n-1}$-a.e. on $\partial\set{\psi > -\delta}$ for a.e. $\delta \in (0, \delta_0)$.
After sending $\e \to 0$, the dominated convergence theorem yields for a.e. $\delta \in (0, \delta_0)$
\begin{align*}
\int_{\set{|\psi| < \delta}} \divo z \dx = \int_{\partial \set{\psi < \delta}} \sigma(\nu) \;d\mathcal H^{n-1} - \int_{\partial \set{\psi > -\delta}} \sigma(-\nu) \;d\mathcal H^{n-1}.
\end{align*}
Sending $\delta \to 0$ along a sequence leads to \eqref{int-div-z}. We recover the second equality in \eqref{int-div-z} by recalling that $z_C \in CH(\psi; U)$ satisfies $\divo z_C = C$ a.e. on $\set{\psi = 0}$.

Let us write $A = \set{\psi = 0}$. For $v = \divo z_{\rm min}$ we have $\int_A v \dx = \int_A C \dx$ by \eqref{int-div-z} and therefore
\begin{align}
\label{l2-rel-div}
\int_A v^2 \dx = \int_A C^2 \dx + \int_A(v - C)^2 \dx \geq \int_A C^2 \dx.
\end{align}
By the Cahn-Hoffman vector field patching Corollary~\ref{co:patching-zero}, the vector field
\begin{align*}
\tilde z = z_C \one_A + z_{\rm min} \one_{U \setminus A}
\end{align*}
is also Cahn-Hoffman with
\begin{align*}
\divo \tilde z = \divo z_C \one_A + \divo z_{\rm min} \one_{U \setminus A}
\qquad \text{a.e. in } U.
\end{align*}
Therefore \eqref{l2-rel-div} implies that $\norm{\divo z_{\rm min}}_{L^2(U)} \geq \norm{\divo \tilde z}_{L^2(U)}$ and we conclude that $\divo \tilde z$ is minimizing. By uniqueness, $\divo z_{\rm min} = C$ a.e. on $A$.
\end{proof}

Let us conclude with a few examples of simple useful facets for which we can compute $\Lambda_0$ explicitly.

\begin{example}\label{ex:wulff-facet}
Wulff facet.

For $r > 0$ consider $\psi(x) := \max(\sigma^\circ(x) - r, 0)$. We have $\set{\psi = 0} = r W_\sigma$.

Take $U$ to be a sufficiently large open ball containing $r W_\sigma$ and consider the vector field
\begin{align*}
z(x) :=
\begin{cases}
\frac{x}{r}, & \sigma^\circ(x) \leq r,\\
\frac{x}{\sigma^\circ(x)}, & \text{otherwise}.
\end{cases}
\end{align*}
Clearly $z \in L^\infty(U)$ and $\divo z \in L^2(U)$ with
\begin{align*}
\divo z = 
\begin{cases}
\frac nr, & \sigma^\circ(x) \leq r,\\
\frac{n-1}{\sigma^\circ(x)}, & \text{otherwise}.
\end{cases}
\end{align*}
It is easy to check that $z \in \partial \sigma(\nabla \psi)$ a.e. Therefore $CH(\psi; U) \neq \emptyset$
and  $\Lambda_0[\psi] = \frac nr$ on $\set{\psi = 0}$ by Lemma~\ref{le:value-of-Lambda}.
\end{example}

\begin{example}
\label{ex:facet-hole}
Facet with a hole; Fig.~\ref{fig:facet-hole}.
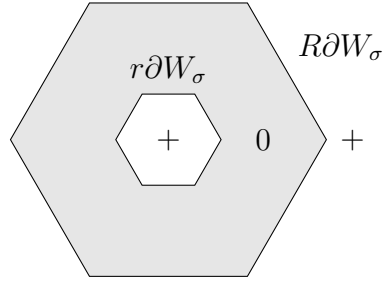
\begin{figure}
\centering
\begin{tikzpicture}[scale=0.7]
\draw[even odd rule,fill=black!10!white] (0:1) -- (60:1) -- node[above] {$r \partial W_\sigma$} (120:1) -- (180:1) -- (240:1) -- (300:1) -- cycle
 (0:3) -- node[above right] {$R \partial W_\sigma$} (60:3) -- (120:3) -- (180:3) -- (240:3) -- (300:3) -- cycle;
\path (0,0) node {$+$} (1.8,0) node {$0$} (3.5,0) node {$+$};
\end{tikzpicture}
\caption{Wulff facet with a hole in Example~\ref{ex:facet-hole}, with sign of $\psi$ indicated.}
\label{fig:facet-hole}
\end{figure}

Suppose that $\sigma^\circ$ is even, i.e., $\sigma^\circ(-x) = \sigma^\circ(x)$ for all $x$. Consider $0 < r < R$ and the function 
\begin{align*}
\psi(x) := \max(r - \sigma^\circ(x), 0, \sigma^\circ(x) - R),
\end{align*}
so that $\set{\psi = 0} = RW_\sigma \setminus \interior r W_\sigma$.
Let us
set
\begin{align*}
a := \frac{R^{n-1} r^{n-1}(R + r)}{R^n - r^n}, \qquad b :=\frac{R^{n-1} + r^{n-1}}{R^n - r^n}
\end{align*}
We claim that the vector field 
\begin{align*}
z(x) :=
\begin{cases}
- \frac{x}{\sigma^\circ(x)}, & \sigma^\circ(x) \leq r,\\
\big(-a (\sigma^\circ(x))^{-n} + b\big) x, & r < \sigma^\circ(x) < R,\\
\frac{x}{\sigma^\circ(x)}, & \sigma^\circ(x) \geq R,\\
\end{cases}
\end{align*}
is a Cahn--Hoffman vector field for $\psi$ on any $U$ away from $x = 0$. To see that, we consider $g(s):= (-as^{-n} +b)s$. We note that $g(r) = -1$ and $g(R) = 1$, and $g$ is increasing on $s > 0$ which yields $-1 < g(s) < 1$ for $r < s <R$. By the assumption that $\sigma^\circ$ is even, we have
\begin{align*}
\sigma^\circ(z(x)) = |g(\sigma^\circ(x))| \leq 1 \qquad r < \sigma^\circ(x) < R.
\end{align*}
This by the characterization of the subdifferential, for example \eqref{subdiff-polar-char}, implies that $z(x) \in \partial \sigma(0) = \partial \sigma(\nabla \psi(0))$ for $r < \sigma^\circ(x) < R$. For other $x$ the inclusion $z(x) \in \partial \sigma(\nabla \psi(x))$ a.e.{} is obvious.

We also see that $z$ is in fact Lipschitz continuous away from $x = 0$. Therefore $\divo z \in L^2(U)$ for any $U$ away from $x = 0$ and hence $z \in CH(\psi; U)$.

A direct computation using $x \cdot \nabla \sigma^\circ(x) = \sigma^\circ(x)$ yields that almost everywhere
\begin{align*}
\divo z = 
\begin{cases}
- \frac{n-1}{\sigma^\circ(x)}, & \sigma^\circ(x) \leq r,\\
n b, & r < \sigma^\circ(x) < R,\\
\frac{n-1}{\sigma^\circ(x)}, & \sigma^\circ(x) \geq R.\\
\end{cases}
\end{align*}
In particular, $\Lambda_0[\psi] = nb = n \frac{R^{n-1} + r^{n-1}}{R^n - r^n}$ by Lemma~\ref{le:value-of-Lambda}, matching the formula \eqref{div-z-value}.
\end{example}

\begin{example}
\label{ex:convex-concave}
Convex-concave facet; Fig.~\ref{fig:facet-hole} with negative sign in the hole.

Consider $0 < r < R$ and the function 
\begin{align*}
\psi(x) := \min(\sigma^\circ(x)-r, \max(0, \sigma^\circ(x) - R)),
\end{align*}
so that again 
$\set{\psi = 0} = RW_\sigma \setminus \interior r W_\sigma$, but this time $\psi < 0$ in $r W_\sigma$.
The vector field 
\begin{align*}
z(x) :=
\frac{x}{\sigma^\circ(x)}
\end{align*}
is a Cahn--Hoffman vector field for $\psi$ on any $U$ away from $0$.

But we can be more precise as in Example~\ref{ex:facet-hole}. Let us
set
\begin{align*}
a := \frac{R^{n-1} r^{n-1}(R - r)}{R^n - r^n}, \qquad b :=\frac{R^{n-1} - r^{n-1}}{R^n - r^n}
\end{align*}
and consider the vector field
\begin{align*}
z(x) :=
\begin{cases}
\frac{x}{\sigma^\circ(x)}, & \sigma^\circ(x) \leq r,\\
\big(a (\sigma^\circ(x))^{-n} + b\big) x, & r < \sigma^\circ(x) < R,\\
\frac{x}{\sigma^\circ(x)}, & \sigma^\circ(x) \geq R.\\
\end{cases}
\end{align*}
This is a Lipschitz continuous vector field away from $x = 0$. Moreover, $\divo z = n b$ almost everywhere for $r < \sigma^\circ(x) < R$.

Let us check that it is a Cahn--Hoffman vector field. The inclusion $z(x) \in \partial \sigma(\nabla \psi(x))$ is clear for $\sigma^\circ(x) < r$ and $R < \sigma^\circ(x)$. Since $\psi(x) \equiv 0$ for $r < \sigma^\circ(x) < R$, we only need to check that $\sigma^\circ(z) \leq 1$ by \eqref{subdiff-polar-char}.

The function $g(s) := (a s^{-n} + b) s$ is convex on $s > 0$ with minimum at $\hat s = (\frac{b}{(n-1)a})^{-1/n}$ with value $g(\hat s) = \frac{n}{n-1} b \hat s > 0$. Therefore $a s^{-n} + b > 0$ for $s > 0$ and we have
\begin{align*}
\sigma^\circ(z(x)) = g(\sigma^\circ(x)) \qquad r < \sigma^\circ(x) < R.
\end{align*}
Since $g(r) = g(R) = 1$, by convexity of $g$ we conclude that $\sigma^\circ(z(x)) \leq 1$ for all $x \neq 0$. Therefore $z$ is a Cahn--Hoffman vector field with constant divergence on the facet $\set{\psi = 0}$ and hence by Lemma~\ref{le:value-of-Lambda} we have $\Lambda_0[\psi] = n b$.
\end{example}

%%%%%%%%%% Subsection 5.4
\subsection{Comparison and approximation}\label{AO4}

We start with the comparison principle for the $\sigma^\circ$-($L^2$) divergence.
Here $\sign s = -1, 0, 1$ if $s < 0$, $s =0$, or $s > 0$ respectively.

\begin{proposition} \label{pr:independence-min-div}
Let $\sigma$ be convex, positively one-homogeneous function on $\Rn$ that is positive away from 0. Let $U$ be an open subset of $\Rn$ and let $\psi_i \in Lip(U)$ with $\set{\psi_i = 0}$ compact subsets of $U$. Let $f_i \in L^2(U)$. If $CH(\psi_i; U) \neq \emptyset$ and 
\begin{align*}
\sign \psi_1 \leq \sign \psi_2, \qquad f_1 \geq f_2
\end{align*}
then
\begin{align*}
\Lambda_{f_1}[\psi_1] \leq \Lambda_{f_2}[\psi_2] \qquad \text{a.e. on } \set{\psi_1 = 0} \cap \set{\psi_2 = 0}.
\end{align*}
\end{proposition}

Before we proceed with the proof, we recall here a technical lemma proved in \cite[Lemma~4.13]{GP1}. It is a variant of a result for continuous functions established in \cite{CGG,ES}; % [CGG,ES] はOK？
see also \cite[Lemma~4.2.9]{Giga_Book}.
\begin{lemma}
\label{le:Lipschitz-ordering}
Suppose that $\psi$ and $\varphi$ are two nonnegative periodic Lipschitz functions on $\Rn$, such that
$\set{\psi = 0} \subset \set{\varphi = 0}$.
Then there exists a Lipschitz continuous function $\theta: [0, \infty) \to [0, \infty)$ such that
$\theta(0) = 0$, $\theta(s) > 0$ for $s > 0$ and $\theta'(s) > 0$ for almost every $s > 0$
and we have
\begin{align*}
\theta \circ \varphi \leq \psi \qquad \text{on $\Rn$.}
\end{align*}
\end{lemma}

\begin{proof}[Proof of Proposition~\ref{pr:independence-min-div}]
For simplicity, we assume $f_1 = f_2 = 0$.
We can assume that $U$ is connected and by scaling assume that $U \subset (-\frac12, \frac12)$. By making $U$ smaller if necessary, we may assume that $\min_{\partial U} |\psi_i| > 0$.

We can modify $\psi_i$ away from $\set{\psi_i = 0}$ to make it nonzero constant near $\partial U$ and then extend it using this constant periodically so that $\psi_i \in Lip(\Tn)$ and $CH(\psi_i; \Tn) \neq \emptyset$ without changing the value of $\Lambda_{0}[\psi_i]$ on $\set{\psi_i = 0}$. If $\dimension = 1$ we might have to do an even extension first if the sign differs on $\inf U$ and $\sup U$. 

By Lemma~\ref{le:Lipschitz-ordering}, we can also find $\theta_1, \theta_2 \in Lip(\R)$ with $\theta_i(0) = 0$ and $\theta_i' > 0$ a.e. so that $\theta_1 \circ \psi_1 \leq \theta_2 \circ \psi_2$ everywhere. Since $CH(\theta_i \circ \psi_i; \Tn) = CH(\psi_i; \Tn)$, we can replace $\psi_i$ with $\theta_i \circ \psi_i$ and assume that $\psi_1 \leq \psi_2$ on $\Tn$. 

We then have $\Lambda_{0}[\psi_i] = -\partial^0 \E(\psi_i)$ on $\set{\psi_i = 0}$ by the characterization of the subdifferential in Corollary~\ref{co:lip-subdiff-char}.  
Due to the comparison Proposition~\ref{pr:resolvent-comp-principle}, we have $\psi_1 \leq \psi_2$ implies $\psi_{1,a} \leq \psi_{2,a}$ for the solutions of the resolvent problem \eqref{resolvent-problem}. On $\set{\psi_1 = 0} \cap \set{\psi_2 = 0}$ we have
\begin{align*}
\frac{\psi_{1,a} - \psi_1}a \leq \frac{\psi_{2,a} - \psi_2}a,
\end{align*}
which using the convergence in Proposition~\ref{pr:resolvent-approximation} and sending $a \to 0$ implies $-\partial^0 \E(\psi_1) \leq -\partial^0 \E(\psi_2)$ a.e. on $\set{\psi_1 = 0} \cap \set{\psi_2 = 0}$.
\end{proof}

\noindent
\textbf{Abstract facets.} The comparison principle for $\Lambda_f[\psi]$ implies that the value on $\set{\psi = 0}$ depends only on $f$ and $\sign \psi$. We define the relation $\sim$ on $\mathcal F = \set{\xi \mid \xi: \Rn \to \R}$ the set of all real-valued function on $\Rn$ as
\begin{align*}
\xi_1 \sim \xi_2 \qquad \Leftrightarrow \qquad \sign \xi_1 = \sign \xi_2.
\end{align*}
This relation is an equivalence relation on $\mathcal F$. We refer to its equivalence classes $[\xi] := \set{\psi \mid \psi\sim\xi} \subset \mathcal F$ as \emph{(abstract) facets}. We write $[\xi_1] \preceq [\xi_2]$ when $\sign \xi_1 \leq \sign \xi_2$ and this relation defines a partial order on the set of all facets $\mathcal F /\sim := \set{[\xi] \mid \xi \in \mathcal F}$.

\noindent
\textbf{Cahn--Hoffman facet.} We say that a facet $[\xi]$ is a \emph{$\sigma^\circ$-$(L^2)$ Cahn--Hoffman facet} if $\set{\xi = 0}$ is compact and there are an open set $U \subset \Rn$, $\set{\xi = 0} \subset U$ and a Lipschitz function $\psi \in [\xi]$ such that $CH(\psi; U) \neq \emptyset$. The facets in Examples~\ref{ex:wulff-facet}--\ref{ex:convex-concave} are Cahn--Hoffman.

\begin{proposition}
\label{pr:lambda-comparison}
For $\sigma^\circ$-$(L^2)$ Cahn--Hoffman facets $[\chi_1]$ and $[\chi_2]$ and functions $f_i \in L^2(\set{\chi_1 = 0} \cup \set{\chi_2 = 0})$ we have
\begin{align*}
[\chi_1] \preceq [\chi_2],\qquad f_1 \geq f_2 \text{ a.e.}
\end{align*}
implies
\begin{align*}
\Lambda_{f_1}[\chi_1] \leq \Lambda_{f_2}[\chi_2] \text{ a.e. on } \set{\chi_1 = 0} \cap \set{\chi_2 = 0}.
\end{align*}
\end{proposition}

We will use $\sigma^\circ$-$(L^2)$ Cahn-Hoffman facets to build test functions for viscosity solutions of the crystalline mean curvature flow and so we need to make sure there are enough of them. In fact, any facet with bounded zero set can be approximated by $\sigma^\circ$-$(L^2)$ Cahn-Hoffman facets monotonically arbitrarily close in the Hausdorff distance.
The following theorem was proven in \cite{MGP1} for $\sigma$ the Euclidean norm, and in \cite{GP2} in full generality.
\begin{theorem}
  \label{th:facet-density}
  Let $\chi$ be an $n$-dimensional facet with $\set{\chi = 0}$ bounded and $\sigma$ an anisotropy.  Given $\rho > 0$ there exists a $\sigma^\circ$-$(L^2)$ Cahn-Hoffman facet $\tilde \chi$ such that $\chi(x) \leq \tilde \chi(x) \leq \sup_{\abs{x - y} \leq \rho}\chi(y)$ for $x \in \Rn$.
\end{theorem}

%\todo{Explain why we are working on $\Tn$ and not $\Rn$ or open sets $U$ (Neumann problem), possible ways to extend this to $\Rn$.}

%%%%%%%%%% Subsection 6
\section{Approach by the theory of viscosity solutions}  \label{AV}

In this section we introduce a notion of \emph{viscosity} solutions for nonlinear partial differential equations that include the very singular term $\divo \nabla \sigma(\nabla u) - f$ that represents an anisotropic curvature with forcing.

For the definition of the anisotropic mean curvature we use the quantity $\Lambda_f$ that was introduced in Section~\ref{AO3}.
It is important to note that if $f$ depends on $x$, the term $\divo \nabla \sigma(\nabla u) - f$ must be carefully defined together and $f$ cannot be added separately. Heuristically, the anisotropic mean curvature flow prefers flat facets in the singular directions of $\sigma$ even in the presence of nonuniform forcing, and so the full quantity $\divo \nabla \sigma(\nabla u) - f$ should be constant on facets. If we considered the forcing $f$ separately in the definition of a viscosity solution, the comparison principle would still be valid however we would have a problem with stability in the approximation by regularized problems and ultimately we could not establish existence of solutions.
For a counterexample to existence see \cite[Sec.~6]{GP3}.

\subsection{Definition of viscosity solutions}

If $\sigma \in C^2(\Rn \setminus \set0)$, it only has a singularity at $p = 0$ and we have everything we need to define the viscosity solution for \eqref{pde}. The following is the notion of the viscosity solution introduced in \cite{MGP1,MGP2} assuming that $F$ does not depend on $x$ and $t$ and there is no forcing term.

\begin{definition}
\label{de:visc-sol-tv-flow}
An upper semicontinuous function $u$ on $\Rn \times (0,\infty)$  is a \emph{viscosity subsolution} of 
\begin{align}
\label{pde-restricted}
u_t + F(\nabla u, \divo \nabla \sigma(\nabla u)) = 0
\end{align}
if the following two conditions hold:
\begin{itemize}
\item[(i)] (conventional test) If $\varphi \in C^2$ near $(\hat x, \hat t)$, $\nabla \varphi(\hat x, \hat t) \neq 0$ and $u - \varphi$ has a local maximum at $(\hat x, \hat t)$, then
\begin{align}
\label{restricted-visc-subs}
\varphi_t(\hat x, \hat t) + F(\nabla \varphi(\hat x, \hat t), \divo \nabla \sigma(\nabla \varphi) (\hat x, \hat t)) \leq 0.
\end{align}
\item[(ii)] (faceted test) If $\varphi(x,t) = \psi(x) + g(t)$ with $g \in C^1(\R)$ and $\psi \in Lip(\Rn)$ so that $[\psi]$ is a $\sigma^\circ$-$(L^2)$ Cahn-Hoffman facet, $\hat x \in \interior \set{\psi = 0}$, $u - \varphi(\cdot - h)$ has a global maximum at $(\hat x, \hat t)$ for all $|h|$ small, then there exists $\delta > 0$ such that
\begin{align}
\label{tv-sing-subsol}
g'(\hat t) + F(0, \essinf_{B_{\delta}(\hat x)}\Lambda_0[\psi]) \leq 0.
\end{align}
\end{itemize}

A lower semi-continuous function is a \emph{viscosity supersolution} if it satisfies the above two conditions with maximum, $\leq$ and $\essinf\Lambda$ replaced by minimum, $\geq$ and $\esssup\Lambda$, respectively.
\end{definition}

Let us remark that in \cite{MGP1,MGP2} the facet test was restricted to test functions where the facet $[\psi]$ has a smooth boundary. However, this is not essential as was observed in later papers.

As you can see, we need to reduce the class of test functions testing at points where $\nabla u = 0$ to be even able to define a reasonable value of $\divo \nabla \sigma(\nabla \varphi)$. 

To include a forcing term $f$ that depends on the $x$ variable, we can follow \cite{GP3} to modify the above definition.
We introduce
\begin{align*}
\underline\Lambda_f[\xi](x) := \lim_{\delta \to 0+} \essinf_{B_\delta(x)} \Lambda_f[\xi], \qquad
\overline\Lambda_f[\xi](x) := \lim_{\delta \to 0+} \esssup_{B_\delta(x)} \Lambda_f[\xi],
\end{align*}
on the interior of $\set{\xi = 0}$, which are well-defined and finite by the comparison principle with Wulff facets in Example~\ref{ex:wulff-facet} as long as $f$ is locally bounded. In fact, in this case $\underline\Lambda_f[\xi]$ is lower semi-continuous while $\overline\Lambda_f[\xi]$ is upper semi-continuous.

Then we can define a viscosity subsolution of the PDE
\begin{align}
\label{pde}
u_t + F(x, t, \nabla u, \divo \nabla \sigma(\nabla u) - f) = 0
\end{align}
following the above definition, but replacing \eqref{restricted-visc-subs} with
\begin{align*}
\varphi_t(\hat x, \hat t) + F(\hat x, \hat t, \nabla \varphi(\hat x, \hat t), \divo \nabla \sigma(\nabla \varphi) (\hat x, \hat t) - f(\hat x, \hat t)) \leq 0.
\end{align*}
and \eqref{tv-sing-subsol} with
\begin{align}
\label{tv-sing-subsol-f}
g'(\hat t) + F(\hat x, \hat t, 0, \underline\Lambda_f[\psi](\hat x)) \leq 0.
\end{align}

This latter condition is slightly weaker than \eqref{tv-sing-subsol} used in \cite{MGP1,MGP2}, and allows for the proof of stability to handle non-constant driving force $f$. 

If the anisotropy $\sigma$ has singularities other than at $p = 0$, the faceted test has to be extended to those gradients of the solution. However, the singular set of $\sigma$ might be in general very complicated and it is not clear how to define a viscosity solution for a general convex anisotropy $\sigma$ (or a convex function $\sigma$) except in one dimension. 

Therefore we restrict our attention to crystalline anisotropies: $\sigma$ is called crystalline if it is a maximum of a finite number of linear functions. In this case, the structure of singularities of $\sigma$ is relatively simple. The ``kind'' of singularity is determined by the dimension of the subdifferential $\partial \sigma(p)$, which corresponds to the expected dimension of the facet in the direction $p$. We introduce the following orthogonal decomposition of the space $\Rn$. For a fixed gradient $\hat p \in \Rn$, define $Z$ to be the linear subspace of $\Rn$ parallel to the affine hull of $\partial \sigma(\hat p)$, see Figure~\ref{fig:slicing}.
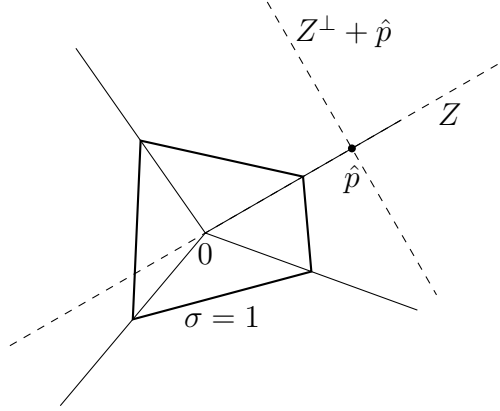
\begin{figure}
\centering
\begin{tikzpicture}[scale=1.5]
\draw[thick] (30:1) -- (125:1) --  (230:1) -- node[below] {$\sigma = 1$}(340:1) -- cycle;
\draw (0,0) -- (30:2) (0,0) -- (125:2) (0,0) --  (230:2)  (0,0) -- (340:2);
\draw[dashed] (30:1.5)   ++ (120:1.5) -- node[right,pos=0.1] {$Z^\perp + \hat p$} ++ (120:-3);
\draw[dashed] (30:-2) -- node[pos=0.9,below] {$Z$} (30:3);
\fill (30:1.5) circle[radius=1pt] node[below=2pt] {$\hat p$};
\path (0,0) node[below] {0};
% \draw[->] (-2,0) -- (2,0) node[right] {$p_1$};
% \draw[->] (0,-2) -- (0,2) node[left] {$p_2$};
\end{tikzpicture}
\caption{An illustration of the orthogonal decomposition $\R^2 = Z \oplus Z^\perp$ at $\hat p$ that lies on a ``one-dimensional'' singularity of $\sigma$. The thick polygon is the boundary $\partial F_\sigma = \set{\sigma = 1}$, and the solid rays from the origin indicate where $\dim \partial \sigma = 1$.}
\label{fig:slicing}
\end{figure}
In other words, $Z$ is the smallest linear subspace such that $\partial \sigma(\hat p) \subset Z + \xi$ for some $\xi \in \Rn$. Set $k := \dim \partial \sigma(\hat p) := \dim Z$. We have an orthogonal decomposition $\Rn = Z \oplus Z^\perp$. We fix orthonormal bases of $Z$, $Z^\perp$ which give two linear isometries $\TT: \R^k \to Z$ and $\TT_\perp : \R^{n-k} \to Z^\perp$. This allows us to write any $x \in \Rn$ uniquely as $x = \TT x' + \TT_\perp x''$ for some $x' \in \R^k$ and $x'' \in \R^{n-k}$. For $k = 0$, $k = n$ we take $x = x''$ and $x = x'$, respectively. If we denote the adjoint of $\TT$ as $\TT^*$ and of $\TT_\perp$ as $\TT_\perp^*$, we have $x' = \TT^* x$ and $x'' = \TT^*_\perp x$.

Using the above decomposition, we can ``slice'' the anisotropy $\sigma$ to extract only the part that contains the singularity by introducing
\begin{align*}
\sigma_{\hat p}^{\rm sl}(w) := \lim_{\lambda \to 0+} \frac{\sigma(\hat p + \lambda \TT w) - \sigma(\hat p)}\lambda, \qquad w \in \R^k.
\end{align*}
This sliced function is again positively one-homogeneous and so we can introduce a curvature-like quantity $\Lambda_{\hat p, f}[\psi]$ for $(\sigma_{\hat p}^{\rm sl})^\circ$-$(L^2)$ Cahn--Hoffman facets $[\psi]$ on $\R^k$ and $f \in \R^k$.

Let us give the definition of viscosity solution assuming that $f \equiv 0$ that appeared in \cite[Def.~4.7]{GP2} with $F$ independent of $x$ and $t$ and $\sigma$ purely crystalline.

\begin{definition}
\label{de:visc-sol-crystalline}
An upper semicontinuous function $u$ on $\Rn \times (0,\infty)$  is a \emph{viscosity subsolution} of 
\begin{align}
\label{pde-no-f}
u_t + F(\nabla u, \divo \nabla \sigma(\nabla u)) = 0
\end{align}
if whenever $\hat p \in \Rn$, $\hat x \in \Rn$, $\hat t \in (0, T)$ and $\varphi$ is \emph{stratified} test function $\varphi(x,t) = \psi(x') + \theta(x'') + \hat p \cdot x + g(t)$ with $g \in C^1(\R)$, $\theta \in C^1(\R^{n-k})$ satisfying $\nabla \theta(\hat x'') = 0$, and $\psi \in Lip(\R^k)$ so that $[\psi]$ is a $(\sigma_{\hat p}^{\rm sl})^\circ$-$(L^2)$ Cahn-Hoffman facet, $\hat x' \in \interior \set{\psi = 0}$, and $u - \varphi(\cdot - h)$ has a global maximum at $(\hat x, \hat t)$ for all $h = \TT h'$ with $h'$ small, then 
\begin{align}
\label{crystalline-visc-condition}
g'(\hat t) + F(\hat p, \essinf_{B_\delta(\hat x)} \Lambda_{\hat p,  0}[\psi]) \leq 0.
\end{align}

A lower semi-continuous function is a \emph{viscosity supersolution} if it satisfies the above two conditions with maximum, $\leq$ and $\essinf$ replaced by minimum, $\geq$ and $\esssup$, respectively.
\end{definition}

When a forcing $f$ that depends on the $x$ variable is involved, the condition \eqref{crystalline-visc-condition} has to be weakened as in \eqref{tv-sing-subsol-f}, replacing $\essinf \Lambda_{\hat p, 0}$ by $\underline\Lambda_{\hat p, \hat f}$, $\hat f(w) = f(\hat x + \TT w)$, for the stability with respect to an approximation by regularized problems to hold. See \cite[Def.~2.7]{GP3} for more details.

\subsection{Comparison principle}
\label{se:comparison}

In this section we review the comparison principle for the PDE \eqref{pde}. There are a few different versions available depending on the assumptions on $F$, $\sigma$ and $f$. 

Let us first suppose that $f \equiv 0$.
The comparison theorem was first proved in \cite{MGP1,MGP2} in the setting of an anisotropic total variation flow with smooth anisotropy $\sigma \in C^2(\Rn \setminus \set0)$ with $\sigma^2$ strictly convex and $F$ independent of the $x$ and $t$ variables in the sense of Definition~\ref{de:visc-sol-tv-flow} on the torus $\Tn = \Rn / \Z^n$.
We follow \cite[Th.~4.1]{MGP2}.

\begin{theorem}
Let $\sigma \in C^2(\Rn \setminus \set0$ be an anisotropy and $F\in C(\Rn \times \R)$ be nonincreasing in the second variable.
Let $u$ and $v$ be respectively a bounded viscosity subsolution and a viscosity supersolution of \eqref{pde-no-f} in the sense of Definition~\ref{de:visc-sol-tv-flow} on $\Tn \times [0, T]$. If $u \leq v$ at $t = 0$ then $u \leq v$ on $\Tn \times (0, T)$.
\end{theorem}

When $\sigma$ is crystalline while there is still no forcing, $f \equiv 0$, and $F$ does not depend on $x$ and $t$, the following comparison principle for viscosity solutions in the sense of Definition~\ref{se:comparison} was proved in \cite{GP1,GP2}. We follow the statement in \cite[Th.~1.4]{GP2}. Instead of a torus, the solutions are assumed to be constant outside of a ball.

\begin{theorem}
Let $\sigma$ be a crystalline anisotropy and $F\in C(\Rn \times \R)$ be nonincreasing in the second variable, $F(0, 0) = 0$. Suppose that $u$ is a viscosity subsolution and $v$ is a viscosity supersolution of \eqref{pde-no-f} on $\Rn \times (0, T)$ in the sense of Definition~\ref{de:visc-sol-crystalline} and that there are constants $R > 0$, $a \leq b$ such that $u = a$ and $v = b$ on $(\Rn \setminus B_R(0)) \times (0, T)$. Then if $u \leq v$ at $t = 0$, then $u \leq v$ on $\Rn \times (0, T)$.
\end{theorem}

If we consider a forcing term $f$ that depends on $x$ and $t$ and solutions of \eqref{pde}, an analogous comparison principle was proven in \cite[Th.~3.1]{GP3}. However, in this case at least one of $u$ or $v$ must be continuous, and either it is also Lipschitz, or further regularity of $F$ must be assumed. For details see \cite{GP3}.

For simplicity of exposition, we present here the proof of a comparison theorem in a simplified setting of a stationary problem. We consider the stationary equation with $\sigma(p) = |p|$, $\nabla \sigma(p) = \frac p{|p|}$, that reads
\begin{align}
\label{stationary-problem}
u - \divo \pth{\frac{\nabla u}{|\nabla u|}} = f,
\end{align}
where $f \in C(\Rn)$ is given.
This equation is of elliptic type.
The definition of viscosity solution is naturally modified to the following.

An upper semicontinuous function $u$ is a viscosity subsolution of \eqref{stationary-problem} if
\begin{itemize}
\item If $\varphi \in C^2$ near $\hat x$, $\nabla \varphi(\hat x) \neq 0$ and $u - \varphi$ has a local maximum at $\hat x$, then
\begin{align*}
u(\hat x) - \divo \pth{\frac{\nabla \varphi}{|\nabla\varphi|}} (\hat x) \leq f(\hat x).
\end{align*}
\item If $\varphi \in Lip$ so that $[\varphi]$ is a $\sigma^\circ$-$(L^2)$ Cahn-Hoffman facet, $\hat x \in \interior \set{\varphi = 0}$, $u - \varphi(\cdot - h)$ has a global maximum at $\hat x$ for all $|h|$ small, then
\begin{align*}
u(\hat x) - \underline\Lambda_f[\varphi](\hat x) \leq 0.
\end{align*}
\end{itemize}
A definition of a supersolution and a solution can be modified analogously.

Recall that
\begin{align*}
\divo \pth{\frac{\nabla \varphi}{|\nabla\varphi|}} =
\frac1{|\nabla\varphi|} \trace \bra{\pth{I - \frac{\nabla\varphi \otimes \nabla\varphi}{|\nabla\varphi|^2}} \nabla^2 \varphi}.
\end{align*}

\begin{remark}
Note that the proof in this section applies with a small modification to general problems
\begin{align*}
F\pth{u, \nabla u, \divo \pth{\tfrac{\nabla u}{|\nabla u|}}} = 0,
\end{align*}
where $F: \R \times \Rn \times \R \to \R$ is a continuous function, satisfying monotonicity
\begin{align*}
F(r, p, \xi) &\leq F(r, p, \eta) \qquad \text{for $\xi \geq \eta$},\\
F(r, p, \xi) &\leq F(s, p, \xi) - \mu (s - r) \qquad \text{for $r < s$},
\end{align*}
where $\mu > 0$ is a constant.
We write the proof for $F(r, p, \xi) := r - \xi$.
\end{remark}

\begin{theorem}
\label{th:visc-comparison-principle}
Suppose that $u$ and $v$ are a viscosity subsolution and a viscosity supersolution on $\Rn$, respectively. Furthermore, assume that $u$ and $v$ are bounded.
If there exist constants $R > 0$ and $a \leq b$ such that $u \equiv a$ and $v \equiv b$ on $\Rn \setminus B_R(0)$ then $u \leq v$ everywhere.
\end{theorem}

To show this theorem, we assume that the conclusion does not hold and
\begin{align*}
m_0 := \sup(u - v) = \max (u - v) > 0.
\end{align*}
To reach a contradiction, we double variables and for $\zeta \in \Rn$ we consider the function
\begin{align*}
\Phi_{\zeta}(x, y) = u(x) - v(y) - \frac{|x - y - \zeta|^2}{2\e}.
\end{align*}
We consider the maximum of $\Phi_\zeta$ as a function of $\zeta$, that is,
\begin{align*}
\ell(\zeta) = \sup_{x,y} \Phi_\zeta.
\end{align*}
It is convenient to introduce the set of points of maxima
\begin{align*}
\mathcal A(\zeta) := \argmax \Phi_\zeta := \set{(x, y) \mid \Phi_\zeta(x, y) = \ell(\zeta)}
\end{align*}
and the set of gradients of $\frac{|x - y - \zeta|^2}{2\e}$ at these points
\begin{align*}
\mathcal B(\zeta) := \set{\frac{x-y-\zeta}\e\mid (x, y) \in \mathcal A(\zeta)}.
\end{align*}

%\todo{$\mathcal A(\zeta)$ is nonempty (compactness?).}

The parameter $\e > 0$ determines how much we penalize $x \neq y$. We have the following standard estimate on $|x - y|$; see \cite{GG1}. We give a proof for completeness.
\begin{lemma}
\label{le:doubling-penalty}
There is $C > 0$ such that for all $\e > 0$, $|\zeta| \leq \sqrt{m_0 \e}$ we have
\begin{align}
\label{xy-distance}
|x - y| \leq C\sqrt{\e} \qquad \text{for all $(x, y) \in \mathcal A(\zeta)$}.
\end{align}
Furthermore
\begin{align*}
\ell(\zeta) \geq \frac{m_0}2.
\end{align*}
\end{lemma}

\begin{proof}
First note that
\begin{align*}
\ell(\zeta) \geq \sup_x(u(x) - v(x)) - \frac{|\zeta|^2}{2\e} \geq m_0 - \frac{m_0}2 = \frac{m_0}2.
\end{align*}
Let $M$ be a bound on $u(x) - v(y)$. Then for $(x,y) \in \mathcal A(\zeta)$ we obtain
\begin{align*}
M - \frac{|x - y - \zeta|^2}{2\e} \geq u(x) - v(y) - \frac{|x - y - \zeta|^2}{2\e} = \ell(\zeta) \geq \frac{m_0}2 > 0.
\end{align*}
Hence
\begin{align*}
|x- y| \leq \sqrt{2M\e} + |\zeta| \leq \sqrt{2M\e} + \sqrt{m_0 \e}.
\end{align*}
\end{proof}

Let $\omega_f$ be the modulus of continuity of $f$, that is, $\omega: [0,\infty) \to [0, \infty)$, $\omega(0) = 0$, $\omega$ nondecreasing such that $|f(x) - f(y)| \leq \omega_f(|x - y|)$ for all $x, y$. Let $\e_0 > 0$ be such that $\omega_f(C\sqrt{\e_0}) < \frac{m_0}4$, where $C$ is from \eqref{xy-distance}.
We consider two possible cases:
\begin{itemize}
\item[1.] There exist $0 < \e < \e_0$ and $|\zeta| \leq \sqrt{m_0 \e}$ such that $\mathcal B(\zeta) \neq \set0$.
\item[2.] $\mathcal B(\zeta) = \set0$ for all $|\zeta| \leq \sqrt{m_0 \e}$, $0 < \e < \e_0$.
\end{itemize}

\medskip
\textbf{Case 1.} We can fix $\e$, $\zeta$ and $(x, y)\in \mathcal A(\zeta)$ such that $0 < \e < \e_0$, $|\zeta| \leq \sqrt{m_0 \e}$ and $p := \frac{x - y - \zeta}\e \neq 0$. Note that this choice together with Lemma~\ref{le:doubling-penalty} implies
\begin{align}
\label{fx-fy}
|f(x) - f(y)| \leq \omega_f(|x - y|) \leq \omega_f(C\sqrt{\e}) \leq \omega_f(C\sqrt{\e_0}) \leq \frac{m_0}4,
\end{align}
and
\begin{align*}
u(x) - v(y) \geq u(x) - v(y) - \frac{|x - y - \zeta|^2}{2\e} = \ell(\zeta) \geq \frac{m_0}2.
\end{align*}

Since the operator is smooth near $\nabla u \neq 0$, we are well within the classical viscosity solution framework for continuous operators.  In particular, we can use the now standard maximum principle for semicontinuous functions, see the exposition in \cite{CIL} or \cite{Giga_Book}.
Therefore there exist symmetric matrices $X, Y$ with $X \leq Y$, sequences $x_n \to x$, $y_n \to y$ and sequences of $C^2$ functions $\varphi_n$, $\psi_n$ such that $u - \varphi_n$ has a local maximum at $x_n$, $v - \psi_n$ has a local minimum at $y_n$, and
\begin{align*}
(x_n, u(x_n), \nabla\varphi_n(x_n), \nabla^2\varphi_n(x_n)) &\to (x, u(x), p, X),\\
(y_n, v(y_n), \nabla\psi_n(y_n), \nabla^2\psi_n(y_n)) &\to (y, u(y), p, Y).
\end{align*}
From the definition of viscosity solution we deduce
\begin{align*}
u - \frac1{|\nabla\varphi_n|} \trace \bra{\pth{I - \frac{\nabla\varphi_n \otimes \nabla\varphi_n}{|\nabla\varphi_n|^2}} \nabla^2 \varphi_n} \leq f\qquad \text{at $x_n$}\\
v - \frac1{|\nabla\psi_n|} \trace \bra{\pth{I - \frac{\nabla\psi_n \otimes \nabla\psi_n}{|\nabla\psi_n|^2}} \nabla^2 \psi_n} \geq f\qquad \text{at $y_n$}.
\end{align*}
In the limit $n\to \infty$, continuity and \eqref{fx-fy} yield
\begin{align*}
u(x) &- \frac1{|p|} \trace \bra{\pth{I - \frac{p \otimes p}{|p|^2}} X} \leq f(x) \\&\leq
f(y) + \frac{m_0}4  \leq v(y) - \frac1{|p|} \trace \bra{\pth{I - \frac{p \otimes p}{|p|^2}} Y} + \frac{m_0}4.
\end{align*}
On the other hand, $u(x) \geq v(y) + \frac{m_0}2$ and $X \leq Y$ imply
\begin{align*}
v(y) - \frac1{|p|} \trace \bra{\pth{I - \frac{p \otimes p}{|p|^2}} Y} + \frac{m_0}4  <
u(x) - \frac1{|p|} \trace \bra{\pth{I - \frac{p \otimes p}{|p|^2}} X}.
\end{align*}
We reach a contradiction.

\medskip
\textbf{Case 2.}
Since we cannot find any maximum of $\Phi_\zeta$ at which the ``gradient'' of $u$ nor $v$ is nonzero, we need to construct admissible faceted test functions for the faceted test in the definition of viscosity solution.

The extra parameter $|\zeta|$ provides a little bit of space to construct these faceted test functions. The following ``constancy'' lemma was proven in a more general settings in \cite[Lemma~7.5]{GG1}. We include the proof in our simple setting for completeness.

\begin{lemma}[Constancy]
Let $G \subset \Rn$ be a closed ball. If for all $\zeta \in G$ there exists $(x,y) \in \mathcal A(\zeta)$ such that $x - y - \zeta = 0$ then $\ell(\zeta)$ is constant on $G$.
\end{lemma}

\begin{proof}
Take $\zeta, \mu \in G$ and $(x, y) \in \mathcal A(\zeta)$ with $x - y -\zeta = 0$. In particular, $\ell(\zeta) = u(x) - v(y)$.
From the definition of $\ell$,
\begin{align*}
\ell(\mu) \geq u(x) - v(y) - \frac{|x - y - \mu|^2}{2\e} = \ell(\zeta)- \frac{|x - y - \mu|^2}{2\e}.
\end{align*}
Since $x - y - \zeta = 0$, we have
\begin{align*}
|x - y - \mu|^2 = |x - y - \zeta + \zeta - \mu|^2 = |\zeta - \mu|^2,
\end{align*}
yielding
\begin{align*}
\ell(\mu) - \ell(\zeta) \geq - \frac{|\zeta - \mu|^2}{2\e},
\end{align*}
and, by symmetry,
\begin{align*}
|\ell(\mu) - \ell(\zeta)| \leq \frac{|\zeta - \mu|^2}{2\e} \qquad \text{for all } \zeta, \mu \in G.
\end{align*}
We conclude that $\ell(\zeta) = \ell(\mu)$ for all $\zeta, \mu \in G$.
\end{proof}

Since in Case 2 we have $\mathcal B(\zeta) = \set0$ for all $\zeta \in G := \cl B_{\sqrt{m_0 \e}}$, we can choose $(\hat x, \hat y) \in \mathcal A(0)$ with $\hat x - \hat y = 0$ and the above lemma yields
\begin{align}
\label{lambda-order}
u(x) - v(y) \leq \ell(x - y) = \ell(0) = u(\hat x) - v(\hat x) \qquad \text{for } |x - y| \leq \sqrt{m_0 \e}.
\end{align}
Let us set $\lambda = \sqrt{m_0 \e}$. Defining
\begin{align*}
\eta_u = \sign (u - u(\hat x)) \qquad
\eta_v = \sign (v - v(\hat x)),
\end{align*}
the inequality \eqref{lambda-order} yields
\begin{align*}
\sup_{\cl B_{\lambda/2}(x)} \eta_u \leq \inf_{\cl B_{\lambda/2}(x)} \eta_v \qquad \text{for all $x$}.
\end{align*}

By the density result Theorem~\ref{th:facet-density}, there are admissible facets $\chi_u$ and $\chi_v$ satisfying
\begin{align*}
\sup_{\cl B_{\lambda/4}(x)} \eta_u \leq \chi_u(x) \leq \sup_{\cl B_{\lambda/2}(x)} \eta_u \leq \inf_{\cl B_{\lambda/2}(x)} \eta_v(x) \leq \chi_v \leq  \inf_{\cl B_{\lambda/4}(x)} \eta_v.
\end{align*}
Clearly $\chi_u = \chi_v = 0$ on $B_{\lambda/4}(\hat x)$.
% \todo{Does this need more explanation...}
By the comparison principle for the curvature operator, Proposition~\ref{pr:lambda-comparison},
\begin{align*}
\essinf_{B_{\lambda/4}(\hat x)} \Lambda_f[\chi_u] \leq \esssup_{B_{\lambda/4}(\hat x)}\Lambda_f[\chi_v],
\end{align*}
which implies
\begin{align}
\label{curvature-order}
\underline\Lambda_f[\chi_u](\hat x) \leq \overline\Lambda_f[\chi_v](\hat x).
\end{align}
Let us choose an admissible support function $\varphi_u \in Lip \cap [\chi_u]$.
Since $u$ is bounded and upper semicontinuous, we can multiply the positive part of $\varphi_u$ by a large positive constant, and the negative part of $\varphi_u$ by a small positive constant, if necessary, to guarantee that
\begin{align*}
u \leq \varphi_u(\cdot - h) + u(\hat x) \qquad \text{for $|h| < \tfrac \lambda8$.}
\end{align*}
Note that the equality is attained at $\hat x$ as $\varphi_u = 0$ on $B_{\lambda/4}(\hat x)$.
Therefore $\varphi_u$ is an admissible faceted test function for the viscosity solution test and 
\begin{align*}
u(\hat x) - \underline\Lambda_f[\varphi_u](\hat x) \leq 0.
\end{align*}
Similarly, we can find $\varphi_v \in Lip \cap [\chi_v]$ with
\begin{align*}
v(\hat x) - \overline\Lambda_f[\varphi_v](\hat x) \geq 0.
\end{align*}
Thus, recalling \eqref{curvature-order}, we have 
\begin{align*}
u(\hat x) \leq \underline\Lambda_f[\varphi_u](\hat x) \leq
\overline\Lambda_f[\varphi_v] \leq v(\hat x) \leq u(\hat x) - m_0 < u(\hat x),
\end{align*}
a contradiction. This finishes the proof of Theorem~\ref{th:visc-comparison-principle}.

% \subsection{Anisotropic case approximation of facet}
%
% \todo{should this be a subsection? but it should go to Section~\ref{se:max-monotone}}
% Now \ref{th:facet-density}

\subsection{Existence of solutions}

The existence of viscosity solutions is usually established using Perron's method: the largest subsolution of the problem is automatically a solution. However, the operation of taking a supremum of a class of viscosity solutions requires a stability property whose validity is unclear for the viscosity solutions considered here in dimensions $n \geq 2$.  In one dimension, Perron's method was used to construct viscosity solutions for \eqref{pde} in \cite{GGN}. This however requires a careful treatment of the nonlocal anisotropic curvature.

The main issue with the stability required for the supremum of subsolutions to be subsolution is the discontinuity of the value of $\Lambda_f[\psi]$ when a facet bends or breaks. In the standard proof of this stability, it is crucial to localize by replacing a test function $\varphi$ by another so that $u - \varphi$ can be assumed to have a \emph{strict} local maximum (or minimum). Due to the discontinuity of our operator $\Lambda_f$ with respect to such bending, this tool is not available.

The approach that was taken in \cite{MGP1,GP1} is via stability with respect to approximation by problems with regularized $\sigma$. % [MGP1,GP1]はOK？
In particular, we consider two ways of approximating crystalline $\sigma$:
\begin{itemize}
\item[(a)] $\sigma_m \in C^2(\Rn)$ with $a_m^{-1} I \leq \nabla^2 \sigma_m \leq a_mI$ for some $a_m > 0$, $\sigma_m$ is a decreasing sequence with $\sigma_m \to \sigma$ locally uniformly.
\item[(b)] $\sigma_m$ are anisotropies with $\sigma_m \in C^2(\Rn \setminus \set0)$ such that $\sigma_m^2$ is strictly convex and $\sigma_m \to \sigma$ locally uniformly.
\end{itemize}

However, for various reasons related to the regularity of the solutions of the approximating problems, we need to assume that $F$ does not depend on the $x$ variable, and consider solutions of the regularized problems
\begin{align}
\label{pde-m}
u_t + F(t, \divo \nabla \sigma_m(\nabla u) - f) = 0.
\end{align}
Since $\sigma_m$ are $C^2$ and convex, the classical theory of viscosity solutions applies, including the unique existence of solutions for given bounded continuous initial data.

We have the following stability result when approximating using (a), see \cite[Th.~4.1]{GP3} or \cite[Th.~8.1]{GP1}, which resembles the usual stability of viscosity solutions in the classical theory.
 Let $\limsup^*$ (resp.\ $\liminf_*$) denote the relaxed upper limit (lower) limit defined by % limit 重複？
\begin{align*}
	\textstyle(\limsup^* u_m)(x,t) &= \limsup_{m\to\infty} \left\{ u_k(y,s) \bigm| |x-y|+|t-s|+1/k<1/m \right\} \\
	\textstyle(\liminf_* u_m)(x,t) &= -\left(\limsup(-u_m)\right)(x,t).
\end{align*}
for a sequence of functions $\{u_m\}$ on $\mathbf{R}^n\times[0,\infty)$.

\begin{theorem}
\label{th:stability}
Let $\sigma$ be a crystalline anisotropy and assume that $F$ does not depend on the $x$ variable and $f \in C(\Rn \times \R)$ is Lipschitz continuous in space, uniformly in time. If $\set{u_m}$ is a locally bounded sequence of viscosity solutions of \eqref{pde-m} with $\sigma_m$ as in (a) above, then $\limsup^*_{m\to \infty} u_m$ is a viscosity subsolution of \eqref{pde}, and $\liminf_{*m\to\infty} u_m$ is a viscosity supersolution of \eqref{pde}.  
\end{theorem}

The main idea of the proof of Theorem~\ref{th:stability} is inspired by the perturbed test function method due to Evans \cite{E}. Let us for simplicity assume that $f \equiv 0$. The crystalline mean curvature, or specifically the operator $\Lambda_0[\psi]$ is nonlocal on the facets of $\psi$. In contrast, the elliptic operators $\divo \nabla \sigma_m(\nabla \psi)$ are local and they are in fact zero on the facets of $\psi$. To recover the nonlocal information in the limit $m \to \infty$, we \emph{perturb} the test function $\psi$ using a sequence $\psi_m$ of uniformly converging $C^2$ functions $\psi_m \to \psi$, such that $\divo \nabla \sigma_m(\nabla \psi_m)$ approximates the value of $\Lambda_0[\psi]$ in a suitable sense at the contact point.

Such approximation is available via the resolvent problem for the regularized energy $\mathcal E_m$, with $\sigma$ replaced by $\sigma_m$ in \eqref{tv-energy}. For given $a > 0$ and $\psi \in L^2(\Tn)$, there exists a unique solution $\psi_{a,m} \in L^2(\Tn)$ of
\begin{align*}
\psi_{a,m} + a \partial \mathcal E_m(\psi_{a,m}) \ni \psi.
\end{align*}
If $\psi \in Lip(\Tn)$, then $\psi_{a,m}$ is Lipschitz uniformly in $a$ and $m$ by the comparison principle like Proposition~\ref{pr:resolvent-comp-principle} and translation invariance of the operator, and in fact it is $C^2$ by the elliptic regularity theory. Since $\mathcal E_m$ Mosco-converges to $\mathcal E$ (see \eqref{mosco} for the definition), we have a convergence of the resolvent solutions $\psi_{a, m} \to \psi_a$ in $L^2(\Tn)$, see \cite{Attouch}, and hence uniformly by the uniform Lipschitz continuity. Using Proposition~\ref{pr:resolvent-approximation}, we can deduce that $\psi_{a, m}$ uniformly approximate $\psi$ as $a \to 0$ and then $m \to \infty$. Functions $\psi_{a, m}$ are used to build test functions for the regularized problem, and allow us to deduce that $\limsup^*_{m\to \infty} u_m$ is a viscosity subsolution of \eqref{pde}, and $\liminf_{*m\to\infty} u_m$ is a viscosity supersolution of \eqref{pde}.  

Approximation using (b) is relevant when considering the crystalline mean curvature flow as a limit of a smooth anisotropic mean curvature flow. To prove the stability for (b), we use the stability Theorem~\ref{th:stability} to approximate each $\sigma_m$ by a sequence of $C^2$ functions $\sigma_{m,\delta}$ and therefore we need to know that a given solution $u_m$ can be approximated by a sequence of solutions $u_{m, \delta}$ with this anisotropy. This is known for example when $u_m$ have continuous bounded initial data. We have the following stability result, \cite[Th.~4.4]{GP3}.

\begin{theorem}
\label{th:stability-b}
Let $\sigma$, $F$ and $f$ be as in Theorem~\ref{th:stability}. Let $T > 0$ and let $u_m$ be a locally bounded sequence of viscosity solutions of \eqref{pde-m} on $\Rn \times (0, T)$ with $\sigma_m$ as in (b) with initial data $u_m(\cdot, 0) = u_{0, m}$, where $u_{0, m} \in C(\Rn)$ are uniformly bounded. Then $\limsup^*_{m\to \infty} u_m$ is a viscosity subsolution of \eqref{pde}, and $\liminf_{*m\to\infty} u_m$ is a viscosity supersolution of \eqref{pde}.  
\end{theorem}

Now with the stability with respect to approximation by the regularized problems established, and the comparison principle discussed in Section~\ref{se:comparison}, we can follow the standard idea to show existence of \eqref{pde} for given initial data when the operator $F$ does not depend on the $x$ variable.
For given bounded uniformly continuous initial data, we take $u_m$ solutions of the regularized problem with initial data $u_0$ from Theorem~\ref{th:stability-b}. By using barriers at $t = 0$, we can show that the limits satisfy 
\begin{align*}
\mathop{\liminf\nolimits_*}\limits_{m\to\infty} u_m\big|_{t =0} \geq u_0, \qquad \mathop{\limsup\nolimits^*}\limits_{m\to \infty} u_m\big|_{t =0} \leq u_0.
\end{align*}
From the comparison principle for \eqref{pde} we immediately have 
\begin{align*}
\mathop{\limsup\nolimits^*}\limits_{m\to \infty} u_m\leq \mathop{\liminf\nolimits_*}\limits_{m\to\infty}.
\end{align*}
This implies that both limits are equal, the convergence is locally uniform, and the limit is a viscosity solution of \eqref{pde}.

If the forcing $f$ depends on $x$, there is an additional difficulty that the comparison principle for semi-continuous solutions is not available, see \cite[Sec.~3]{GP3}. The comparison principle established in \cite{GP3} requires that at least one of the solutions is continuous. Fortunately, for operators $F$ that come from the level set formulation of geometric motions one can prove uniform Lipschitz bounds in space and uniform H\"older bounds in time on the approximating sequence $u_m$ for Lipschitz initial data $u_0$, see \cite[Sec.~5]{GP3}. Therefore the convergence $u_m$ is locally uniform for subsequences and the limits are a priori continuous. In particular, the restricted comparison principle applies and existence of solutions can be established. We have the following existence theorem, \cite[Th.~1.1]{GP3}.

\begin{theorem}
\label{th:existence-with-f}
Assume that $g \in C(\mathcal S^{n-1} \times \R)$ is Lipschitz continuous in the second variable uniformly in the first variable and non-decreasing in the second variable, $\sigma$ is a crystalline anisotropy and $f \in C(\Rn \times \R)$ is Lipschitz continuous in space uniformly in time. Then there is a unique global-in-time level set flow to
\begin{align*}
V = g(\nu, \kappa_\sigma + f(x,t))
\end{align*}
when the initial hypersurface is compact.
\end{theorem}

\begin{remark} \label{GGLE} % Remark 6.12 (11/24)
If $f$ is constant, then the global Lipschitz continuity of $F$ is unnecessary \cite{GP1}, \cite{GP2}. In particular, it applies to \eqref{2EXP}. In the case $n=2$, it applies to a general anisotropy under a slightly different definition of a solution \cite{GG4}. Note that the level set equation for $V=\kappa_\sigma$ is
\[
	u_t = |\nabla u| \operatorname{div}\nabla\sigma(\nabla u)
\]
so that each level set of $u$ moves by $V=\kappa_\sigma$. The level set flow is a level set of a viscosity solution $u$. Its uniqueness (up to fattening) is guaranteed by the comparison principle and an invariance under a change of the depended variable $u$ (representing its level sets) together with Lemma~\ref{le:Lipschitz-ordering}. This procedure is standard for a level set flow; see e.g.\ \cite{G06}. The terminology of the level set flow here is different from that in Section \ref{AD}.
\end{remark}

%%%%%%%%%% Subsection 6.4
\subsection{Convergence of various approximations} \label{AV4}

It is well-known that the solution of the mean curvature flow equation is approximated by that of the Allen-Cahn equation;
 see \cite{DSch}, \cite{BrK}, \cite{XChen}, \cite{ESS}.
 Anisotropic version of the Allen-Cahn equation is introduced by \cite{MWBCS}, which is an $L^2$-gradient flow of
\[
	F_\varepsilon(v) = \int_{\mathbb{R}^n} \left\{ \frac{1}{2} \sigma (\nabla v)^2
	+ \frac{1}{\varepsilon^2} \left( W(v) - \varepsilon\lambda F(v) \right) \right\} dx. 
\]
Here, $W(v)$ is a double-well potential typically $W(v)=(v^2-1)^2/2$ and $F(v)=Cv$ with constant $C$ for simplicity.
 The parameter $\lambda>0$ should be chosen in a suitable way.
% In the case $W(v) = (v^2-1)^2/2$, it should be $\lambda=2/3$.
 In an explicit form, the anisotropic Allen-Cahn equation reads
\begin{align}
\label{allen-cahn}
	\beta(\nabla v)v_t - \operatorname{div} \left( \sigma(\nabla v) \zeta(\nabla v) \right)
	+ \frac{1}{\varepsilon^2} \left( W'(v) - \varepsilon\lambda C \right) = 0 
\end{align}
with some kinetic coefficient $\beta>0$ which is positively one-homogeneous;
 here $\zeta(p)=\nabla_p\sigma(p)$.
 For a given closed interface $\Gamma_0$, we consider a function $v^\varepsilon_0$ which converges to $-1$ in an open set surrounded by $\Gamma_0$ and to $1$ outside the closure of the open set.
 The way of convergence is taken in a suitable way.
% 原稿 6-2
 It is expected that the solution of the anisotropic Allen-Cahn equation with initial data $v^\varepsilon_0$ converges to $1$ inside an open set surrounded by $\Gamma_t$ and $-1$ outside $\Gamma_t$ and this open set, where $\Gamma_t$ is a (generalized) solution to the interface equation
\[
	\beta(\mathbf{n})V = \sigma(\mathbf{n}) (\kappa_\sigma - C).
\]
% 原稿 6-4
(Here $\lambda$ should be taken as $\lambda=2/3$ if $W(v)=(v^2-1)^2/2$.) 
 Formal asymptotic analysis is carried out by \cite{MWBCS}, \cite{WS} and \cite{BP95}, which derives the interface equation.
 For smooth anisotropy with $\beta \equiv 1$, the convergence is established by \cite{ElS1} when the solution of the interface equation is smooth,
 here $W$ is taken as double-obstacle type, for example, $W(v)=1-v^2$ in $|v| \leq 1$ and $W(v)=\infty$ for $|v|>1$.
 This result is extended when $\Gamma_t$ is a generalized solution (a level-set solution allowing fattening).
 In \cite{GOS} it is shown that such convergence is uniform in $\sigma$ provided that the Frank diagram $F_\sigma$ is bounded by a ball both from inside and outside.
 It does not depend on regularity of $\sigma$.

For crystalline $\sigma$ under $\beta \equiv 1$, the convergence with some rate is established for planar crystalline flow \cite{BGN}.
 It is somewhat extended to higher dimension for a special class of solutions of the interface equation;
 its existence is not clear \cite{BN}.
 Several explicit examples of convergence are given by \cite{TC}.
 One of the reasons why $\beta\equiv 1$ is assumed is that the notion of solutions for the Allen-Cahn equation is unclear.
 Maybe a viscosity approach will resolve this issue.

Since our solution for the interface equation for crystalline $\sigma$ is obtained as a limit of smoother problems as in the previous subsection, combining uniform convergence with respect to $\sigma$ we are able to prove the convergence as $\varepsilon\to 0$ by approximating $\beta$ and $\sigma$ by smooth function;
 see \cite[Theorem 2.4]{GOS}.
 Note that in two dimensional case, the stability was proved in \cite{GG4}.

% 原稿 6-5
Another typical way to approximate a solution is what is called Chambolle's scheme introduced by \cite{Cha}.
 We here give its anisotropic version \cite{CC}, \cite{CN07}.
 We consider
\[
	V = M (\nu) \kappa_\sigma.
\]
We set the support function of the polar of $1/M$ (Frank diagram of $M$) by $M^0$, i.e.,
\[
	M^0 (x) :=\sup \left\{ x \cdot p \bigm|
	|p| \leq 1/M \left( p/|p| \right) \right\}.
\]
Here $M$ is assumed to be positive on $S^{n-1}$.
 The function $M^0$ is convex, positively $1$-homogeneous in $\mathbb{R}^n$ and it is positive outside the origin.
 However, it may not satisfy the symmetry $M(x)=M(-x)$ so that $\operatorname{dist}_{M_0}(x,y)=M^0(x-y)$ is a non-symmetric distance.
 For a given bounded set $E_0$ in $\mathbb{R}^n$, let $d_{M^0}(x,E_0)$ denote its anisotropic signed distance, i.e.,
\[
	d_{M^0} (x,E_0) := \operatorname{dist}_{M^0}(x,E_0) - \operatorname{dist}_{M^0}(x,E^\compl_0), \quad
	x \in \mathbb{R}^n,
\]
where
\[
	\operatorname{dist}_{M^0}(x,E_0) := \inf_{y \in E_0} \operatorname{dist}_{M^0}(x,y).
\]
We next consider an energy functional of the form
\[
	J_h (v, E_0) = \int_\Omega \left\{ \sigma (\nabla v)
	+ \frac{1}{2h} | v-d_{M^0}|^2 \right\} dx
\]
for a domain $\Omega$ containing $E_0$ with a small parameter $h>0$.
 This value is finite in $L^2(\Omega)\cap BV(\Omega)$ so we regard $J_h$ as a lower semicontinuous convex functional on $L^2(\Omega)$ by interpreting its value equal to $\infty$ on $L^2(\Omega)\backslash BV(\Omega)$.
% 原稿 6-7
 It admits a unique minimizer $w=\operatorname{argmin} J_h$.
 We introduce the operator $T_h$ as
 \[
	T_h(E_0) = \left\{ x \in \mathbb{R}^n \mid w(x) \leq 0 \right\}.
\]
 An approximate flow is defined by applying the above step iteratively as
 \begin{align}
 \label{discrete-flow}
	E^h(t) = T^{\lfloor t/h\rfloor}_h (E_0),
 \end{align}
where $\lfloor s \rfloor$ denotes its integral part of $s>0$.
 We expect that $E^h$ converges to the level-set solution of $V=M(\nu)\kappa_\sigma$ as $h\to 0$, for example, in the Hausdorff distance sense uniformly in $t\in[0,T]$ with finite $T$.
 Let us give a very heuristic argument.
 We consider the isotropic case $V=\kappa$ so that $M=1$ and $\sigma(p)=|p|$.
 Then the minimizer $w$ satisfies the resolvent equation
\[
	\frac{w-d}{h} - \divo \frac{\nabla w}{|\nabla w|} = 0,
\]
where $d$ denotes the Euclidean signed distance of $E_0$.
 This is the implicit Euler scheme for the total variation flow.
  The signed distance function satisfies $|\nabla d| = 1$ on the interface $\Gamma_t$ so $V \approx \frac{w - d}h$ and it is expected that the zero level of $w$ approximates the solution $\Gamma_t$.

The isotropic case of this scheme was first introduced in \cite{Cha}, which gives a monotone way to realize the time discrete scheme proposed by \cite{ATW};
 see also \cite{LS}.
 In \cite{Cha} $L^1$ convergence: $E^h(t)\to E(t)$ on $[0,T]$, where $E(t)$ is the level set solution of $V=\kappa$ (starting from a closed set $E_0$ with $E_0=\overline{\operatorname{int} E_0}$) was established provided that no fattening phenomena occur.
% 原稿 6-8
 Its anisotropic extension is done by \cite{CC} in the case when $E_0$ is convex and compact under the assumption that $\sigma/M$ is constant on $S^{n-1}$;
 see \cite{CN07} for non-convex initial data;
 here anisotropy is assumed to be smooth.
 In \cite{BCCN} for a non-smooth $\sigma$ including crystalline, a unique solution for $V=\sigma\kappa_\sigma$ is constructed when $E_0$ is convex and compact by defining a solution by the distance function.
 For smooth anisotropy for a bounded nonconvex initial data, the Hausdorff convergence is proved in \cite{EGI}, where they prove locally uniform convergence of an associated function
\[
	u^h (x,t) = \left( S^{[t/h]}_h u_0 \right) (x)
\]
with
\[
	(S^h u_0) (x) = \sup \left\{ \mu\in\mathbb{R} \bigm|
	x \in T_h \left( \left\{ x \in \mathbb{R}^n \mid 
	u_0(x) \geq \mu \right\} \right) \right\}. 
\]
Although it is remarked in \cite{EGI} and \cite{CN07}, the case when $\sigma$ and $M$ are unrelated is not discussed in detail.
 In \cite{Ik} a proof based on the distance function is given for several choices of $\sigma$ and $M$ and general initial data not necessarily compact mostly for smooth case.
 However, it is also shown in \cite{Ik} that if the solution of crystalline anisotropy has a stability property we are able to prove the convergence of Chambolle's scheme by approximating $M$ and $\sigma$.
% 原稿 6-9
 Since at that time, the stability is only available in two dimensional case \cite{GG4}, convergence result in \cite{Ik} looks limited but it applies to general dimension at least for purely crystalline anisotropy since the stability holds for general dimension as discussed in the previous subsection.
 The reason why $M$ and $\sigma$ are approximated by a smoother one in Chambolle's scheme in \cite{Ik} seems to avoid analysis for the resolvent equation for non-smooth $M$ and $\sigma$, so it seems that it is not substantial.

In the next section we discuss a notion of solutions based on distance functions to the evolving surface that can be showed to be the limits of the discrete evolutions \eqref{discrete-flow} given by Chambolle's scheme, see Theorem~\ref{th:atw-approximation}.

%%%%%%% Section 7
\section{Approach by distance functions} \label{AD}

In this section we discuss an alternative approach to defining a notion of solutions of the crystalline mean curvature flow that appeared in a series of papers by Chambolle, Morini, Novaga and Ponsiglione \cite{CMP,CMNP,CMNP_APDE}. % [CMP,CMNP,CMNP_APDE]はOK？
The main idea is to require that the distance function to an evolving set is a sub/supersolution of a related partial differential equation in the sense of distributions.

This approach applies to a form of the crystalline mean curvature flow that is \emph{linear} in the curvature term:
\begin{align}
\label{cmnp-flow}
V = M(\nu) (\kappa_\sigma - f).
\end{align}
However, both $\sigma$ and $M$ can be arbitrary anisotropies, not necessarily crystalline. For simplicity, we will assume that both $\sigma$ and $M$ are even, that is, $\sigma(p) = \sigma(-p)$ and $M(p) = M(-p)$ for all $p \in \Rn$. This restriction however does not appear to be essential. Moreover, the initial data $E^0$ can be an unbounded closed set, and the forcing term needs to be only $f \in L^\infty(\Rn \times (0, T))$ with $f(\cdot, t)$ Lipschitz uniformly in $t$.

The distance function must be adapted to the mobility $M$. As in \cite{CMNP} for any norm $\eta$ we denote
\begin{align*}
\dist^\eta(x, E) := \inf_{y \in E} \eta(x - y), \qquad E \subset \Rn.
\end{align*}
Note that $\dist^\eta(x, \emptyset) = +\infty$.

Let $E_n \subset \Rn$ be a sequence of closed sets and $E \subset \Rn$ a closed set. We say that $E_n$ converges to $E$ in \emph{Kuratowski sense}, and write $E_n \stackrel{\mathcal{K}}\to E$, if
$\dist^\eta(\cdot, E_n) \to \dist^\eta(\cdot, E)$ locally uniformly in $\Rn$ for some norm $\eta$. It is easy to see that if this converges for one norm, it converges for all norms.

The following definition appeared in \cite{CMNP}.

\begin{definition}
\label{de:flow}
Let $E^0 \subseteq \Rn$ be a closed set. Let $E$ be a closed set in $\Rn \times [0, +\infty)$ and for each $t \geq 0$ define $E(t):= \set{x \in \Rn: (x,t) \in E}$. We say that $E$ is a \emph{superflow} of \eqref{cmnp-flow} with initial datum $E^0$ if:
\begin{itemize}
\item[(a)] $E(0) \subseteq E^0$,
\item[(b)] $E(s) \stackrel{\mathcal{K}}{\to} E(t)$ as $s \nearrow t$ for all $t > 0$,
\item[(c)] If $E(t) = \emptyset$ for some $t \geq 0$, then $E(s) = \emptyset$ for all $s > t$.
\item[(d)] Set $T^* := \inf\set{t > 0: E(s) = \emptyset \text{ for } s \geq t}$, and
\begin{align*}
d(x,t) := \dist^{M^\circ}(x, E(t)) \qquad \text{for all } (x, t) \in \Rn \times (0, T^*) \setminus E.
\end{align*}
Then there exists $K > 0$ such that the inequality
\begin{align}
\label{dist-flow-inequality}
d_t \geq \divo z + f - K d
\end{align}
holds in the distributional sense in $\Rn \times (0, T^*) \setminus E$ for a suitable $z \in L^\infty(\Rn \times (0, T^*))$ such that $z \in \partial \sigma(\nabla d)$ a.e., $\divo z$ is a Radon measure in $\Rn \times(0, T^*)\setminus E$, and
\begin{align*}
(\divo z)^+ \in L^\infty(\set{(x, t) \in \Rn \times(0, T^*): d(x,t) \geq \delta}) \qquad \text{for every $\delta \in (0, 1)$}.
\end{align*}
\end{itemize}

\medskip

An open set $A \subset \Rn \times [0, +\infty)$ is a \emph{subflow} of \eqref{cmnp-flow} with initial datum $E^0$ if $A^\compl$ is a superflow of \eqref{cmnp-flow} with $f$ replaced by $-f$ and with initial datum $(\interior E^\circ)^\compl$.

A closed set $E \subset \Rn \times [0, +\infty)$ is a \emph{solution} of \eqref{cmnp-flow} with initial datum $E^0$ if it is a superflow and if $\interior E$ is a subflow, both with initial datum $E^0$.
\end{definition}

%\todo{I am not sure about the correct definition of the convergence in Kuratowski sense and the meaning of (b) and (c). Usually Kuratowski convergence is defined for \emph{non-empty} sets. But that's not what is in \cite{CMNP}. According to that definition, even if one of $E_n$ is empty, then $E$ must be empty. This would trivially imply (c) from (b). Maybe it would be better to write (b) only for $t \leq T^*$? Note that $E(T^*) \neq \emptyset$ if $E(0) \neq \emptyset$.}
The condition (b) is meant to prevent a possibility that $E$ expands discontinuously, for example a bubble closing up, which cannot be ruled out by \eqref{dist-flow-inequality}.

Note that $K$ is related to the Lipschitz constant of $f$ with respect to the distance induced by $M$.
In fact, in the smooth case $\sigma$, $M$, $M^\circ \in C^2(\Rn \setminus \set0)$, $f$ continuous, then $E$ is a superflow in the sense of Definition~\ref{de:flow} if and only if $-\one_E$ is a viscosity supersolution of the level set equation
\begin{align*}
u_t = M(\nabla u)(\divo \nabla\sigma(\nabla u) + f),
\end{align*}
in $\Rn \times (0, T^*]$; see \cite[Lemma~2.6]{CMNP}. For viscosity supersolution $-\one_E$ we can take $K = Lip(f)$ in \eqref{dist-flow-inequality}.

We cannot in general expect uniqueness of a solution in the sense of Definition~\ref{de:flow} since there may occur fattening phenomena.
The comparison principle between superflows and subflows requires a positive distance between initial data and therefore by itself does not provide uniqueness. The following theorem appeared in \cite{CMNP}.

\begin{theorem}[{c.f. \cite[Theorem~2.7]{CMNP}}]
\label{th:flow-comparison}
Let $E$ be a superflow with initial datum $E^0$ and $F$ be a subflow with initial datum $F^0$ in the sense of Definition~\ref{de:flow}. If $\dist^{M^\circ}(E^0, (F^0)^\compl) =: \delta > 0$, then
\begin{align*}
\dist^{M^\circ}(E(t), F(t)^\compl) \geq \delta e^{-Kt} \qquad \text{for all } t\geq0,
\end{align*}
where $K > 0$ is the constant in \eqref{dist-flow-inequality} for both $E$ and $F$.
\end{theorem}

To obtain uniqueness, \cite{CMNP} introduce the associated level-set flow.

\begin{definition}
\label{de:level-set-flow}
Let $u^0$ be a uniformly continuous function on $\Rn$. We say that a lower semicontinuous function $u : \Rn \times[0, \infty) \to \R$ is a \emph{level-set supersolution} corresponding to \eqref{cmnp-flow} with initial datum $u^0$ if $u(\cdot, 0) \geq u^0$ and if for a.e. $\lambda \in \R$ the closed sublevel set $\set{u \leq \lambda}$ is a superflow of \ref{de:flow} in the sense of Definition~\ref{de:flow} with initial datum $\set{u_0 \leq \lambda}$.

Similarly, an upper semicontinuous function $u: \Rn \times [0, \infty) \to \R$ is a \emph{level-set subsolution} corresponding to \eqref{cmnp-flow} with initial datum $u^0$ if $-u$ is a level-set supersolution in the previous sense, with initial datum $-u_0$ and with $f$ replaced by $-f$.

A continuous function $u: \Rn \times [0, \infty) \to \R$ is a \emph{level-set solution} corresponding to \eqref{cmnp-flow} with initial datum $u^0$ if it is both a level-set supersolution and level-set subsolution with the same initial datum. 
\end{definition}
Our terminology here is different from that in \cite[Chapter 5]{G06}.
 A superflow here is called a set-theoretic supersolution in \cite{G06}.
 A level set supersolution in \cite{G06} is a superflow given by sublevel set of a continuous level-set supersolution.
 
The following comparison theorem was proven in \cite{CMNP}.

\begin{theorem}[{c.f. \cite[Theorem~2.5]{CMNP}}]
\label{th:level-set-comparison}
Let $u^0$, $v^0$ be uniformly continuous functions on $\Rn$ and let $u$, $v$ be respectively a level-set subsolution with initial datum $u^0$ and a level-set supersolution with initial datum $v^0$, in the sense of Definition~\ref{de:level-set-flow}. If $u^0\leq v^0$ then $u \leq v$.
\end{theorem}

The main idea of going from Theorem~\ref{th:flow-comparison} is that due to the uniform continuity, the superflow $\set{u \geq \lambda_1}$ and the superflow $\set{v \leq \lambda_2}$ for $\lambda_1 > \lambda_2$ are initially separated by a positive distance so that Theorem~\ref{de:level-set-flow} applies.

It remains to establish the existence of the level-set solutions. In the smooth case, the notion in the sense of Definition~\ref{de:level-set-flow} is equivalent to the standard notion of viscosity solutions. In general, an approximation by a sequence of smooth anisotropies $M_n$, $\sigma_n$ and a stability result established in \cite[Theorem~2.8]{CMNP_APDE} allows to construct a level-set solution as the limit of viscosity solutions. However, the stability result requires that the approximating sequence $M_n$ is \emph{uniformly $\sigma_n$ regular}, that is, it is required that there exists $\e_0 > 0$ such that
\begin{align*}
M_n = M_{0,n} + \e_0 \sigma_n
\end{align*}
for all $n$ for some convex functions $M_{0, n}$. Or equivalently, the Wulff shapes $W_{M_n}$ is must satisfy interior $W_\sigma$ condition uniformly in $n$.
Intuitively, if $M$ is $\sigma$ regular the level sets of $d := \dist^{M^\circ}(\cdot, E)$ have $\sigma$-curvature bounded by $C/d$ for some constant $C >0$.

In particular, this stability result is only able to construct level-set solutions in the sense of Definition~\ref{de:level-set-flow} if $M$ is $\sigma$-regular. Therefore the authors of \cite{CMNP_APDE} propose a definition of a solution via approximation.

\begin{definition}[{c.f. \cite[Definition~3.6]{CMNP_APDE}}]
\label{de:sol-via-approximation}
A continuous function $u: \Rn \times [0, \infty) \to \R$ is a \emph{solution via approximation} to the level set flow corresponding to \ref{de:flow} with initial datum $u^0$ if there exists a sequence $\set{M_n}$ of $\sigma$-regular mobilities such that $M_n \to M$ and, denoting $u_n$ the unique level-set solution of \ref{de:flow} with mobility $M_n$ and initial datum $u^0$, we have $u_n \to u$ locally uniformly in $\Rn \times [0, \infty)$.
\end{definition}

Such a solution always exists and is independent of the approximating sequence $\set{M_n}$.

\begin{theorem}[{c.f. \cite[Theorem~3.7]{CMNP_APDE}}]
\label{th:sol-by-approximation-existence}
Let $u^0$ be a uniformly continuous function on $\Rn$.
There exists a unique solution $u$ in the sense of Definition~\ref{de:sol-via-approximation} with initial datum $u^0$.
\end{theorem}

Alternatively, the level-set flow solutions in Definition~\ref{de:level-set-flow} and the solutions via approximation Definition~\ref{de:sol-via-approximation} can be constructed using a minimizing movement scheme; see \cite{CMNP} and the discussion in Section~\ref{AV4}. To be more precise, for given initial data $u^0$ one can define the level set discrete evolution $u_h: \Rn \times \R \to \R$ as
\begin{align*}
u_h(x, t) := \inf\set{\lambda \in \R: x \in E_{\lambda,h}(t)},
\end{align*}
where $E_{\lambda, h}(t)$ is the discrete evolution given by Chambolle's scheme in \eqref{discrete-flow} with $E_0 := \set{u_0 \leq \lambda}$. The following result was proved in \cite[Th.~5.7]{CMNP}.

\begin{theorem}
\label{th:atw-approximation}
Let $u^0$ be a uniformly continuous function on $\Rn$. The unique solution of \eqref{cmnp-flow} in Theorem~\ref{th:sol-by-approximation-existence} is the locally uniform limit in $\Rn \times [0, +\infty)$ as $h \to 0^+$ of the level set minimizing movements $u_h$.
\end{theorem}

Here are the types of solutions that are currently available if velocity law is linear in curvature, i.e., of the form \eqref{cmnp-flow}, and the initial data $u^0$ is constant outside of a bounded ball:

\begin{itemize}
\item $\sigma$ smooth, $M$ arbitrary: classical viscosity solutions \cite{CGG}
\item $\sigma$ purely crystalline, $M$ arbitrary: crystalline viscosity solutions \cite{GP3}
\item $\sigma$ arbitrary, $M$ is $\sigma$-regular: level-set solutions \cite{CMNP,CMNP_APDE} % [CMNP,CMNP_APDE]はOK？
\item $M$, $\sigma$ arbitrary: solutions via approximation \cite{CMNP,CMNP_APDE} % [CMNP,CMNP_APDE]はOK？
\end{itemize}

If the velocity law is not linear in curvature, only the viscosity solutions are currently available. On the other hand, the latter two notions apply also to general uniformly continuous initial data.

If the law is linear in the curvature, $\sigma$ is purely crystalline and $u^0$ is constant outside of a large ball, so that the notions of crystalline viscosity solutions and solutions via approximation both apply, they also give the same solutions. This can be seen by applying stability properties under the approximation of $\sigma$ by smooth $\sigma_n$. 

\begin{center}
\begin{tabular}{|l|c|c|}
\hline
Notion of solutions & $\sigma$ & $M$\\
\hline
classical viscosity solutions \cite{CGG} & $C^2$ & any+\\
crystalline viscosity solutions \cite{GP3}& purely crystalline& any+\\
level-set solutions  \cite{CMNP,CMNP_APDE}&any& $\sigma$-regular\\ % [CMNP,CMNP_APDE]はOK？
solutions via approximation \cite{CMNP,CMNP_APDE}& any & any\\ % [CMNP,CMNP_APDE]はOK？
\hline
\end{tabular}
\end{center}
any+: allows any nonnegative function, not just anisotropies.

%%%%%%% Section 8
\section{Some numerics} \label{SN}

The study of the crystalline mean curvature flow using numerical methods goes back to the seminal work of J.~E.~Taylor, who developed the \emph{crystalline algorithm} based on the polygonal flow in Section~\ref{PF} in both two and three dimensions \cite{T0,Taylor3Dsim}, including spiral growth in two dimensions and observation of possible facet breaking in three dimensions. % [T0,Taylor3Dsim]はOK？
Examples of facet breaking were further numerically investigated in \cite{NP}.

In higher dimension, the crystalline algorithm is limited to evolutions in which topological changes or facet breaking do not occur, or the result of facet breaking can be computed and produces facets with somewhat simple topology.
In a more general situation, the level set method is popular to track the evolution past singularities. However, the level set equation for the crystalline mean curvature is rather singular and so its direct use is limited.

An anisotropic version of the Allen--Cahn equation was used to approximate the crystalline mean curvature flow in three dimensions in \cite{PP}. In particular, an example of facet bending was demonstrated.

A.~Chambolle reformulated the minimizing movements scheme of \cite{ATW} and \cite{LS} for anisotropic mean curvature flow in terms of the signed distance function as the level set function and proposed a numerical method to solve the resulting minimization problem in \cite{Cha} (see Section~\ref{AV4} for more details). In \cite{OOTT} it was observed that the minimization problem in Chambolle's scheme can be solved efficiently using the split-Bregman method for the total variation minimization \cite{GO}, and presented computational results for two dimensional crystalline mean curvature flow. However, the method easily generalizes to any dimension; see \cite{Po} computational results for three dimensional evolutions.

It is also possible to regularize the crystalline anisotropy and consider the almost-crystalline but smooth anisotropic mean curvature flow, with many numerical methods available. 
One way to approximate the smooth anisotropic mean curvature flow numerically is using the Allen--Cahn equation \eqref{allen-cahn} with double obstacle potentials (see \S \ref{AV4})  \cite{BGN_NM,BGN_IFB,BGN_ADV}. % [BGN_NM,BGN_IFB,BGN_ADV]はOK？
For estimates of the Allen--Cahn approximation see for example \cite{ElPS}.
Another possibility is to track the evolving surface explicitly using a parametric approach \cite{Dziuk, BGN_ZAMM,BGN_IMA}. % [Dziuk, BGN_ZAMM,BGN_IMA]はOK？

For an extensive review of the early numerical approaches see \cite{DDE}.

%%%%%%% Section 9
\section{Volume-preserving and fourth-order problems} \label{FO}

%%%%%%%%%% Subsection 9.1
\subsection{Volume preserving flow} \label{FO1}

In many applications it is important to impose that the volume of the set surrounded by the evolving surface is preserved. Examples include crystal growth, droplet motion and bubbles. A common way to achieve this for the mean curvature flow is to add a Lagrange multiplier to the velocity law. Consider a family of hypersurfaces $\set{\Gamma_t}$ with $\Gamma_t = \partial \Omega_t$ for some evolving set $\set{\Omega_t}$ that evolves with the velocity law
\begin{align*}
V = g(\nu, \kappa_\sigma + \lambda)\qquad \text{on $\Gamma_t = \partial \Omega_t$.}
\end{align*}
Here the forcing term $\lambda = \lambda(t)$ is chosen so that
\begin{align*}
|\Omega_t| = |\Omega_0| \qquad t \geq 0.
\end{align*}
If $\set{\partial \Omega_t}$ is sufficiently smooth, we have
\begin{align*}
\frac{d}{dt} |\Omega_t| = \int_{\partial \Omega_t} V \;d\mathcal{H}^{n-1},
\end{align*}
and $\lambda(t)$ must be chosen so that
\begin{align*}
\int_{\partial \Omega_t} g(\nu, \kappa_\sigma + \lambda(t)) \;d\mathcal{H}^{n-1} = 0, \qquad t \geq 0.
\end{align*}
In general, the regularity of $\lambda$ is not clear.

The problem has been studied in the case of linear dependence on $\kappa_\sigma$,
\begin{align*}
V = M(\nu)(\kappa_\sigma + \lambda).
\end{align*}
For convex initial data, the existence of solutions and convergence to the Wulff shape $W_\sigma$ was shown in \cite{And01} for smooth $\sigma$, and in \cite{BCCN09} for nonsmooth $\sigma$, generalizing the classical result for the isotropic mean curvature flow of \cite{Hui}. 
For a planar crystalline flow, a similar result has been proved by \cite{Ya02}.
 Moreover, it approximates corresponding smooth problems as proved in \cite{UYa}.

For general initial data, the existence of solutions still remains mostly open. In the isotropic case, global existence results are available under a certain energy convergence assumption \cite{MSS16,LS17}. % ［MSS16,LS17］はOK？

One can also consider initial data for which topological changes do not occur like star-shaped sets in the isotropic case \cite{KK20}  or sets that satisfy a certain reflection symmetry property in the anisotropic case including some crystalline flow \cite{KKP}.

%%%%%%%%%% Subsection 9.2
\subsection{Fourth-order problem} \label{FO2}

% 原稿9.2　1/6
We begin with a fourth-order model to describe a relaxation process of a crystal surface by surface diffusion under the roughening temperature, which is proposed by \cite{Sp} as mentioned in Section \ref{SM}.
 It is explicitly written as
\[
	w_t = -\Delta \big( \operatorname{div}\left(\nabla w/|\nabla w| \right) + \beta \operatorname{div} \left(|\nabla w| \nabla w \right) \big)
\]
with $\beta>0$, where $w(x,t)$ represents the height of a crystal at $x$ and at time $t$.
 Fortunately, this can be handled by the theory of maximal monotone operators \cite{GG10}, \cite{GK}.
 Let $H^1_{\mathrm{av}}(\mathbb{T}^n)$ denote the space of average-free $H^1$ functions equipped with the inner product
\[
	(f, g)_1 := \sum^n_{i=1} \int_{\mathbb{T}^n} \partial_{x_i} f \partial_{x_i} g \; dx.
\]
In other words,
\[
	H^1_{\mathrm{av}}(\mathbb{T}^n) = \left\{ f \in L^2(\mathbb{T}^n) \Bigm| 
	\|f\|_{H^1} = (f,f)^{1/2}_{H^1} < \infty,\ 
	\int_{\mathbb{T}^n} f dx = 0 \right\}.
\]
It is of course a Hilbert space.
 This space is densely embedded in
\[
	L^2_{\mathrm{av}}(\mathbb{T}^n) = \left\{ f \in L^2(\mathbb{T}^n) \Bigm| 
	\int_{\mathbb{T}^n} f dx = 0 \right\}.
\]
The dual space of $H^1_{\mathrm{av}}$ (under $L^2$ pairing) is denoted by $H^{-1}_{\mathrm{av}}$.
 The canonical isomorphism from $H^1_{\mathrm{av}}$ to $H^{-1}_{\mathrm{av}}$ is denoted by $-\Delta$ and it agrees with the usual minus Laplacian for distributions.
 The space $H^{-1}_{\mathrm{av}}(\mathbb{T}^n)$ is a Hilbert space equipped with the inner product
\[
	(f,g)_{-1} := \left\langle(-\Delta)^{-1} f, g \right\rangle,
\]
where $\langle\ ,\ \rangle$ denotes a canonical pairing of $H^1_{\mathrm{av}}$ and $H^{-1}_{\mathrm{av}}$.
% 原稿9.2　2/6
 This $H^{-1}_{\mathrm{av}}(\mathbb{T}^n)$ is our basic Hilbert space.
 We set energy
\[
	\mathcal{E}_{\beta,p}(w) := \int_{\mathbb{T}^n} |\nabla w| + \frac{\beta}{p} \int_{\mathbb{T}^n} |\nabla w|^p dx
\]
with $p>1$, $\beta\geq 0$.
 We consider the gradient flow of $\mathcal{E}_{\beta,p}$ in $H^{-1}_{\mathrm{av}}(\mathbb{T}^n)$, i.e.,
\begin{equation} \label{4GF}
	w_t \in -\partial\mathcal{E}_{\beta,p}(w).
\end{equation}
Formally, this is an equation
\[
	w_t = -\Delta \Big( \operatorname{div}\left( \nabla w/|\nabla w|\right) + \beta \operatorname{div}\left( |\nabla w|^{p-2} \nabla w \right)\Big). 
\]
If $\beta=0$, this is nothing but the fourth-order total variation flow.
 A general theory guarantees the global-in-time existence of a solution to \eqref{4GF} with $\beta\geq 0$, $p>1$ for any initial data $w_0 \in H^{-1}_{\mathrm{av}}(\mathbb{T}^n)$ since $\mathcal{E}_{\beta,p}$ is a lower semicontinuous convex functional on $H^{-1}_{\mathrm{av}}(\mathbb{T}^n)$.
 The important difference between second-order and fourth-order is that in the latter the comparison principle fails.
 Here is an example for the case $\beta=0$, which implies that the comparison principle should not hold.
\begin{theorem}[\cite{GG10}] \label{9DIS}
For the fourth-order total variation flow \eqref{4GF} ($\beta=0$), the solution may become discontinuous in space even if the initial data is Lipschitz continuous.
\end{theorem}
In \cite{GG10}, this is proved by giving an explicit example for $n=1$, which works for general $n$.
 For the second-order problem, the comparison principle yields Lipschitz preserving property.
 Indeed, if the initial data $w_0$ is $L$-Lipschitz, then
\[
	w_0(x) \leq w_0(x+h) + Lh =: w_{0h}.
\]
The solution starting with $w_{0h}$ is $w(x+h,t)+Lh$.
 If the comparison principle were valid, we would have
\[
	w(x,t) \leq w(x+h,t) + Lh.
\]
Similarly,
\[
	w(x,t) \geq w(x+h, t) - Lh,
\]
so we would have $\left|w(x,t) - w(x+h,t)\right|\leq Lh$.
 Theorem \ref{9DIS} shows that the comparison principle fails for \eqref{4GF} with $\beta=0$. 

% 原稿9.2　3/6
Note that for $\beta>0$, $w(\cdot,t)$ is spatially continuous for $n=1$ since $\mathcal{E}_{\beta,p}(w)<\infty$ implies continuity.

There is a characterization of the subdifferential $\partial \mathcal{E}_{\beta,p}$ in $H^{-1}_{\mathrm{av}}(\mathbb{T}^n)$ or  similar space see \cite{Ka1}, \cite{Ka2} for $\beta>0$ and \cite{GK} for $\beta=0$.
 The minimal section is also calculated in \cite{Ka1} and \cite{GG10} in the case $n=1$;
 for radial case with $\beta>0$, see \cite{Ka2}.
 There are a few differences between second-order and fourth-order problem.
 First, the value of $\partial^\circ \mathcal{E}_{\beta,p}$ on a facet is not determined in a neighborhood of a facet in fourth-order problem.
 This is in some sense expected because of a ``nonlocal property'' of a norm on $H^{-1}_{\mathrm{av}}$.
 Second, the value of $\partial^\circ \mathcal{E}_{\beta,p}$ may contain $\delta$-type function ($n=1$), which yields instant discontinuity of a solution in Theorem \ref{9DIS}.

Of course, there are several common properties between second-order and fourth-order problems.
 For example, the solution will stop to move in finite time.
 In fourth-order problems, it is only known for $n=1,2,3,4$.
 Let $T_*(w_0)$ be the extinction time of the solution of \eqref{4GF}, i.e.,
\[
	T_*(w_0) = \sup \left\{ t \in \mathbb{R} \mid w(x,t) \not\equiv 0 \right\}.
\] 
\begin{theorem}[\cite{GK}] \label{9FEX}
Let $w$ be the solution of \eqref{4GF} with initial data $w_0 \in H^{-1}_{\mathrm{av}}$.
 There exists a constant $C$ depending only on $\omega_i$ and $n$ ($\mathbb{T}^n=\Pi^n_{i=1}(\mathbb{R}/\omega_i \mathbb{Z})$) (independent of dilation) such that
\begin{align*}
	&T_*(w_0) \leq C\|w_0\|_{H^{-1}_{\mathrm{av}}} \quad\text{for}\quad n=4 \\
	&T_*(w_0) \leq \frac{\|w_0\|_X}{a} \left( \left(1+\frac{a \|w_0\|^\alpha_{H^{-1}_{\mathrm{av}}}}{C\|w_0\|^\alpha_X} \right)^{1/\alpha} -1 \right)
	\quad\text{for}\quad 1\leq n \leq 4,\ 1\leq p\leq\infty 
\end{align*}
with $\theta\in\left(\frac{1}{2},1\right]$ satisfying $1+\frac{n}{2}=\theta(n-1)+(1-\theta)(3+n/p)$, where $a=(\omega_1\cdots\omega_N)^{1/p}$, $\alpha=2-1/p$ and $\|w_0\|_X=\left\|(-\Delta)^{-1} w_0\right\|_{\dot{W}^{-1,p}}$.
\end{theorem}
%
% 原稿9.2　4/6
Here, $\dot{W}^{-1,p}$ is the dual of the homogeneous Sobolev space $\dot{W}^{1,p}$, i.e.,
\[
	\| f \|_{\dot{W}^{-1,p}} = \sup \left\{ \int_{\mathbb{T}^n} f\varphi \; dx \Bigm|
	\varphi \in C^\infty(\mathbb{T}^n),\ \| \nabla\varphi \|_{L^{p'}} \leq 1 \right\}, \quad
	1/p+1/p'=1.
\]

The proof for $n=4$ is easy, so we give it here for $\beta=0$;
 the case $\beta>0$ can be proved essentially in the same way.
 We multiply the equation
\[
	w_t = (-\Delta) \operatorname{div} \left( \nabla w/|\nabla w| \right) 
\]
with $(-\Delta)^{-1} w$ and integrate in space to get a dissipation identity
\begin{equation} \label{DISE}
	\frac{1}{2} \frac{d}{dt} \| w \|^2_{H^{-1}_{\mathrm{av}}}
	= \int_{\mathbb{T}^n} | \nabla w |
\end{equation}
since $(u,v)_{-1} = \left\langle (-\Delta)^{-1} u,v \right\rangle$.
 In the case $n=4$ and $\theta=1$, by the Sobolev and the Calder\'on-Zygmund inequality for $\nabla(-\Delta)^{-1/2}$, we have
\[
	\| w \|_{H^{-1}_{\mathrm{av}}} = \left\| (-\Delta)^{-1/2} w \right\|_{L^2}
	\leq A' \left\| (-\Delta)^{-1/2} w \right\|_{L^p}
	\leq A_p \| w \|_{L^p}, \quad
	 1/2=1/p-1/4 
\]
for some constants $A'$ and $A_p$.
 Again by the Sobolev inequality, there is a constant $S$ satisfying
\[
	\| w \|_{L^{4/3}} \leq S \int_{\mathbb{T}^n} | \nabla w |.
\]
We now conclude that
\[
	\| w \|_{H^{-1}_{\mathrm{av}}} \leq A_{4/3} S \int_{\mathbb{T}^n} | \nabla w |.
\]
Thus we conclude
\[
	\frac{1}{2} \frac{d}{dt} \| w \|^2_{H^{-1}_{\mathrm{av}}} 
	\leq -(A_{4/3} S)^{-1}  \| w \|_{H^{-1}_{\mathrm{av}}}, 
\]
which yields $T_*(w_0) \leq C\|w_0\|_{H^{-1}_{\mathrm{av}}}$ with $C=A_{4/3}S$.
 For general case, we establish an interpolation inequality
\[
	\| w \|_{H^{-1}_{\mathrm{av}}} \leq C \left\| (-\Delta)^{-1} w \right\|^{1-\theta}_{\dot{W}^{-1,p}}
	\left( \int_{\mathbb{T}^n} |\nabla w| \right)^\theta
\]
and a rough growth estimate for a weaker norm
\[
	\frac{d}{dt} \left\| (-\Delta)^{-1} w \right\|_{\dot{W}^{-1,p}}
	\leq a^{1/p}. 
\]
We then apply these inequalities to the dissipation identity \eqref{DISE} to get the desired estimate.
 For details, see \cite{GK}, \cite{GKM}.
 Combining a dissipation identity, an interpolation inequality and a growth of a weaker norm is also a key idea to estimate the coarsening rate in a surface diffusion flow as studied in \cite{KO}. % 文法？

% 原稿9.2　5/6
There are several numerical studies for the above fourth-order singular diffusion equations.
 A numerical computation for $\beta>0$, $p=3$ is done by \cite{KV}.
 Their numerical scheme regularizes the singularity.
 A duality based numerical scheme which applies the forward-backward splitting has been proposed in \cite{GMR}.
 A Bregman method is adjusted to the fourth-order problem by \cite{GU}, where the singularity at $\nabla w=0$ is not regularized.

% 原稿9.2　6/6
We are interested in a polygonal flow by surface diffusion.
 Formally, a typical example is $V=-\Delta\kappa_\sigma$ when $\sigma$ is crystalline.
 In \cite{CRCT} evolution by polygonal flow is proposed and there are several numerical tests.
 However, there is no general notion for a solution of closed curves.
 It is not clear what class of polygonal flows is preserved during evolution.
 Recently, in \cite{GG21} it is shown that there is a special class of periodic piecewise linear graph-like curves which is preserved under the evolution provided that the problem is written as a gradient flow of a lower semicontinuous convex function.

% 修正 2/2
If the dependence on $\kappa_\sigma$ is nonlinear like in \eqref{4EXP}, no notion of a general solution is known. By studying  a special solution of \eqref{4EXP}, a new phenomenon is found in \cite{LLMM} with discussion on a relation with a step motion. There is numerical work to calculate \eqref{4EXP} in \cite{CLLMW}.

\end{document}